\documentclass[leqno,a4paper]{amsart}
\usepackage{amssymb}
\usepackage[arc,arrow,curve,matrix]{xy}

\numberwithin{equation}{section}

\newtheorem{theorem}{Theorem}[section]
\newtheorem{corollary}[theorem]{Corollary}
\newtheorem{lemma}[theorem]{Lemma}
\newtheorem{proposition}[theorem]{Proposition}

\theoremstyle{definition}
\newtheorem{definition}[theorem]{Definition}
\newtheorem{remark}[theorem]{Remark}
\newtheorem{example}[theorem]{Example}
\newtheorem{problem}[theorem]{Problem}
\newtheorem{question}[theorem]{Question}

\newcommand{\add}{\operatorname{add}}
\newcommand{\alg}{\operatorname{alg}}
\newcommand{\ch}{\operatorname{char}}
\newcommand{\End}{\operatorname{End}}
\newcommand{\ev}{\operatorname{ev}}
\newcommand{\genus}{\operatorname{genus}}
\newcommand{\GL}{\operatorname{GL}}
\newcommand{\Hom}{\operatorname{Hom}}
\newcommand{\hh}{\operatorname{H}}
\newcommand{\HC}{\operatorname{HC}}
\newcommand{\HH}{\operatorname{HH}}
\newcommand{\lcm}{\operatorname{lcm}}
\newcommand{\modf}{\operatorname{mod}}
\newcommand{\per}{\operatorname{per}}
\newcommand{\PSL}{\operatorname{PSL}}
\newcommand{\rank}{\operatorname{rank}}
\newcommand{\soc}{\operatorname{soc}}
\DeclareMathOperator{\stHom}{\underline{\Hom}}
\DeclareMathOperator{\stmod}{\underline{\modf}}
\DeclareMathOperator{\RHom}{\mathbf{R}Hom}

\newcommand{\balpha}{\bar{\alpha}}
\newcommand{\oa}{\omega_\alpha}
\newcommand{\oba}{\omega_{\balpha}}
\newcommand{\eps}{\varepsilon}
\newcommand{\vphi}{\varphi}
\newcommand{\wh}{\widehat}

\newcommand{\cA}{\mathcal{A}}
\newcommand{\cB}{\mathcal{B}}
\newcommand{\cC}{\mathcal{C}}
\newcommand{\cD}{\mathcal{D}}
\newcommand{\cJ}{\mathcal{J}}
\newcommand{\cK}{\mathcal{K}}
\newcommand{\gL}{\Lambda}
\newcommand{\dgL}{\cD^b(\gL)}
\newcommand{\cP}{\mathcal{P}}
\newcommand{\cQ}{\mathcal{Q}}
\newcommand{\cT}{\mathcal{T}}
\newcommand{\bZ}{\mathbb{Z}}

\begin{document}

\title[]{From groups to clusters}

\author[]{Sefi Ladkani}
\address{Department of Mathematics \\
University of Haifa \\
Mount Carmel, Haifa 31905, Israel}
\email{ladkani.math@gmail.com}

\thanks{The author acknowledges support
from DFG grant LA 2732/1.1 within the framework of the priority
program SPP 1388 Representation Theory.}

\begin{abstract}
We construct a new class of symmetric algebras of
tame representation type that are also the endomorphism algebras of
cluster-tilting objects in 2-Calabi-Yau triangulated categories, hence
all their non-projective indecomposable modules are $\Omega$-periodic of
period dividing 4. Our construction is based on the combinatorial notion of
triangulation quivers, which arise naturally from triangulations of oriented
surfaces with marked points.

This class of algebras contains the algebras of
quaternion type introduced and studied by Erdmann with relation to certain
blocks of group algebras. On the other hand, it contains also the Jacobian
algebras of the quivers with potentials associated by Fomin-Shapiro-Thurston
and Labardini-Fragoso to triangulations of closed surfaces with punctures,
hence our
construction may serve as a bridge between the modular representation theory
of finite groups and the theory of cluster algebras.
\end{abstract}



\maketitle

\setcounter{tocdepth}{1}
\tableofcontents

\section*{Introduction}

The aim of this survey is to report on new connections between the
representation theory of finite groups and the theory of cluster algebras.

Blocks of group algebras form an important class of indecomposable, symmetric
finite-dimensional algebras. Blocks of finite representation type are Morita
equivalent to Brauer tree algebras and are well understood.
In order to understand blocks of tame representation type,
Erdmann~\cite{Erdmann90} introduced
the classes of algebras of dihedral, semi-dihedral and quaternion type,
which are defined by properties of their Auslander-Reiten quiver,
proved that blocks with dihedral, semi-dihedral or generalized quaternion
defect group belong to the respective class of algebras and classified
the possible quivers with relations these algebras may have.

One of these classes consists of the algebras of quaternion type, which are
indecomposable, symmetric algebras of tame representation type having
non-singular Cartan matrix, with the property that
any indecomposable non-projective module is $\Omega$-periodic of period
dividing $4$, where $\Omega$ is Heller's syzygy functor. This class
of algebras is closed under derived equivalences~\cite{Holm99}.

2-Calabi-Yau triangulated categories with cluster-tilting objects arise in the
additive categorification of cluster algebras with skew-symmetric exchange
matrices~\cite{Amiot09,BMRRT06,Keller10,Keller11}.
The role of the clusters in the cluster algebra is played by 
the cluster-tilting objects, whose endomorphism algebras, called
\emph{2-CY-tilted algebras}, have remarkable representation theoretic and
homological properties~\cite{BMR07,KellerReiten07}.

We show that symmetric algebras $\gL$ that are in addition 2-CY-tilted have
interesting structural properties analogous to those of the algebras of
quaternion type; firstly, the functor $\Omega^4$ is
isomorphic to the identity functor on the stable module category
$\stmod \gL$ (Proposition~\ref{p:period});
secondly, such algebras tend to come in derived equivalence classes
(Proposition~\ref{p:dereq}). More precisely, if $\gL=\End_{\cC}(T)$ 
for a cluster-tilting object $T$ in a 2-Calabi-Yau triangulated category $\cC$,
then the 2-CY-tilted algebra $\gL'=\End_{\cC}(T')$ is derived equivalent
to $\gL$ for any other cluster-tilting object $T'$ obtained from $T$
by a finite sequence of Iyama-Yoshino~\cite{IyamaYoshino08} mutations.

Motivated by this analogy one is naturally led to ask whether the algebras
of quaternion type can be realized as 2-CY-tilted algebras, and even more
generally, what are the symmetric algebras that are also 2-CY-tilted?

In this survey we provide an affirmative answer to the first question
and attempt to answer the second question, first by 
classifying the symmetric, 2-CY-tilted algebras of finite representation type
and then by constructing a new class of
symmetric, 2-CY-tilted algebras of tame representation type.
Note that there are also many wild symmetric, 2-CY-tilted algebras, but we
will not discuss them here.
Let us describe the main results along with the structure of this survey.

In Section~\ref{sec:motivation} we review some basic notions including
blocks, stable categories, symmetric algebras, periodic modules and the
definition of algebras of quaternion type.
We also introduce the algebras of quasi-quaternion type, which are defined
similarly to the algebras of quaternion type,
the only difference being the omission of the condition
that the Cartan matrix is non-singular.

In Section~\ref{sec:sym2CY} we investigate symmetric 2-CY-tilted algebras.
We start by recalling the definition and basic properties of 2-CY-tilted
algebras. Since many of them arise as Jacobian algebras of quivers
with potentials~\cite{Amiot09,DWZ08,Keller11},
we review this notion as well, and introduce the
notion of hyperpotential~\cite{Ladkani14a} which is useful over ground fields
of positive characteristic.
Then we present results concerning the periodicity of
modules and derived equivalences for these algebras.

A classification of symmetric, 2-CY-tilted, indecomposable algebras of finite
representation type which are not simple is presented in
Section~\ref{sec:sym2CYfin}.
We show that these algebras are precisely the Brauer tree algebras
with at most two simple modules (Theorem~\ref{t:sym2CYfin}).

Then, we construct a large class of symmetric, 2-CY-tilted algebras of tame
representation type (Theorem~\ref{t:quasi}).
Our construction is based on the combinatorial notion of \emph{triangulation
quivers}, which are quivers with the property that for each vertex
the set of incoming arrows and that of outgoing arrows have cardinality
$2$, together with bijections between these sets that combine to yield
a permutation on the set of all arrows which is of order dividing $3$.
Triangulation quivers can be built from ideal triangulations of surfaces
with marked points in a way which is analogous to, but different than the
construction of the adjacency quivers of Fomin, Shapiro and
Thurston~\cite{FST08} arising in their work on cluster algebras from
surfaces.

The ingredients behind our construction are presented in
Sections~\ref{sec:quivers}, \ref{sec:surface}, \ref{sec:algebras}
and~\ref{sec:triangquasi}.
Section~\ref{sec:quivers} forms the combinatorial heart of this survey.
We introduce ribbon quivers and the dual notion of ribbon graphs,
define the subclass of triangulation quivers, and present a block decomposition
of the latter into three basic building blocks.
Section~\ref{sec:surface} explains how triangulations of marked surfaces
give rise to triangulation quivers. We discuss the
differences and similarities to adjacency quivers and provide a dimer model
perspective on these constructions.

In Section~\ref{sec:algebras} we introduce two classes of algebras which turn
out to be important for our study, one consists of the well known
Brauer graph algebras~\cite{Alperin86,Benson98,Kauer98}, while the other
is the newly defined \emph{triangulation algebras}. Roughly speaking,
a Brauer graph algebra arises from any ribbon quiver and
auxiliary data given in the form of scalars and positive integer
multiplicities, whereas
a triangulation algebra arises from any triangulation quiver with similar
auxiliary data.

In Section~\ref{sec:triangquasi} we investigate triangulation algebras in
more detail and prove that they are finite-dimensional, tame, symmetric,
2-CY-tilted algebras and hence of quasi-quaternion type.
By using Iyama-Yoshino mutations of
cluster-tilting objects~\cite{IyamaYoshino08} we are able to construct even
more, derived equivalent, algebras with the same properties. The
finite-dimensionality of the triangulation algebras relies on computations
inside complete path algebras of quivers, whereas the proof of their
representation type uses the observation that apart from a few exceptions,
the triangulation algebras are deformations of the corresponding
Brauer graph algebras~(Proposition~\ref{p:degen}).

Our construction yields new symmetric 2-CY-tilted algebras in addition to the
ones constructed by Burban, Iyama, Keller and Reiten~\cite{BIKR08} arising from
the stable categories of maximal Cohen-Macaulay modules over odd-dimensional
isolated hypersurface singularities.
Moreover, it provides new insights on the important problem of classifying
the self-injective algebras with periodic module categories, as
the algebras we construct are instances of
new tame symmetric algebras with periodic modules which seem not to appear in
the classification announced by Erdmann and
Skowro\'{n}ski~\cite[Theorem~6.2]{ES08}.

In Section~\ref{sec:known} we prove that 
our newly constructed class of algebras
contains two known classes of algebras as subclasses.
Firstly, it contains all the members in Erdmann's lists of
algebras of quaternion type (Theorem~\ref{t:quat2CY}).
Since these algebras turn out to be 2-CY-tilted, this gives a new proof of
the fact that they are indeed of quaternion type, which was first shown
in~\cite{ES06} by constructing bimodule resolutions.
As a consequence, we are able to characterize all the blocks of group algebras
that are 2-CY-tilted algebras (Proposition~\ref{p:block2CY}).

In order to illustrate the advantage of
this new point of view on the algebras of quaternion
type, we discover new algebras of quaternion type which seem not to appear
in the existing lists (Proposition~\ref{p:newquat}).

Secondly,
our newly constructed class of algebras contains also all the Jacobian algebras
of the quivers with potentials associated by
Labardini-Fragoso~\cite{Labardini09} to triangulations of closed surfaces with
punctures. As a consequence, we deduce that the latter algebras are
finite-dimensional of quasi-quaternion type and their derived equivalence
class depends only on the surface and not on the particular triangulation
(Corollary~\ref{c:Jaclosed}, see also~\cite{Ladkani12}).

Our newly constructed class contains also all the symmetric algebras of
tubular type $(2,2,2,2)$ and their socle deformations classified in~\cite{BS03,BS04}
(Proposition~\ref{p:sympoly}). 

In Section~\ref{sec:mut} we introduce a notion of mutation on triangulation
quivers and compare it to various other notions of mutations existing in the
literature, including flips of triangulations, Kauer's elementary
moves~\cite{Kauer98} for Brauer graph algebras and mutations of quivers with
potentials~\cite{DWZ08}. 
We observe that the Brauer graph algebras arising from different
triangulations of the
same marked surface are derived equivalent (Corollary~\ref{c:BGAsurface}),
a result which has also been obtained by Marsh and
Schroll~\cite{MarshSchroll14}, however the algebras they consider in the case
of surfaces with non-empty boundary are different.
Analogously, under mild conditions the triangulation algebras of
triangulation quivers related by a mutation are derived equivalent
(Proposition~\ref{p:mutriang}).

Finally we outline an application to the theory of quivers with potentials.
Non-degenerate potentials are important in various approaches to the
categorification of cluster algebras~\cite{DWZ10,Plamondon11}.
It was proved by Derksen, Weyman and Zelevinsky~\cite{DWZ08} that over an
uncountable field, any quiver without loops or 2-cycles has at least one
non-degenerate potential. For certain classes of quivers, a non-degenerate
potential is unique~\cite{GLS16,Ladkani13}.
On the other hand, we construct infinitely many families of quivers, each
having infinitely many non-degenerate potentials with pairwise non-isomorphic
Jacobian algebras (Corollary~\ref{c:infpot}).

\subsection*{Acknowledgements}
I discussed various aspects of this work with Thorsten Holm, Maxim Kontsevich,
Robert Marsh and Andrzej Skowro\'{n}ski. I thank them for their interest.
I would like to thank Peter Littelmann for his encouragement while writing
this survey.
I thank also the anonymous referee for many valuable comments and suggestions.

Parts of the material were presented in various talks that I gave in Germany
during Summer 2013 and in Israel
during Winter 2014/15, as well as at the workshop on Cluster algebras and
finite-dimensional algebras that was held in June 2015 at Leicester, UK.
They were also scheduled to be presented at the ARTA conference that was held
in September 2013 at Torun, Poland, and at the final meeting of the
priority program ``Representation Theory''
that took place at Bad Honnef, Germany in March 2015.
I thank the organizers of these meetings for their invitations.

Many of the results presented in this survey were obtained during my stay at
the University of Bonn which was supported by my DFG grant LA~2732/1-1 within
the framework of the priority program SPP 1388 ``Representation Theory''.

A report containing these results~\cite{Ladkani14b} was
written during my visit to the IHES at Bures-sur-Yvette in the spring of 2014.
In a subsequent visit during Spring 2015
some aspects of the theory were refined.
I would like to thank the IHES for the hospitality and the
inspiring atmosphere.

While writing this survey I have been supported by the
Center for Absorption in Science, Ministry of Aliyah and Immigrant Absorption,
State of Israel.

\section{Motivation: blocks of tame representation type}
\label{sec:motivation}

\subsection{Group algebras}
\label{ssec:group}

Let $G$ be a finite group and $K$ be a field. The group algebra
$KG$ can be written as a direct product of indecomposable rings,
which are called \emph{blocks}.
By Maschke's theorem, if the characteristic of $K$ does not divide 
the order of $G$,
then $KG$, and hence each block, is semi-simple. In particular, when $K$
is also algebraically closed, each block is isomorphic to a matrix ring
over~$K$.

However, when the characteristic of $K$, denoted here and throughout the
paper by $\ch K$, divides the order of $G$, a block may not be
semi-simple anymore. The \emph{defect group} of a block $B$
measures how far it is from being semi-simple. It may be defined
as a minimal subgroup $D$ of $G$ such that any $B$-module
is $D$-projective (i.e.\ it is isomorphic to a direct summand of
$W \otimes_{KD} KG$ for some $KD$-module $W$).
A defect group is a $p$-subgroup of $G$ (where $p=\ch K$),
determined up to $G$-conjugacy. A defect group of the principal block
(the block which the trivial $KG$-module $K$ belongs to) is
a $p$-Sylow subgroup of $G$, and
a block is semi-simple if and only if its defect group is trivial.
We refer to the survey article~\cite{Linckelmann11} for further details.

Many aspects of the representation theory of a block are controlled by its
defect group. One such important aspect is the representation type.
Indeed, if $B$ is a block with defect group $D$ over an algebraically closed
field of
characteristic $p$, then $B$ is of finite representation type if and only if
$D$ is cyclic~\cite{Higman54},
while $B$ is of tame (but not finite) representation type if and only if
$p=2$ and $D$ is either dihedral, semi-dihedral of generalized quaternion
group~\cite{BD75}.
In all other cases, $B$ is of wild representation type.

Blocks of finite representation type, that it, blocks with cyclic defect group,
are Morita equivalent to Brauer tree algebras~\cite{Dade66,Janusz69} and hence
are well understood.
In order to understand blocks of tame representation type (over algebraically
closed fields), Erdmann introduced families of symmetric algebras defined by
properties of their Auslander-Reiten quivers. These are the algebras
of dihedral, semi-dihedral and quaternion type.
She showed that a block with dihedral (respectively, semi-dihedral, generalized
quaternion) defect group is an algebra of the corresponding type and moreover
she classified the quivers with relations these algebras may possibly
have~\cite{Erdmann90}.

In this section we focus on the algebras of quaternion type and start by
reviewing the relevant notions.

\subsection{Stable categories and periodicity}
Let $A$ be a finite-dimensional algebra over a field $K$. Denote by 
$\modf A$ the category of finitely generated right $A$-modules,
and by $\cD^b(A) = \cD^b(\modf A)$ its bounded derived category.
The latter contains as triangulated subcategory the category
$\per A$ of perfect complexes whose objects are bounded complexes
of finitely generated projective $A$-modules.
The Verdier quotient $\cD^b(A)/\per A$ is known as the
\emph{singularity category} of $A$, see~\cite{Orlov09}. 
Its name comes from the fact that it vanishes precisely when $A$ has finite
global dimension (i.e.\ $A$ is ``smooth'')~\cite[Remark~1.9]{Orlov09},
as in this case any $A$-module has a finite projective resolution,
thus any object in $\cD^b(A)$ is isomorphic to a perfect complex.

Assume that the algebra $A$ is \emph{self-injective}, i.e.\ $A$ is injective
as left and right module over itself, and
consider the \emph{stable module category} $\stmod A$
whose objects are the same as those of $\modf A$ and the space of morphisms
between any two objects $M, N \in \stmod A$ is given by
\[
\underline{\Hom}_A(M,N) = \Hom_A(M,N)/\cP(M,N)
\]
where $\cP(M,N)$ consists of all the morphisms $M \to N$ in $\modf A$
which factor through some projective module over $A$.

By a result of Happel~\cite[Theorem~I.2.6]{Happel88}, the additive category
$\stmod A$ is triangulated. Moreover, by a theorem of
Rickard~\cite[Theorem~2.1]{Rickard89}, $\stmod A$ can be identified with the
singularity category of $A$.

Let $M \in \stmod A$ and consider a projective cover $P_M$ of $M$. Define
a module $\Omega M$ by the exact sequence (in $\modf A$)
\[
0 \to \Omega M \to P_M \to M \to 0 .
\]
The \emph{syzygy} $\Omega M$ is well defined in the category $\stmod A$ and
gives rise to \emph{Heller's syzygy functor}
$\Omega \colon \stmod A \to \stmod A$, see~\cite{Heller61}.
Sometimes, when we want to stress the role of the
algebra $A$, we will denote the syzygy functor by $\Omega_A$ instead of $\Omega$.

Similarly, by taking an injective envelope $I_M$ of $M$ and the exact sequence
\[
0 \to M \to I_M \to \Omega^{-1} M  \to 0
\]
one can define the cosyzygy functor $\Omega^{-1} \colon \stmod A \to \stmod A$,
which is an inverse of $\Omega$.
The suspension functor of the triangulated category $\stmod A$ is
given by $\Omega^{-1}$.

An important class of self-injective algebras is formed by the
symmetric algebras, which we now define. First, observe that
for a finite-dimensional algebra $A$ over $K$, the vector space
$DA = \Hom_K(A,K)$ is an $A$-$A$-bimodule.

\begin{definition}
A finite-dimensional $K$-algebra $A$ is \emph{symmetric} if
$A \simeq DA$ as $A$-$A$-bimodules. Here and throughout the paper,
the symbol $\simeq$ denotes isomorphism.
\end{definition}

We recall a few alternative characterizations of symmetric algebras.
In order to formulate them, we need the notion of a Calabi-Yau
triangulated category which is given below.

\begin{definition}
Let $d \in \bZ$.
A $K$-linear triangulated category $\cT$ with suspension $\Sigma$
and finite-dimensional morphism spaces is \emph{$d$-Calabi-Yau} if
there exist functorial isomorphisms
\[
\Hom_{\cT}(X, Y) \simeq D \Hom_{\cT}(Y, \Sigma^d X)
\]
for all $X, Y \in \cT$.
\end{definition}

\begin{proposition} \label{p:symmetric}
The following conditions are equivalent for a finite-dimensional $K$-algebra
$A$.
\begin{enumerate}
\renewcommand{\theenumi}{\alph{enumi}}
\item \label{it:p:sym}
$A$ is symmetric;

\item \label{it:p:symform}
There exists a symmetrizing form on $A$, that is,
a $K$-linear map $\lambda \colon A \to K$
whose kernel does not contain any non-trivial left ideal of $A$
and moreover $\lambda(xy)=\lambda(yx)$ for any $x,y \in A$;

\item \label{it:p:symper}
The triangulated category $\per A$ is $0$-Calabi-Yau.

\item \label{it:p:sym0CY}
$A$ is isomorphic to the endomorphism algebra of an object in a triangulated
$0$-Calabi-Yau category.
\end{enumerate}
\end{proposition}
\begin{proof}
The equivalence of~\eqref{it:p:sym} and~\eqref{it:p:symform} is standard, see
e.g.~\cite[Theorem~IV.2.2]{SY11}.
The implication~\eqref{it:p:sym}$\Rightarrow$\eqref{it:p:symper} follows from
the fact that for any finite-dimensional algebra $A$ one has
\[
\Hom_{\cD^b(A)}(X, Y) \simeq D\Hom_{\cD^b(A)}(Y,
X \stackrel{\mathbf{L}}{\otimes}_A DA)
\]
for any $X \in \per A$ and $Y \in \cD^b(A)$, see the proof
of~\cite[Theorem~I.4.6]{Happel88}.
For the implication~\eqref{it:p:symper}$\Rightarrow$\eqref{it:p:sym0CY},
observe that $A \simeq \End_{\per A}(A)$.
For~\eqref{it:p:sym0CY}$\Rightarrow$\eqref{it:p:sym}, note that if $X$ is an
object in a $0$-Calabi-Yau triangulated category $\cT$ and $A = \End_{\cT}(X)$,
then the functorial isomorphism $\Hom_{\cT}(X,X) \simeq D\Hom_{\cT}(X,X)$
implies that $A \simeq DA$ as $A$-$A$-bimodules.
\end{proof}

\begin{corollary}[\protect{\cite[Theorem~IV.4.1]{SY11}}]
\label{c:sym:idempotent}
Let $A$ be a symmetric algebra and $e \in A$ an idempotent. Then the algebra
$eAe$ is also symmetric.
\end{corollary}
\begin{proof}
One has $D(eAe) \simeq e(DA)e$; alternatively,
use Proposition~\ref{p:symmetric} for the algebra
$eAe \simeq \End_{\per A}(eA)$.
\end{proof}

\begin{example}
Any group algebra $KG$ is symmetric. Indeed, a symmetrizing form on $KG$ is
given by
\[
\lambda(\sum_{g \in G} a_g g) = a_1 .
\]
It follows from Corollary~\ref{c:sym:idempotent} that any block of a group
algebra is also symmetric.
\end{example}

\begin{example}
If $A$ is any finite-dimensional $K$-algebra, the bimodule structure
on $DA$ allows to define a symmetric algebra whose underlying vector space is
$A \oplus DA$ called the \emph{trivial extension algebra} of $A$ and
denoted by $T(A)$. The elements of $T(A)$ are
pairs $(a,\mu)$ where $a \in A$ and $\mu \in DA$. Addition and
multiplication are given by the formulae
\begin{align*}
(a, \mu) + (a', \mu') = (a+a', \mu + \mu') \\
(a,\mu) \cdot (a',\mu') = (aa', a \mu'+\mu a')
\end{align*}
for $a,a' \in A$ and $\mu, \mu' \in DA$.
The symmetrizing form on $T(A)$ is given by $\lambda(a,\mu) = \mu(1)$.
\end{example}

\begin{remark} \label{rem:CYminus1}
The stable category of a symmetric algebra $A$ is
$(-1)$-Calabi-Yau, i.e.\
\[
\stHom_A (M, N) \simeq D \stHom_A (N, \Omega M)
\]
for $M, N \in \stmod A$, see for example~\cite[Proposition~1.2]{ES06}
and the end of~\cite[\S1]{Amiot09}.
\end{remark}

\begin{definition}
A module $M \in \stmod A$ is \emph{$\Omega$-periodic} if
$\Omega^r M \simeq M$ for some integer $r>0$.
\end{definition}

The category $\stmod A$ has Auslander-Reiten sequences, and 
when $A$ is symmetric there is a
close connection between the Auslander-Reiten translation $\tau$ on
$\stmod A$ and the syzygy $\Omega$, namely $\tau = \Omega^2$.
In particular, a module is $\Omega$-periodic if and only if it
is $\tau$-periodic.

\begin{example} \label{ex:Kxn}
Let $n \geq 1$ and consider the algebra $A = K[x]/(x^n)$.
It is a commutative, local, symmetric algebra over $K$ of finite
representation type whose indecomposable modules are given by $M_i = x^i A$
for $0 \leq i < n$. The module $M_0 = A$ is projective, and the exact
sequence
\[
0 \to x^{n-i}A \to A \xrightarrow{x^i \cdot -} x^i A \to 0
\]
shows that $\Omega(M_i) = M_{n-i}$ for any $0 < i < n$.
Hence $\Omega^2 M \simeq M$ for any $M \in \stmod A$.
Note that if $\ch K = p$ and $n=p^e$ for some $e \geq 1$,
then $A \simeq KG$ for $G=\bZ/p^e\bZ$ and its defect group
equals $G$.
\end{example}

\subsection{Algebras of quaternion type}
In this section we assume that the ground field $K$ is algebraically
closed. The algebras of quaternion type were introduced by Erdmann, and
we refer to the articles~\cite{Erdmann88} and the monograph~\cite{Erdmann90}
for a detailed presentation.

\begin{definition}[\cite{Erdmann88}]
A finite-dimensional algebra $A$ is of \emph{quaternion type} if
\begin{enumerate}
\renewcommand{\theenumi}{\roman{enumi}}
\item
$A$ is symmetric, indecomposable as a ring;
\item
$A$ has tame (but not finite) representation type;
\item
$\Omega^4 M \simeq M$ for any $M \in \stmod A$;
\item \label{it:q:cartan}
$\det C_A \neq 0$, where $C_A$ denotes the Cartan matrix of $A$. 
\end{enumerate}
\end{definition}

Recall that the \emph{Cartan matrix} of a basic algebra $A$ is the
$n$-by-$n$ matrix with integer entries given by
$(C_A)_{i,j} = \dim_K e_i A e_j$, where $e_1, e_2, \dots, e_n$ form a complete
set of primitive orthogonal idempotents in $A$.
The motivation behind condition~\eqref{it:q:cartan} lies in the fact that
if $B$ is a block over a field of characteristic $p$, then the determinant
of its Cartan matrix is a power of $p$.

Let $n \geq 3$. The \emph{generalized quaternion group} $Q_{2^n}$ is
given by generators and relations as follows:
\[
Q_{2^n} = 
\langle
x,y \mid x^{2^{n-2}}=y^2 \,,\, y^4=1 \,,\, y^{-1}xy=x^{-1}
\rangle .
\]
In particular, for $n=3$ we recover the usual quaternion group with
$8$ elements.

Erdmann proved the following facts:
\begin{enumerate}
\renewcommand{\theenumi}{\alph{enumi}}
\item
Blocks of group algebras with generalized quaternion defect groups are of
quaternion type.

\item \label{it:q12lists}
An algebra of quaternion type is Morita equivalent to an algebra in 12
families of symmetric algebras given by quivers with relations.
In particular, an algebra of quaternion type has at most three isomorphism
classes of simple modules.
\end{enumerate}

The lists of the quivers with relations of point~\eqref{it:q12lists}
can be found in~\cite[pp.\ 303--306]{Erdmann90}
or in the survey articles~\cite[Theorem~5.5]{ES08}
and~\cite[Theorem~8.4]{Skowronski06}.
Later, Holm~\cite{Holm99}
presented a derived equivalence classification of the algebras appearing in
these lists and proved that these algebras are indeed tame. Finally,
in~\cite{ES06} Erdmann and Skowro\'{n}ski showed that the algebras in these lists
have the required periodicity property and hence they are indeed of quaternion
type.

\begin{example}
One of the families in Erdmann's list consists of local algebras whose
quiver is
\[
\xymatrix@=1pc{
{\bullet} \ar@(dl,ul)[]^{\alpha} \ar@(dr,ur)[]_{\beta}
}
\]
with the relations
\begin{align*}
\alpha^2 = (\beta \alpha)^{m-1} \beta &,&
\beta^2 = (\alpha \beta)^{m-1} \alpha &,&
\alpha \beta^2 = \alpha^2 \beta = \beta \alpha^2 = \beta^2 \alpha = 0
\end{align*}
depending on an integer parameter $m \geq 2$. 
When the ground field $K$ is algebraically closed, $\ch K = 2$ and $m=2^{n-2}$
for some $n \geq 3$, this algebra is the group algebra $K Q_{2^n}$ of the
generalized quaternion group $Q_{2^n}$.
\end{example}

\subsection{Algebras of quasi-quaternion type}

It seems natural to lift the restriction on the Cartan determinant in the
definition of algebras of quaternion type and consider a wider
class of algebras, which we call algebras of quasi-quaternion type.

\begin{definition} \label{def:quasi}
A finite-dimensional algebra $A$ is of \emph{quasi-quaternion type} if:
\begin{enumerate}
\renewcommand{\theenumi}{\roman{enumi}}
\item
$A$ is symmetric, indecomposable as a ring;
\item
$A$ has tame (but not finite) representation type;
\item \label{it:qq:period4}
$\Omega^4 M \simeq M$ for any $M \in \stmod A$;
\end{enumerate}
\end{definition}

\begin{remark}
Since $\tau=\Omega^2$, the stable Auslander-Reiten quiver of an algebra of
quasi-quaternion type consists of tubes of ranks $1$ and $2$.
\end{remark}

In analogy with Erdmann's description of the algebras of quaternion type,
the following problem arises naturally.
\begin{problem}
Describe the algebras of quasi-quaternion type.
\end{problem}

Algebras of quasi-quaternion type are in particular tame symmetric algebras
with periodic modules. A classification of the latter algebras has been
announced in~\cite[Theorem~6.2]{ES08},
see also~\cite[Theorem~8.7]{Skowronski06}. However, many of the algebras
of quasi-quaternion type to be constructed in Section~\ref{sec:triangquasi}
seem to be missing from the aforementioned classification.

Since the derived equivalence of self-injective algebras implies their stable 
equivalence~\cite[Corollary~2.2]{Rickard89}
and stable equivalence preserves representation type~\cite{Krause97},
an argument as in Prop.~2.1 and Prop.~2.2 of~\cite{Holm99} yields the
following observation.
\begin{proposition} \label{p:derquasi}
Any algebra which is derived equivalent to an algebra of quasi-quaternion
type is also of quasi-quaternion type.
\end{proposition}

One approach to guarantee the condition~\eqref{it:qq:period4} in the
definition of algebras of quasi-quaternion type
is to show that the algebra $A$ is periodic as $A$-$A$-bimodule with
period dividing $4$, that is, 
$\Omega_{A^e}^4(A) \simeq A$, where $A^e = A^{op} \otimes_K A$.
This is usually done using a projective resolution of $A$ as
a bimodule over itself. In fact, such strategy is used in~\cite{ES06}
to prove that the algebras in Erdmann's list are of quaternion type.

We suggest an alternative approach using 2-Calabi-Yau categories.
It turns out that symmetric algebras that are also the endomorphism
algebras of cluster-tilting objects in such categories always satisfy the 
periodicity condition~\eqref{it:qq:period4}. We explain this in the
next section.

\section{Symmetric 2-CY-tilted algebras}
\label{sec:sym2CY}

In this section we study properties of symmetric algebras that are also
2-CY-tilted, i.e.\ being isomorphic to the endomorphism algebras of
cluster-tilting objects in 2-Calabi-Yau triangulated categories. 

We start by recalling the definition and basic properties of 2-CY-tilted
algebras. Since many of them arise as Jacobian algebras of quivers with
potentials, we review this notion as well, and introduce the notion of
hyperpotential which is useful over ground fields of positive characteristic.
Then we present two new results whose details will appear elsewhere;
the first concerns the periodicity of modules over symmetric 2-CY-tilted
algebras (Proposition~\ref{p:period}), and the second concerns derived
equivalences of neighboring 2-CY-tilted algebras (Proposition~\ref{p:dereq}). 

As a consequence of the first result, we deduce that indecomposable, tame,
symmetric, 2-CY-tilted algebras are of quasi-quaternion type. For more
background on 2-CY-tilted algebras, we refer the reader to the survey
article~\cite{Reiten10}.

\subsection{2-CY-tilted algebras}

Let $\cC$ be a $K$-linear triangulated category with
suspension $\Sigma$. We assume:
\begin{itemize}
\item
$\cC$ has finite-dimensional morphism spaces.

\item
$\cC$ is Krull Schmidt (i.e.\ any object has a decomposition
into a finite direct sum of indecomposables which is unique up to
isomorphism and change of order).

\item
$\cC$ is $2$-Calabi-Yau.
\end{itemize}

Such triangulated categories $\cC$ arise in the additive categorification
of cluster algebras, see the survey~\cite{Keller10}.
The role of the clusters in a cluster algebra
is played by cluster-tilting objects in the category $\cC$.

\begin{definition}
An object $T \in \cC$ is \emph{cluster-tilting} if:
\begin{enumerate}
\renewcommand{\theenumi}{\roman{enumi}}
\item
$\Hom_{\cC}(T, \Sigma T) = 0$;

\item
For any $X \in \cC$ with $\Hom_{\cC}(T, \Sigma X) = 0$,
we have that $X \in \add T$, where 
$\add T$ denotes the full subcategory of $\cC$ consisting of
the objects isomorphic to finite direct sums of summands of $T$.
\end{enumerate}
\end{definition}

\begin{definition}
An algebra is called \emph{2-CY-tilted} if it is isomorphic to
an algebra of the form $\End_{\cC}(T)$ with $\cC$ as above and
$T$ a cluster-tilting object in $\cC$.
\end{definition}

The cluster categories associated to quivers without oriented cycles
\cite{BMRRT06} were the first instances of triangulated 2-Calabi-Yau
categories with cluster-tilting object.
They are constructed as orbit categories of the
bounded derived category of the path algebra of the quiver
with respect to a suitable auto-equivalence~\cite{Keller05}.
The corresponding endomorphism algebras of cluster-tilting objects
are called cluster-tilted algebras~\cite{BMR07}.
Self-injective cluster-tilted algebras were classified by
Ringel~\cite{Ringel08}; there are very few such algebras
as all of them are of
finite representation type and up to Morita equivalence there are at most two
such algebras having a given number of non-isomorphic simple modules.
In particular, except for the quiver $A_1$ with
one vertex whose cluster-tilted algebra equals the ground field,
cluster-tilted algebras are never symmetric.

More generally, 2-CY-tilted algebras were investigated by Keller and
Reiten~\cite{KellerReiten07}. The next proposition records the
relevant properties we need.

\begin{proposition} \label{p:2CYtilted}
Let $\gL$ be a 2-CY-tilted algebra. Then:
\begin{enumerate}
\renewcommand{\theenumi}{\alph{enumi}}
\item
\emph{\cite[Prop.~2.1]{KellerReiten07}}
$\gL$ is Gorenstein of dimension at most 1, i.e.\ the
projective dimension of any injective module and the injective dimension of
any projective module are at most 1;

\item
\emph{\cite[Theorem~3.3]{KellerReiten07}}
The singularity category of $\gL$ is $3$-Calabi-Yau.
\end{enumerate}
\end{proposition}

Given a 2-CY-tilted algebra $\gL$, there is a procedure to construct new
2-CY-tilted algebras from idempotents of $\gL$.
The corresponding statement for cluster-tilted algebras has been shown
in~\cite[Theorem~2.13]{BMR08}, see also~\cite[Theorem~5]{CalderoKeller06}.
The general case follows from Calabi-Yau reduction~\cite{IyamaYoshino08},
see also~\cite[\S{II.2}]{BIRS09}. For the convenience of the reader, we give
the short proof.
\begin{proposition}
Let $\gL$ be a 2-CY-tilted algebra and let $e \in \gL$ be an idempotent.
Then the algebra $\gL/\gL e \gL$ is 2-CY-tilted.
\end{proposition}
\begin{proof}
Let $\gL=\End_{\cC}(T)$ where $\cC$ is a triangulated 2-Calabi-Yau category
and $T$ is a cluster-tilting object in $\cC$. Let $T'$ be the summand of $T$ 
corresponding to the idempotent $e$. The category
$\cC' = \{X \in \cC : \Hom_{\cC}(X,\Sigma T')=0\}/(\add T')$
is a triangulated 2-Calabi-Yau category by~\cite[Theorem~4.7]{IyamaYoshino08}
and $T$ is a cluster-tilting object in $\cC'$
by~\cite[Theorem~4.9]{IyamaYoshino08}.
Finally, $\End_{\cC'}(T) \simeq \gL/\gL e \gL$.
\end{proof}

\subsection{Quivers with potentials}
\label{ssec:QP}

Thanks to the works of Amiot~\cite{Amiot09} and Keller~\cite{Keller11},
a rich source of 2-Calabi-Yau triangulated categories with cluster-tilting
object is provided 
by quivers with potentials whose Jacobian algebras are finite-dimensional.
Quivers with potentials and their Jacobian algebras were defined and
studied by Derksen, Weyman and Zelevinsky~\cite{DWZ08}.

A \emph{quiver} is a finite directed graph. Formally, it is a quadruple
$Q=(Q_0,Q_1,s,t)$ where $Q_0$ and $Q_1$ are finite sets (of \emph{vertices}
and \emph{arrows}, respectively) and $s,t \colon Q_1 \to Q_0$ are functions
specifying for each arrow its starting and terminating
vertex, respectively.

The \emph{path algebra} $KQ$ has the set of paths of $Q$ as a basis, with the
product of two paths being their concatenation, if defined, and zero otherwise.
The \emph{complete path algebra} $\wh{KQ}$ is the completion of
$KQ$ with respect to the ideal generated by all the arrows of $Q$.
It is a topological algebra, with a topological basis given by
the paths of $Q$. Thus, an element in $\wh{KQ}$ is a possibly
infinite linear combination of paths. We denote by $\bar{I}$ the closure of an
ideal $I$ in $\wh{KQ}$.

\begin{example}
The path algebra of the quiver with one vertex and one loop at that vertex
is the ring $K[x]$ of polynomials in one variable, 
whereas the complete path algebra is the ring $K[[x]]$ of power series 
in one variable.
\end{example}

A \emph{cycle} in $Q$ is a path that starts and ends at the same vertex.
One can consider the equivalence relation on the set of cycles given by
rotations, i.e.\
\[
\alpha_1 \alpha_2 \dots \alpha_n \sim
\alpha_i \dots \alpha_n \alpha_1 \dots \alpha_{i-1}
\]
for a cycle $\alpha_1 \alpha_2 \dots \alpha_n$ and $1 \leq i \leq n$.

The \emph{zeroth continuous Hochschild homology} $\HH_0(\wh{KQ})$ is $\wh{KQ}/\overline{[\wh{KQ},\wh{KQ}]}$, i.e.\ the
quotient of $\wh{KQ}$ by the closure of the subspace spanned by all the
commutators of elements in $\wh{KQ}$. It has a topological
basis given by the equivalence classes of cycles of $Q$ modulo rotation.

\begin{definition}[\protect{\cite[Definition~3.1]{DWZ08}}]
A \emph{potential} on $Q$ is an element in $\HH_0(\wh{KQ})$.
In explicit terms, a potential is a (possibly infinite) linear
combination of cycles in $Q$, considered up to rotations.

A pair $(Q,W)$ where $Q$ is a quiver and $W$ is a potential on $Q$
is called a \emph{quiver with potential}.
\end{definition}

For any arrow $\alpha$ of $Q$, there is a \emph{cyclic derivative} map
$\partial_\alpha \colon \HH_0(\wh{KQ}) \to \wh{KQ}$ which is
the unique continuous linear map whose value on each
cycle $\alpha_1 \alpha_2 \dots \alpha_n$ is given by
\[
\partial_\alpha(\alpha_1 \alpha_2 \dots \alpha_n) =
\sum_{i \,:\, \alpha_i = \alpha}
\alpha_{i+1} \dots \alpha_n \alpha_1 \dots \alpha_{i-1}
\]
where the sum goes over all indices $1 \leq i \leq n$ such that
$\alpha_i=\alpha$.

\begin{definition}[\cite{DWZ08}]
Let $(Q,W)$ be a quiver with potential. Its \emph{Jacobian algebra}
$\cP(Q,W)$ is the quotient of the complete path algebra $\wh{KQ}$
by the closure of its ideal generated by the
cyclic derivatives $\partial_\alpha W$ with respect to the arrows
$\alpha$ of $Q$,
\[
\cP(Q,W) = \wh{KQ} / \overline{(\partial_\alpha W : \alpha \in Q_1)} .
\]
\end{definition}

\begin{remark}
When the potential $W$ is a finite linear combination of cycles, one
can also consider a non-complete version of the Jacobian algebra, 
namely, the quotient of the path algebra $KQ$ by its ideal generated
by the cyclic derivatives of $W$. While in many cases this variation
gives the same result,
the next example shows that in general these two notions differ.
\end{remark}

\begin{example} \label{ex:QP}
Let $Q$ be the quiver
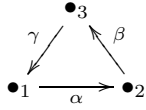
\[
\xymatrix@=0.5pc{
& {\bullet_3} \ar[ddl]_{\gamma} \\ \\
{\bullet_1} \ar[rr]_{\alpha} && {\bullet_2} \ar[uul]_{\beta}
}
\]
with the potential $W = \alpha \beta \gamma - \alpha \beta \gamma
\alpha \beta \gamma$. Let $\cJ$ be the closure of the ideal generated
by the cyclic derivatives of $W$, so that $\cP(Q,W) = \wh{KQ}/\cJ$.

Computing the cyclic derivative with respect to the arrow $\gamma$,
we get
\[
\partial_\gamma W = \alpha \beta - \alpha \beta \gamma \alpha \beta
- \alpha \beta \gamma \alpha \beta
= \alpha \beta - 2 \alpha \beta \gamma \alpha \beta ,
\]
hence $\alpha \beta - 2 \alpha \beta \gamma \alpha \beta \in \cJ$.
Therefore, for any $n \geq 1$,
\[
\alpha \beta - (2 \alpha \beta \gamma)^n \alpha \beta
= \sum_{i=0}^{n-1}
(2 \alpha \beta \gamma)^i 
(\alpha \beta - 2 \alpha \beta \gamma \alpha \beta) \in \cJ .
\]

Since $\cJ$ is closed, this implies that $\alpha \beta \in \cJ$.
Moreover, one can verify that $\cP(Q,W) \simeq KQ/(\alpha \beta,
\beta \gamma, \gamma \alpha)$. In particular, we see that in the
presentation of the Jacobian algebra as quiver with relations,
the relations are not necessarily the cyclic derivatives of the potential.

Consider now the non-complete Jacobian algebra $A$ and assume that
$\ch K \neq 2$. Since $\alpha \beta = 2 \alpha \beta \gamma \alpha \beta$
in $A$, one has
\begin{align*}
2 \alpha \beta \gamma &= 4 \alpha \beta \gamma \alpha \beta \gamma
= (2 \alpha \beta \gamma)^2 ,\\
(e_1 - 2 \alpha \beta \gamma)^2
&= e_1^2 - 4 \alpha \beta \gamma + 4 \alpha \beta \gamma \alpha \beta \gamma
= e_1 - 2 \alpha \beta \gamma ,\\
2 \alpha \beta \gamma (e_1 - 2 \alpha \beta \gamma) &= 
(e_1 - 2 \alpha \beta \gamma) 2 \alpha \beta \gamma = 0 ,
\end{align*}
hence the idempotents in $A$ corresponding to the paths of length zero are
no longer primitive; for example, $e_1$ can be written as a sum
$e_1 = (e_1 - 2 \alpha \beta \gamma) + 2 \alpha \beta \gamma$
of two orthogonal idempotents. Using these idempotents one can
verify that the algebra $A$ decomposes into a direct
sum of $\cP(Q,W)$ and the matrix ring $M_3(K)$.
\end{example}

For a quiver with potential $(Q,W)$, Ginzburg~\cite[\S4.2]{Ginzburg06}
has defined a dg-algebra $\Gamma(Q,W)$ which is concentrated in
non-positive degrees and its zeroth cohomology is isomorphic to the
Jacobian algebra, i.e.\ $\hh^0(\Gamma(Q,W)) \simeq \cP(Q,W)$.
In~\cite[Theorem~6.3]{Keller11}, Keller shows that the Ginzburg dg-algebra
$\Gamma=\Gamma(Q,W)$ is homologically smooth and bimodule 3-Calabi-Yau,
that is, $\RHom_{\Gamma^e}(\Gamma, \Gamma^e) \simeq \Gamma[-3]$ in
$\cD(\Gamma^e)$, where $\Gamma^e = \Gamma^{op} \otimes_K \Gamma$.

Given a dg-algebra $\Gamma$ which is concentrated in non-positive degrees,
homologically smooth, bimodule 3-Calabi-Yau and whose zeroth cohomology
$\hh^0(\Gamma)$ is finite-dimensional, 
Amiot constructs in~\cite[\S2]{Amiot09} a triangulated 
2-Calabi-Yau category with a cluster-tilting object
whose endomorphism algebra is $\hh^0(\Gamma)$. She then
applies this construction to $\Gamma(Q,W)$ for quivers with potentials
$(Q,W)$ whose Jacobian algebra is finite-dimensional to obtain
the \emph{generalized cluster category} associated with $(Q,W)$
\cite[Theorem~3.5]{Amiot09}.

\begin{proposition}[\protect{\cite[Corollary~3.6]{Amiot09}}]
Any finite-dimensional Jacobian algebra of a quiver with potential is
2-CY-tilted.
\end{proposition}

A notion of equivalence of quivers with potentials was introduced by Derksen,
Weyman and Zelevinsky~\cite{DWZ08}. Let $Q$ be a quiver. Any continuous
algebra automorphism $\vphi \colon \wh{KQ} \to \wh{KQ}$ induces a continuous
linear
automorphism, denoted $\overline{\vphi}$, of the topological vector space
$\HH_0(\wh{KQ})=\wh{KQ}/\overline{[\wh{KQ},\wh{KQ}]}$.
For a vertex $i$ of $Q$, denote by $e_i$ the path of length zero at $i$.
It is an idempotent of the algebra $\wh{KQ}$.

\begin{definition}[\protect{\cite[Definition~4.2]{DWZ08}}]
Two potentials $W$ and $W'$ on $Q$ are \emph{right equivalent} if there exists
a continuous algebra automorphism $\vphi \colon \wh{KQ} \to \wh{KQ}$ satisfying
$\vphi(e_i)=e_i$ for each $i \in Q_0$ and $W'=\overline{\vphi}(W)$ in
$\HH_0(\wh{KQ})$.
\end{definition}

The Ginzburg dg-algebras of right equivalent potentials are isomorphic
and hence also their Jacobian algebras~\cite[Lemmas~2.8 and~2.9]{KellerYang11}.
If the latter are finite-dimensional, then the associated 2-Calabi-Yau
categories are equivalent as triangulated categories, since they depend only on
the corresponding Ginzburg dg-algebras.

\begin{example}
Consider the quiver with potential $(Q,W)$ of Example~\ref{ex:QP} and let
$W'=\alpha \beta \gamma$ be another potential on $Q$.
A continuous algebra automorphism of $\wh{KQ}$ fixing each $e_i$ is determined
by its value on the arrows. The endomorphism $\vphi$ whose value on the arrows
is given by
\begin{align*}
\vphi(\alpha) = \alpha - \alpha \beta \gamma \alpha &,&
\vphi(\beta)  = \beta &,&
\vphi(\gamma) = \gamma
\end{align*}
is an automorphism of $\wh{KQ}$; indeed,
\[
\vphi^{-1}(\alpha) = \alpha + \alpha \beta \gamma \alpha
+ 2 (\alpha \beta \gamma)^2 \alpha + 5 (\alpha \beta \gamma)^3 \alpha
+ 14 (\alpha \beta \gamma)^4 \alpha + \dots
\]
(where the coefficients are the Catalan numbers).
Moreover, $\vphi(W')=W$, hence the potentials $W$ and $W'$ are right
equivalent.
\end{example}

\subsection{Hyperpotentials}
The following extension of the notion of a potential, introduced 
in~\cite{Ladkani14a}, allows to prove that certain algebras defined over
ground fields of positive characteristic are 2-CY-tilted. This will
be particularly important when considering blocks of group algebras.

\begin{definition}[\cite{Ladkani14a}]
A \emph{hyperpotential} on $Q$ is an element in $\HH_1(\wh{KQ})$.
In explicit terms, it is a collection of elements
$(\rho_\alpha)_{\alpha \in Q_1}$ in $\wh{KQ}$ indexed by the arrows of $Q$ satisfying the
following conditions:
\begin{enumerate}
\renewcommand{\theenumi}{\roman{enumi}}
\item
If $\alpha \colon i \to j$ then $\rho_\alpha \in e_j \wh{KQ} e_i$.
In other words, $\rho_\alpha$ is a (possibly infinite) linear combination
of paths starting at $j$ and ending at $i$.

\item \label{it:comm}
$\sum_{\alpha \in Q_1} \alpha \rho_\alpha =
\sum_{\alpha \in Q_1} \rho_\alpha \alpha$ in $\wh{KQ}$.
\end{enumerate}

The \emph{Jacobian algebra} of $(\rho_\alpha)_{\alpha \in Q_1}$ is the
quotient of $\wh{KQ}$ by the closure of the ideal generated by
the elements $\rho_\alpha$,
\[
\cP(Q, (\rho_\alpha)_{\alpha \in Q_1}) = \wh{KQ}/
\overline{(\rho_\alpha : \alpha \in Q_1)} .
\]
\end{definition}

Any potential $W$ gives rise to a hyperpotential by taking its cyclic
derivatives $(\partial_\alpha W)_{\alpha \in Q_1}$. This is essentially
Connes' map $B$ from $\HC_0(\wh{KQ})$ to $\HH_1(\wh{KQ})$.
Conversely, when $\ch K = 0$, any hyperpotential arises in this way, see
the discussion at the end of~\cite[\S6.1]{Keller11}.

It is possible to define a Ginzburg dg-algebra for a hyperpotential and
follow Keller's proof to show that it has the same homological properties
as in the case of potentials, see~\cite{Ladkani14a}.
Therefore Amiot's construction applies and we deduce the following.

\begin{proposition} \label{p:hyperCY}
Any finite-dimensional Jacobian algebra of a quiver with hyperpotential is
2-CY-tilted.
\end{proposition}

\begin{example} \label{ex:KxnCY}
Consider the algebra $A=K[x]/(x^n)$ of Example~\ref{ex:Kxn}
over a field $K$ with characteristic $p \geq 0$, and
consider the quiver $Q$ consisting of one vertex and one loop, denoted $x$,
at that vertex.
If $p$ does not divide $n+1$, then for any $c \in K^{\times}$, the algebra
$A$ is the Jacobian algebra of the potential $W=cx^{n+1}$ on $Q$.
However, if $p$ divides $n+1$, then $A$ is not a Jacobian algebra of a
potential on $Q$.
Nevertheless, the sequence consisting of the single element $x^n$ is always
a hyperpotential on $Q$,
hence $A$ is 2-CY-tilted regardless of the characteristic of $K$.
\end{example}

\subsection{Periodicity}

A large class of symmetric 2-CY-tilted algebras has been constructed by
Burban, Iyama, Keller and Reiten~\cite{BIKR08}. In their construction, the
ambient 2-Calabi-Yau triangulated categories are the stable categories of
maximal Cohen-Macaulay
modules over odd dimensional isolated hypersurface singularities.
These categories are also $0$-Calabi-Yau since the square of the suspension
functor is isomorphic to the identity. Therefore, the endomorphism algebra of
any object is symmetric (cf.\ Proposition~\ref{p:symmetric}).

The next proposition provides a partial converse. We start with one
cluster-tilting object in a 2-Calabi-Yau category $\cC$ whose endomorphism
algebra $\gL$ is symmetric and study the implications this has on the structure
of $\cC$ and $\stmod \gL$.

\begin{proposition} \label{p:period}
Let $\gL$ be a finite-dimensional symmetric algebra that is also
2-CY-tilted, i.e.\ $\gL = \End_{\cC}(T)$ for some cluster-tilting object $T$
within a triangulated 2-Calabi-Yau category $\cC$ with suspension functor
$\Sigma$.
\begin{enumerate}
\renewcommand{\theenumi}{\alph{enumi}}
\item \label{it:p:omega4}
The functor $\Omega^4$ on the stable module category $\stmod \gL$
is isomorphic to the identity, hence all non-projective
$\gL$-modules are $\Omega$-periodic with period dividing $4$.

\item \label{it:p:sigma2}
The functor $\Sigma^2$ acts as the identity on the objects of $\cC$.

\item \label{it:p:nonrigid}
Assume that $\gL$ is a Jacobian algebra of a hyperpotential.
Then this hyperpotential is rigid if and only if $\gL$ is semi-simple.
\end{enumerate}
\end{proposition}

For part~\eqref{it:p:nonrigid}, note that rigid quivers with potentials have
been defined in~\cite[Definitions~3.4 and~6.10]{DWZ08} in terms of vanishing
of the deformation space of their Jacobian algebras. This definition carries
over without any modification to hyperpotentials. In particular,
a hyperpotential with finite-dimensional Jacobian algebra $\gL$ is
\emph{rigid} if and only if $\HH_0(\gL) = \gL/[\gL,\gL]$ is spanned by the
images of the primitive idempotents corresponding to the vertices.

Let us give the short proof of part~\eqref{it:p:omega4}. We note that
parts~\eqref{it:p:omega4} and~\eqref{it:p:sigma2} of the proposition have
also been recently observed by Valdivieso-Diaz~\cite{Valdivieso13}.

\begin{proof}[Proof of part~(a)]
On the one hand, $\gL$ is symmetric, hence $\stmod \gL$ is $(-1)$-Calabi-Yau
(Remark~\ref{rem:CYminus1}). On the other hand, $\gL$ is $2$-CY-tilted, hence
$\stmod \gL$ is $3$-Calabi-Yau (Prop.~\ref{p:2CYtilted}). The uniqueness of
the Serre functor implies that the fourth power of the suspension on
$\stmod \gL$ is isomorphic to the identity functor, and since the suspension
is $\Omega^{-1}$, we get the result.
\end{proof}

\begin{example}
Let $n \geq 1$ and consider the algebra $A = K[x]/(x^n)$. It is symmetric
and 2-CY-tilted (Example~\ref{ex:KxnCY}). By Proposition~\ref{p:period},
$\Omega^4_A M \simeq M$ for any $M \in \stmod A$.
Indeed, in this case even $\Omega^2_A M \simeq M$, see Example~\ref{ex:Kxn}.
\end{example}

As a direct consequence of Proposition~\ref{p:period}\eqref{it:p:omega4} and
Definition~\ref{def:quasi}, we obtain the next statement.

\begin{corollary}
An indecomposable, symmetric, 2-CY-tilted algebra of tame representation type
is of quasi-quaternion type.
\end{corollary}

\subsection{Derived equivalences}
\label{ssec:dereq}

In this section all cluster-tilting objects are assumed to be
\emph{basic}, i.e.\ they decompose into a direct sum of
non-isomorphic indecomposable objects.
Iyama and Yoshino~\cite{IyamaYoshino08} have shown that there is
a well-defined notion of mutation of (basic) cluster-tilting objects in a
triangulated 2-Calabi-Yau category $\cC$.

\begin{proposition}[\protect{\cite[Theorem~5.3]{IyamaYoshino08}}]
Let $T$ be a cluster-tilting object in $\cC$, let $X$ be an 
indecomposable summand of $T$ and write $T = \bar{T} \oplus X$.
Then there exists a unique indecomposable object $X'$ of $\cC$ which
is not isomorphic to $X$ such that $T' = \bar{T} \oplus X'$ is 
a cluster-tilting object in $\cC$.
\end{proposition}

The cluster-tilting object $T'$ in the proposition is called the
\emph{Iyama-Yoshino mutation} of $T$ at $X$. The algebras
$\gL=\End_{\cC}(T)$ and $\gL'=\End_{\cC}(T')$ are said to be
\emph{neighboring} 2-CY-tilted algebras.

Let $(Q,W)$ be a quiver with potential and let $k$ be a vertex in $Q$
such that no 2-cycle (i.e.\ a cycle of length 2) passes through $k$.
Derksen, Weyman and Zelevinsky have defined in~\cite[\S5]{DWZ08}
the \emph{mutation} of $(Q,W)$ at $k$, which is a quiver with
potential denoted $\mu_k(Q,W)$.
Buan, Iyama, Reiten and Smith have shown in~\cite{BIRS11} that 
under some mild conditions the notions of Iyama-Yoshino
mutation and mutation of quivers with potentials are compatible.
This is expressed in the next proposition.
\begin{proposition}[\protect{\cite[Theorem~5.2]{BIRS11}}] \label{p:QPmutIY}
Let $T$ be a cluster-tilting object in $\cC$. Assume that
$\End_{\cC}(T) \simeq \cP(Q,W)$ for some quiver with potential $(Q,W)$
and that $\End_{\cC}(T)$ satisfies the vanishing condition.
Let $k$ be a vertex of $Q$ such that no $2$-cycle passes through $k$,
let $X$ be the corresponding indecomposable summand of $T$ and let
$T'$ be the Iyama-Yoshino mutation of $T$ at $X$. Then
$\End_{\cC}(T') \simeq \cP(\mu_k(Q,W))$.
\end{proposition}
For the precise formulation of the vanishing condition we refer the
reader to~\cite{BIRS11}, but for our purposes it is sufficient to note
that this condition holds when the algebra $\End_{\cC}(T)$ is
self-injective, and in particular when it is symmetric.

Neighboring 2-CY-tilted algebras are nearly Morita equivalent in the
sense of Ringel~\cite{Ringel07}, that is, there is an equivalence of categories
\[
\modf \gL/\add S \simeq \modf \gL'/\add S'
\]
where $S$ (respectively, $S'$) is the simple module which is the top of the
indecomposable projective $\gL$-module (respectively, $\gL'$-module)
corresponding to the summand $X$ of $T$ (respectively, $X'$ of $T'$),
provided there are ``no loops'', i.e.\ any non-isomorphism $X \to X$
(or $X' \to X'$) factors through $\add \bar{T}$,
see~\cite[Proposition~2.2]{KellerReiten07}.
However, neighboring 2-CY-tilted algebras are not necessarily derived
equivalent, see for example~\cite[Example~5.2]{Ladkani10}.

The next statement concerns the derived equivalence of neighboring
2-CY-tilted algebras. 
It is an improvement of~\cite[Theorem~5.3]{Ladkani10} which has
turned out to be a very useful tool in derived equivalence
classifications of various cluster-tilted algebras and Jacobian
algebras~\cite{BHL13,BHL14,Ladkani11a}.
The derived equivalences are instances of
(refined version of) good mutations introduced in our previous
work~\cite{Ladkani10}.
Before formulating the result, we recall some relevant notions.

Let $\gL$ be a basic algebra and $P$ an indecomposable projective
$\gL$-module and write $\gL = P \oplus Q$. Consider the silting
mutations in the sense of Aihara and Iyama~\cite{AiharaIyama12}
of $\gL$ at $P$ within the triangulated category $\per \gL$
of perfect complexes, which are the following two-term complexes
\begin{align} \label{e:silt}
U^-_P(\gL) = (P \to Q') \oplus Q &,&
U^+_P(\gL) = (Q'' \to P) \oplus Q ,
\end{align}
where $Q', Q'' \in \add Q$, the maps are left (resp., right)
$(\add Q)$-approximations and $Q, Q', Q''$ are in degree 0.
These two-term complexes of projective modules are known also as
Okuyama-Rickard complexes. In~\cite{Ladkani10} we considered these
complexes in relation with our definition of mutations of algebras.

An algebra is \emph{weakly symmetric} if for any simple module, its
projective cover is isomorphic to its injective envelope. Symmetric
algebras are weakly symmetric and if $\gL$ is weakly symmetric, then
the complexes $U^-_P(\gL)$ and $U_P^+(\gL)$ are tilting complexes.

\begin{proposition} \label{p:dereq}
Let $T$ be a cluster-tilting object in a triangulated 2-Calabi-Yau category
$\cC$, let $X$ be an indecomposable summand of $T$ and let $T'$ be the
Iyama-Yoshino mutation of $T$ at $X$.
Consider the algebras $\gL = \End_{\cC}(T)$ and $\gL'= \End_{\cC}(T')$.
Let $P$ be the indecomposable projective $\gL$-module corresponding to $X$ and
let $P'$ be the indecomposable projective $\gL'$-module corresponding to $X'$.
\begin{enumerate}
\renewcommand{\theenumi}{\alph{enumi}}
\item
If $U^-_P(\gL)$ and $U^+_{P'}(\gL')$ are tilting complexes
(over $\gL$ and $\gL'$, respectively), then
\[
\End_{\dgL} U^-_P(\gL) \simeq \gL' \qquad \text{and} \qquad
\End_{\cD^b(\gL')} U^+_{P'}(\gL') \simeq \gL .
\]

\item
If $U^+_P(\gL)$ and $U^-_{P'}(\gL')$ are tilting complexes
(over $\gL$ and $\gL'$, respectively), then
\[
\End_{\dgL} U^+_P(\gL) \simeq \gL' \qquad \text{and} \qquad
\End_{\cD^b(\gL')} U^-_{P'}(\gL') \simeq \gL .
\]

\item \label{it:algmut}
If $\gL$ is weakly symmetric, then $\gL'$ is also weakly symmetric
by~\cite[\S4.2]{HerschendIyama11}, hence all
the complexes $U^-_P(\gL)$, $U^+_P(\gL)$, $U^-_{P'}(\gL')$
and $U^+_{P'}(\gL')$ are tilting complexes and
\[
\End_{\dgL} U^-_P(\gL) \simeq \gL' \simeq
\End_{\dgL} U^+_P(\gL) .
\]
In particular, $\gL$ and $\gL'$ are derived equivalent.

\item
If $\gL$ is symmetric then $\gL'$ is symmetric.
\end{enumerate}
\end{proposition}
We note that there are related works by Dugas~\cite{Dugas15}
concerning derived equivalences of symmetric algebras and by
Mizuno~\cite{Mizuno15} concerning derived equivalences of
self-injective quivers with potential.

As the category of perfect complexes over a symmetric algebra is 0-Calabi-Yau,
the derived equivalences in part~\eqref{it:algmut} can be
considered as 0-CY analogs of the derived equivalences of
Iyama-Reiten~\cite{IyamaReiten08} and
Keller-Yang~\cite[Theorem~6.2]{KellerYang11}
for 3-CY-algebras.

\begin{definition}
Let $T$ be a cluster-tilting object in a triangulated 2-Calabi-Yau
category $\cC$.
A cluster-tilting object $T'$ in $\cC$ is \emph{reachable} from $T$
if it can be obtained from $T$ by finitely many Iyama-Yoshino mutations at
indecomposable summands.
\end{definition}

\begin{corollary} \label{c:dereq}
Let $T$ be a cluster-tilting object in a triangulated 2-Calabi-Yau 
category $\cC$ and assume that $\gL=\End_{\cC}(T)$ is (weakly) symmetric.
Then for any cluster-tilting object $T'$ in $\cC$ that is reachable from $T$,
the algebra $\gL'=\End_{\cC}(T')$ is (weakly) symmetric and derived
equivalent to $\gL$.
\end{corollary}

\begin{remark}
There are examples of triangulated 2-Calabi-Yau categories $\cC$ with a
cluster-tilting object $T$ such that $\Sigma T$ is not reachable from
$T$, see~\cite[\S3]{Ladkani13} and~\cite[Example~4.3]{Plamondon13}.
Interestingly, in all of these examples the algebra $\End_{\cC}(T)$ is
symmetric. Note, however, that $\End_{\cC}(\Sigma T) \simeq \End_{\cC}(T)$ and
in particular these algebras are derived equivalent.
\end{remark}

We can rephrase part~\eqref{it:algmut} of Proposition~\ref{p:dereq} as
follows.
\begin{corollary} \label{c:sym2CYmut}
Let $\gL$ be a weakly symmetric 2-CY-tilted algebra and let $P$ be
an indecomposable projective $\gL$-module.
Then the two algebras $\End_{\dgL} U^-_P(\gL)$
and $\End_{\dgL} U^+_P(\gL)$ are isomorphic,
2-CY-tilted and derived equivalent to~$\gL$.
\end{corollary}

We see that derived equivalences of a particular kind preserve the property
of an algebra being symmetric 2-CY-tilted. One may ask whether this is still
true for arbitrary derived equivalences.

\begin{question} \label{q:sym2CYder}
Let $\gL$ be a symmetric 2-CY-tilted algebra and let $\gL'$ be an
algebra derived equivalent to $\gL$. Is $\gL'$ also 2-CY-tilted?
\end{question}

One may also ask if a converse to Proposition~\ref{p:period}\eqref{it:p:omega4}
holds.

\begin{question} \label{q:sym42CY}
Let $\gL$ be a symmetric algebra such that $\Omega_\gL^4 M \simeq M$ for
any $M \in \stmod \gL$. Is $\gL$ then 2-CY-tilted?
\end{question}

Observe that by Proposition~\ref{p:derquasi} and Proposition~\ref{p:period},
an affirmative answer to Question~\ref{q:sym42CY} will yield
an affirmative answer to Question~\ref{q:sym2CYder}.
We note that the answer to Question~\ref{q:sym42CY}
(and hence Question~\ref{q:sym2CYder}) is positive in the following cases:
$\gL$ is of finite representation type (Theorem~\ref{t:sym2CYfin});
$\gL$ is tame with non-singular Cartan matrix (Theorem~\ref{t:quat2CY}); or
$\gL$ is tame of polynomial growth (Proposition~\ref{p:sympoly}).

\section{Ribbon quivers and triangulation quivers}
\label{sec:quivers}

In this section we develop a theory of ribbon quivers and ribbon
graphs, with an emphasis on a particular class of ribbon quivers called
triangulation quivers. The connections to ideal triangulations of marked
surfaces and dimer models will be explained in Section~\ref{sec:surface}.
Ribbon quivers and triangulation quivers are the combinatorial ingredients
underlying the definition of Brauer graph algebras and triangulation algebras
which will be introduced in Section~\ref{sec:algebras} and studied later in
this survey.
The combinatorial statements in this section will be stated without proofs,
and the details will appear elsewhere.

\subsection{Ribbon quivers}
\label{ssec:ribbon}

Recall from Section~\ref{ssec:QP} that a quiver $Q$ is quadruple
$Q=(Q_0,Q_1,s,t)$ where $Q_0,Q_1$ are finite sets and
$s,t \colon Q_1 \to Q_0$.

\begin{definition} \label{def:ribbon}
A \emph{ribbon quiver} is a pair $(Q,f)$ consisting of a quiver
$Q$ and a permutation $f \colon Q_1 \to Q_1$ on its set of arrows
satisfying the following conditions:
\begin{enumerate}
\renewcommand{\theenumi}{\roman{enumi}}
\item \label{it:ribbon:deg2}
At each vertex $i \in Q_0$ there are exactly two arrows starting at $i$ and
two arrows ending at $i$;

\item \label{it:ribbon:f}
For each arrow $\alpha \in Q_1$, the arrow $f(\alpha)$ starts where $\alpha$
ends.
\end{enumerate}
Note that loops are allowed in $Q$. A loop at a vertex is counted both
as an incoming and outgoing arrow at that vertex.
\end{definition}

\begin{example} \label{ex:ribbon1}
Consider a ribbon quiver $(Q,f)$ with one vertex.
Condition~\eqref{it:ribbon:deg2} implies that $Q$ must have two loops
as in the following picture
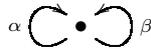
\[
\xymatrix@=1pc{
{\bullet} \ar@(dl,ul)[]^{\alpha} \ar@(dr,ur)[]_{\beta}
}
\]
and condition~\eqref{it:ribbon:f} is empty in this case, so that $f$ equals
one of the two permutations $f_1$ or $f_2$ on $Q_1$ given in cycle form 
by $f_1 = (\alpha)(\beta)$ and $f_2 = (\alpha \, \beta)$.
In particular we see that the underlying quiver does not determine the
ribbon quiver structure.
\end{example}

Let $(Q,f)$ be a ribbon quiver.
Since at each vertex of $Q$ there are exactly two outgoing arrows,
there is an involution $\alpha \mapsto \balpha$ on $Q_1$ mapping
each arrow $\alpha$ to the other arrow starting at the vertex $s(\alpha)$.
Composing it with $f$ gives rise to the permutation
$g \colon Q_1 \to Q_1$ given by $g(\alpha) = \overline{f(\alpha)}$
so that for each arrow $\alpha$,
the set $\{f(\alpha),g(\alpha)\}$ consists of
the two arrows starting at the vertex which $\alpha$ ends at.

Denote by $Q_1^f$ and $Q_1^g$ the subsets of arrows fixed by $f$ and $g$,
respectively, i.e.\ $Q_1^f = \{ \alpha \in Q_1 : f(\alpha)=\alpha \}$ and
$Q_1^g = \{ \alpha \in Q_1 : g(\alpha)=\alpha \}$. The set of loops in $Q$
thus decomposes as a disjoint union $Q_1^f \cup Q_1^g$.

Given a quiver $Q$ satisfying condition~\eqref{it:ribbon:deg2} in the
definition, the data of the permutation $f$ is equivalent to the data of the
permutation $g$. Thus from now on when considering a ribbon quiver $(Q,f)$
we will freely refer to the involution $\alpha \mapsto \balpha$ and
the permutation $g$ as defined above.

\begin{lemma} \label{l:gf2fg2}
Let $\alpha \in Q_1$. Then $f^{-1}(\alpha)=g^{-1}(\balpha)$
and $gf^{-2}(\alpha) = fg^{-2}(\balpha)$.
\end{lemma}

\begin{definition}
Let $(Q,f)$ be a ribbon quiver and define $g \colon Q_1 \to Q_1$ by
$g(\alpha) = \overline{f(\alpha)}$.
The \emph{dual} of $(Q,f)$ is the ribbon quiver $(Q,g)$.
\end{definition}

\begin{example}
In Example~\ref{ex:ribbon1}, $\bar{\alpha} = \beta$ and $\bar{\beta} = \alpha$,
so in cycle form $g_1 = (\alpha \, \beta) = f_2$ and
$g_2 = (\alpha)(\beta) = f_1$.
Hence $(Q,f_1)$ and $(Q,f_2)$ are dual to each other.
\end{example}

\begin{definition}
Let $(Q,f)$ and $(Q',f')$ be ribbon quivers with $Q=(Q_0, Q_1, s, t)$,
$Q'=(Q'_0, Q'_1, s', t')$. Recall that a pair of bijections
$\vphi_0 \colon Q_0 \xrightarrow{\sim} Q'_0$ and
$\vphi_1 \colon Q_1 \xrightarrow{\sim} Q'_1$ is an \emph{isomorphism} between
the quivers $Q$ and $Q'$ if $\vphi_0 s = s' \vphi_1$ and
$\vphi_0 t = t' \vphi_1$.
If, in addition, $\vphi_1 f = f' \vphi_1$ and $\vphi_1(\balpha)
= \overline{\vphi_1(\alpha)}$ for any $\alpha \in Q_1$, 
we say that $(\vphi_0, \vphi_1)$ is \emph{isomorphism} between the ribbon
quivers $(Q,f)$ and $(Q',f')$.
\end{definition}

Ribbon quivers are closely related to ribbon graphs.
To avoid confusion, we shall use the term ``node'' for the graph
in order to distinguish it from a vertex in the quiver.
Informally speaking, a ribbon graph is a graph consisting of nodes and
edges together with a cyclic ordering of the edges around each node.
This can be made more formal in the next definition.

\begin{definition}
A \emph{ribbon graph} is a triple $(H, \iota, \sigma)$ where $H$ is a finite
set, $\iota$ is an involution on $H$ without fixed points and
$\sigma$ is a permutation on $H$.
\end{definition}

The elements of $H$ are called \emph{half-edges}.
A ribbon graph gives rise to a graph $(V,E)$ (possibly with loops and
multiple edges between nodes) as follows.
The set $V$ of nodes consists of the cycles of $\sigma$ and
the set $E$ of edges consists of the cycles of $\iota$.
An edge $e \in E$ can be written as $(h \, \iota(h))$ for some $h \in H$.
The $\sigma$-cycles that $h$ and $\iota(h)$ belong to are the nodes that
$e$ is incident to.
Moreover, $\sigma$ induces a cyclic ordering of the edges around each node.

Conversely, given a graph $(V,E)$ with a cyclic ordering of the edges around
each node, we think of each edge $e \in E$ incident to the nodes $v', v'' \in V$
(which may coincide) as composed of two half-edges $e'$ and $e''$, with $e'$
incident to $v'$ and $e''$ incident to $v''$. This yields a ribbon graph
$(H,\iota,\sigma)$ where $H$ is the set of all half-edges,
$\iota = \prod_{e \in E} (e' \, e'')$ is the product of all the transpositions
$(e' \, e'')$ for $e \in E$, and
for any half-edge $h$ incident to a node $v$, the half-edge $\sigma(h)$
is the one following $h$ in the cyclic order around $v$.

\begin{example} \label{ex:ribbon1g}
Consider a ribbon graph with one edge. In this case the set $H$ of half-edges
consists of two elements, which we denote by $\alpha$ and $\beta$, and the
involution $\iota$ can be written as $\iota = (\alpha \, \beta)$ in cycle form.
The permutation $\sigma$ equals one of the two permutations $\sigma_1$ or
$\sigma_2$ given in cycle form by $\sigma_1 = (\alpha \, \beta)$ and
$\sigma_2 = (\alpha)(\beta)$.

The corresponding graphs, with their half-edges labeled, are shown in the
picture below. Since $\sigma_1$ has one cycle, the graph of
$(H,\iota,\sigma_1)$, shown to the left, has one node.
Similarly, since $\sigma_2$ has two cycles, the graph of $(H,\iota,\sigma_2)$,
shown to the right, has two nodes.
\begin{align*}
\begin{array}{cc}
\xymatrix{{\circ} \ar@{-}@(ul,dl)_(0.15)\alpha_(0.85)\beta} &
\sigma_1 = (\alpha \, \beta)
\end{array}
&& &&
\begin{array}{cc}
\xymatrix{{\circ} \ar@{-}[r]^(0.25)\alpha_(0.75)\beta & {\circ}} &
\sigma_2 = (\alpha)(\beta)
\end{array}
\end{align*}
\end{example}

\begin{definition}
Let $(H,\iota,\sigma)$ and $(H',\iota',\sigma')$ be ribbon graphs.
An \emph{isomorphism} between $H$ and $H'$ is a bijection
$\vphi \colon H \to H'$ satisfying $\iota' \vphi = \vphi \iota$
and $\sigma' \vphi = \vphi \sigma$.
\end{definition}

Any ribbon quiver $(Q,f)$ gives rise to a ribbon graph
$(H,\iota,\sigma)$ by taking $H=Q_1$ and defining $\iota(\alpha)=\balpha$
and $\sigma(\alpha) = \overline{f(\alpha)}$ for each $\alpha \in Q_1$.

Conversely, a ribbon graph $(H,\iota,\sigma)$ gives rise to a ribbon quiver
$(Q,f)$ as follows.
Set $Q_1=H$ and take $Q_0$ to be the set of cycles of $\iota$. Define
the maps $s,t \colon Q_1 \to Q_0$ and the permutation $f \colon Q_1 \to Q_1$
by letting, for each $h \in H$, $s(h)$ to be the $\iota$-cycle that $h$ belongs
to and setting $t = s \sigma$ and $f = \iota \sigma$.

Note that these two constructions are inverses of each other, hence we deduce
the following.

\begin{proposition} \label{p:ribbon}
There is a bijection between the set of
isomorphism classes of ribbon quivers and the set of isomorphism classes
of ribbon graphs,
\[
\left(
\{\text{ribbon quivers}\}/\simeq
\right) \longleftrightarrow
\left(
\{\text{ribbon graphs} \}/\simeq
\right) .
\]
\end{proposition}

Under this bijection, the various notions concerning ribbon quivers and
ribbon graphs are related as in the dictionary given in Table~\ref{tab:ribbon}.

\begin{table}
\begin{center}
\begin{tabular}{ccc}
$(Q,f)$ & $(H,\iota,\sigma)$ & $(V,E)$ \\ \hline
vertex & cycle of $\iota$ & edge \\
arrow & element of $H$ & half-edge \\
$f$ & $\iota \sigma$ \\
$g$ & $\sigma$ & cyclic ordering \\
cycle of $g$ & cycle of $\sigma$ & node
\end{tabular}
\end{center}
\caption{Dictionary between ribbon quivers and ribbon graphs.}
\label{tab:ribbon}
\end{table}

\begin{example} \label{ex:ribbonG1}
We illustrate the bijection between the ribbon quivers with one vertex
discussed in Example~\ref{ex:ribbon1} and the ribbon graphs with one edge
discussed in Example~\ref{ex:ribbon1g}.
We denote the set of half-edges by $\{\alpha, \beta\}$ and let
$\iota = (\alpha \, \beta)$. The underlying quiver $Q$ is always
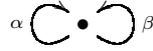
\[
\xymatrix@=1pc{
{\bullet} \ar@(dl,ul)[]^{\alpha} \ar@(dr,ur)[]_{\beta}
}
\]
and the corresponding graphs are shown in the right column below.
\[
\begin{array}{lclcc}
f_1 = (\alpha)(\beta) & &
\sigma_1 = g_1 = (\alpha \, \beta) & &
\xymatrix{{\circ} \ar@{-}@(ul,dl)}
\\[3pt]
f_2 = (\alpha \, \beta) & &
\sigma_2 = g_2 = (\alpha)(\beta) & &
\xymatrix{{\circ} \ar@{-}[r] & {\circ}}
\end{array}
\]
\end{example}

\medskip

The data of a graph can be encoded in matrix form in the following way.
Let $(V,E)$ be a graph. For a node $v \in V$, define a vector
$\chi_v \in \bZ^E$ by
\[
\chi_v(e) = \begin{cases}
2 & \text{$e$ is a loop incident to $v$,} \\
1 & \text{$e$ is incident to $v$ but is not a loop,} \\
0 & \text{$e$ is not incident to $v$}
\end{cases}
\]
and think of it as a row vector. Obviously, $\chi_v(e) \geq 0$ and
$\sum_{v \in V} \chi_v(e) = 2$ for any $e \in E$, so by arranging the vectors
$\chi_v$ as a $V \times E$ matrix, one gets an integer matrix with non-negative
entries whose sum of rows equals the constant vector $(2,2,\dots,2)$.
Conversely, any such matrix $\chi$ gives rise to a graph whose nodes are
indexed by the rows of $\chi$, its edges are indexed by the columns of $\chi$
and the incidence relations are read from the entries $\chi_v(e)$.

Now let $(Q,f)$ be a ribbon quiver. In the underlying graph $(V,E)$ of the
ribbon graph corresponding to $(Q,f)$ under the bijection of
Proposition~\ref{p:ribbon}, the set $V$ corresponds to the set $\Omega_g$ of
the cycles of the permutation $g$, the set $E$ corresponds to the set $Q_0$
of vertices of $Q$ and the entries of the matrix $\chi$ are given by
$\chi_{\omega}(i) = |\{\alpha \in \omega : s(\alpha)=i\}|$
for any $g$-cycle $\omega \in \Omega_g$ and vertex $i \in Q_0$.

\subsection{Triangulation quivers}

\begin{definition}
A \emph{triangulation quiver} is a ribbon quiver $(Q,f)$ such that
$f^3$ is the identity on the set of arrows.
\end{definition}

\begin{example} \label{ex:triang1}
Considering the ribbon quivers with one vertex of Example~\ref{ex:ribbon1},
we see that $(Q,f_1)$ is a triangulation quiver whereas $(Q,f_2)$ is not.
\end{example}

\begin{remark}
Given any quiver $Q$ satisfying condition~\eqref{it:ribbon:deg2} of
Definition~\ref{def:ribbon}, there is always at least one 
(and in general, many) permutation(s) $f$ on
the arrows making $(Q,f)$ a ribbon quiver. Indeed, for each $i \in Q_0$ 
label by $\alpha, \beta$ the arrows ending at $i$ and by $\gamma, \delta$
the arrows starting at $i$ and set, for instance,
$f(\alpha)=\gamma$ and $f(\beta)=\delta$.

However, as the next example demonstrates, there may not exist
a permutation~$f$ making $(Q,f)$ a triangulation quiver.
In other words, the existence of a triangulation
quiver $(Q,f)$ imposes some restrictions on the shape of a quiver $Q$.
\end{remark}

\begin{example}
Up to isomorphism, there are two ribbon quivers whose underlying quiver is the
one given below,
\[
\xymatrix{
{\bullet_1} \ar@<1.5ex>[r] \ar@<-0.5ex>[r] &
{\bullet_2} \ar@<1.5ex>[l] \ar@<-0.5ex>[l]
}
\]
Namely, denoting the arrows from $1$ to $2$ by $\alpha$, $\gamma$ and those
from $2$ to $1$ by $\beta$, $\delta$, the ribbon quivers are given by the
permutations
$(\alpha \beta)(\gamma \delta)$ and $(\alpha \beta \gamma \delta)$.
None of them is a triangulation quiver.
\end{example}

We have seen that not every quiver satisfying condition~\eqref{it:ribbon:deg2}
of Definition~\ref{def:ribbon} is an underlying quiver of a triangulation quiver.
The next proposition tells us that if such triangulation quiver exists,
then it is unique up to isomorphism.

\begin{proposition} \label{p:triauniq}
Let $(Q,f)$ and $(Q',f')$ be two triangulation quivers. If the
quivers $Q$ and $Q'$ are isomorphic, then $(Q,f)$ and $(Q',f')$
are isomorphic as ribbon quivers.
\end{proposition}

Since the number of triangulation quivers with a given number of vertices
is finite, they can be enumerated on a computer.
Table~\ref{tab:quivers} lists (up to isomorphism)
the connected triangulation quivers with at most
three vertices and their corresponding ribbon graphs.
Note that the ribbon graph of quiver $2$ could have also been drawn as
\[
\xymatrix{
{\circ} \ar@{-}[r]^1 & \circ \ar@{-}@(ur,dr)[]^{2}
}
\]
but the drawing in Table~\ref{tab:quivers} emphasizes the relation of this
quiver to the punctured monogon, as we shall see in~Section~\ref{sec:surface}.

\begin{table}
\[
\begin{array}{c|cc}
& \text{\textbf{Triangulation quiver}}
& \text{\textbf{Ribbon graph}}
\\ \hline & & \\
1 &
\begin{array}{c}
\xymatrix@=1pc{
{\bullet_1} \ar@(dl,ul)[]^{\alpha} \ar@(dr,ur)[]_{\beta}
}
\end{array}
&
\begin{array}{c}
\xymatrix{
\circ \ar@{-}@(ul,dl)_{1}
}
\end{array}
\\ & (\alpha) (\beta)
\\ \hline
2 &
\begin{array}{c}
\xymatrix{
{\bullet_1} \ar@(ul,dl)[]_{\alpha} \ar@<-0.5ex>[r]_{\beta}
& {\bullet_2} \ar@(dr,ur)[]_{\eta} \ar@<-0.5ex>[l]_{\gamma}
}
\\ 
(\alpha \beta \gamma) (\eta)
\end{array}
&
\begin{array}{c}
\xymatrix@=0.75pc{
{_2} & {\circ} \ar@{-}[rr]^1
& & {\circ} \ar@{-}@(ul,u)[lll]_{} \ar@{-}@(dl,d)[lll]_{}
}
\end{array}
\\ \hline
3a &
\begin{array}{c}
\xymatrix{
{\bullet_1} \ar@(ul,dl)[]_{\alpha} \ar@<-0.5ex>[r]_{\beta}
& {\bullet_2} \ar@<-0.5ex>[r]_{\delta} \ar@<-0.5ex>[l]_{\gamma}
& {\bullet_3} \ar@(dr,ur)[]_{\xi} \ar@<-0.5ex>[l]_{\eta}
}
\\
(\alpha \beta \gamma) (\delta \xi \eta)
\end{array}
&
\begin{array}{c}
\xymatrix@=0.75pc{
{_2} & {\circ} \ar@{-}[rr]^1
& & {\circ} \ar@{-}@(ul,u)[lll]_{} \ar@{-}@(dl,d)[lll]_{}
\ar@{-}[rr]^3
& & {\circ}
}
\end{array}
\\ \hline
3b &
\begin{array}{c}
\xymatrix@=1pc{
& {\bullet_3} \ar@<-0.5ex>[ddl]_{\alpha_3} \ar@<-0.5ex>[ddr]_{\beta_2} \\ \\
{\bullet_1} \ar@<-0.5ex>[rr]_{\alpha_1} \ar@<-0.5ex>[uur]_{\beta_3}
&& {\bullet_2} \ar@<-0.5ex>[ll]_{\beta_1} \ar@<-0.5ex>[uul]_{\alpha_2}
}
\\
(\alpha_1 \alpha_2 \alpha_3) (\beta_3 \beta_2 \beta_1)
\end{array}
&
\begin{array}{c}
\xymatrix@=1pc{
{\circ} \ar@{-}[rr]^3 && {\circ} \ar@{-}[ddl]^2 \\ \\
& {\circ} \ar@{-}[uul]^1
}
\end{array}
\\ \hline
3' &
\begin{array}{c}
\xymatrix@=1pc{
& {\bullet_3} \ar@(ur,ul)[]_{\alpha_3} \ar[ddl]_{\beta_3} \\ \\
{\bullet_1} \ar@(ul,dl)[]_{\alpha_1} \ar[rr]_{\beta_1}
&& {\bullet_2} \ar@(dr,ur)[]_{\alpha_2} \ar[uul]_{\beta_2}
}
\\
(\alpha_1) (\alpha_2) (\alpha_3) (\beta_1 \beta_2 \beta_3)
\end{array}
&
\begin{array}{c}
\xymatrix{
{\circ} \ar@{-}@(ur,ul)[]_3 \ar@{-}@(l,dl)[]_1 \ar@{-}@(r,dr)[]^2
}
\end{array}
\\ \hline
3'' &
\begin{array}{c}
\xymatrix@=1pc{
& {\bullet_3} \ar@<-0.5ex>[ddl]_{\alpha_2} \ar@<0.5ex>[ddl]^(.4){\alpha_5} \\ \\
{\bullet_1} \ar@<-0.5ex>[rr]_{\alpha_0} \ar@<0.5ex>[rr]^(.4){\alpha_3} &&
{\bullet_2} \ar@<-0.5ex>[uul]_{\alpha_1} \ar@<0.5ex>[uul]^(.4){\alpha_4}
}
\\
(\alpha_4 \alpha_2 \alpha_0) (\alpha_5 \alpha_3 \alpha_1)
\end{array}
&
\begin{array}{c}
\xymatrix@=0.75pc{
&& {_2} \ar@{-}@(r,r)[ddd] \\ \\ \\
&& {\circ} \ar@{-}@(l,l)[uuu] \ar@{-}@(ul,ul)[ddll] \ar@{-}@(dl,dl)[ddrr]
\\ \\
{_1} \ar@{-}@(dr,dr)[uurr]
&& && {_3} \ar@{-}@(ur,ur)[uull] \\
{}
}
\end{array}
\\ \hline
\end{array}
\]
\caption{The connected triangulation quivers with at most 3 vertices.
We list the triangulation quivers and the corresponding ribbon graphs,
where we write the permutation~$f$ in cycle form below each quiver.}
\label{tab:quivers}
\end{table}

\begin{remark}
As the entries in rows $3'$ and $3''$ of Table~\ref{tab:quivers} demonstrate,
two different ribbon graphs can have the same underlying graph (in this case
a node with three loops).
\end{remark}

\begin{remark}
In a triangulation quiver $(Q,f)$,
the permutations $\alpha \mapsto \balpha$ and $\alpha \mapsto f(\alpha)$
are of orders $2$ and $3$, respectively, hence the group
$\PSL_2(\bZ)$, which is the free product of the cyclic groups
$\bZ/2\bZ$ and $\bZ/3\bZ$, acts on the set of arrows $Q_1$.
This action is transitive when $Q$ is connected.
\end{remark}

The dual of a triangulation quiver $(Q,f)$ need not be a triangulation quiver.
However, when it is, then by Proposition~\ref{p:triauniq}, it must be isomorphic
to $(Q,f)$, hence $(Q,f)$ is \emph{self dual}. The next proposition shows that
there are only two connected self dual triangulation quivers.

\begin{proposition}
A connected triangulation quiver whose dual is also a triangulation quiver
is isomorphic to one of the two triangulation quivers shown in
Figure~\ref{fig:selfdual}.
\end{proposition}

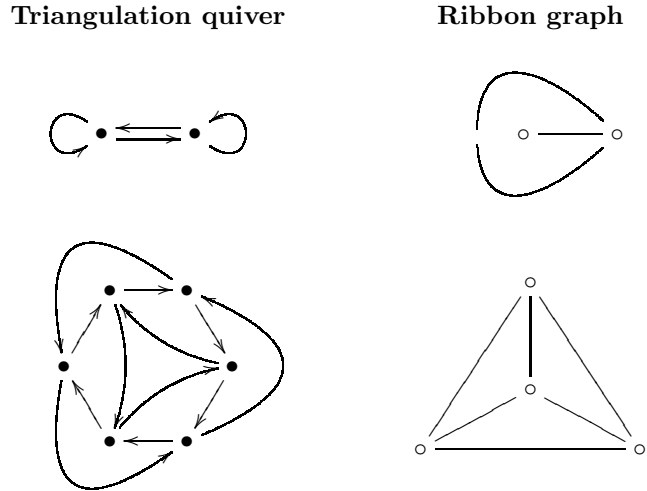
\begin{figure}
\[
\begin{array}{ccc}
 \text{\textbf{Triangulation quiver}}
&& \text{\textbf{Ribbon graph}} \\ \\
\begin{array}{c}
\xymatrix{
{\bullet} \ar@(ul,dl)[] \ar@<-0.5ex>[r]
& {\bullet} \ar@(dr,ur)[] \ar@<-0.5ex>[l]
}
\end{array}
& \qquad &
\begin{array}{c}
\xymatrix@=0.75pc{
&& {\circ} \ar@{-}[rr]
&& {\circ} \ar@{-}@(ul,u)[lll]_{} \ar@{-}@(dl,d)[lll]_{}
}
\end{array}
\\ \\ \\
\begin{array}{c}
\xymatrix@=0.5pc{
& {\bullet} \ar[rr] \ar@/^/[dddd]
&& {\bullet} \ar[ddr] \ar@/_3pc/[ddlll] \ar@{<-}@/^3pc/[dddd] \\ \\
{\bullet} \ar[uur] & && &
{\bullet} \ar[ddl] \ar@/^/[uulll] \ar@{<-}@/_/[ddlll] \\ \\
& {\bullet} \ar[uul]
&& {\bullet} \ar[ll] \ar@{<-}@/^3pc/[uulll]
}
\end{array}
&&
\begin{array}{c}
\xymatrix@=1pc{
&& {\circ} \ar@{-}[dddll] \ar@{-}[dddrr] \ar@{-}[dd] \\ \\
&& {\circ} \ar@{-}[dll] \ar@{-}[drr] \\
{\circ} \ar@{-}[rrrr] &&&& {\circ}
}
\end{array}
\end{array}
\]
\caption{The connected self dual triangulation quivers and the corresponding
ribbon graphs, a punctured monogon (top) and a tetrahedron (bottom).}
\label{fig:selfdual}
\end{figure}

We call the ribbon graph with two nodes appearing in Figure~\ref{fig:selfdual}
a \emph{punctured monogon}, for reasons that will become apparent in
Section~\ref{sec:surface}. Similarly, we call the ribbon graph with
four nodes appearing in Figure~\ref{fig:selfdual} a \emph{tetrahedron}.
In the triangulation quiver corresponding to the tetrahedron there are four
$f$-cycles and four $g$-cycles, each of length $3$, and for any arrow $\alpha$,
each of the arrows $\alpha, f(\alpha), \balpha, f(\balpha)$ belongs to a
different $g$-cycle.

\subsection{Block decomposition of triangulation quivers}
\label{ssec:blocks}

In this section we analyze the structure of triangulation quivers in terms
of three types of building blocks. This is similar in spirit to the
block decomposition of~\cite[\S13]{FST08}, however the number of blocks in
our case is smaller and only full matchings are used.

\begin{definition}
A \emph{block} is one of the three pairs, each consisting of a quiver
and a permutation on its set of arrows, shown in Figure~\ref{fig:blocks}.
A vertex of a block marked with white circle ($\circ$) is called
an \emph{outlet}.
\end{definition}

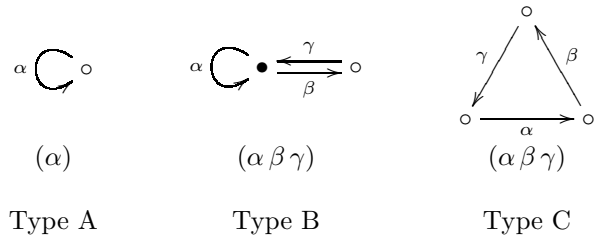
\begin{figure}
\[
\begin{array}{ccccc}
\begin{array}{c}
\xymatrix{
{\circ} \ar@(ul,dl)[]_{\alpha}
}
\end{array}
& &
\begin{array}{c}
\xymatrix{
{\bullet} \ar@(ul,dl)[]_{\alpha}
\ar@<-0.5ex>[r]_{\beta} & 
{\circ} \ar@<-0.5ex>[l]_{\gamma}
}
\end{array}
& &
\begin{array}{c}
\xymatrix@=1pc{
& {\circ} \ar[ddl]_{\gamma} \\ \\
{\circ} \ar[rr]_{\alpha} &&
{\circ} \ar[uul]_{\beta}
}
\end{array}
\\
(\alpha)
& &
(\alpha \, \beta \, \gamma)
& &
(\alpha \, \beta \, \gamma)
\\ \\
\text{Type A}
& &
\text{Type B}
& &
\text{Type C}
\end{array}
\]
\caption{Blocks for triangulation quivers. The permutation is given in
cycle form below each quiver.}
\label{fig:blocks}
\end{figure}

Let $B_1, B_2, \dots, B_s$ be a collection of blocks.
Denote by $V_1, V_2, \dots, V_s$ their corresponding sets of outlets
and let $V = \bigsqcup_{i=1}^s V_i$ be their disjoint union.
A \emph{matching} on $V$ is an involution $\theta \colon V \to V$ without
fixed points such that $\theta(V_i) \cap V_i$ is empty for each
$1 \leq i \leq s$
(in other words, an outlet cannot be matched to an outlet in the same block).

Given a collection of blocks and a matching $\theta$ on their outlets,
construct a quiver $Q$ and a permutation $f$ on its set of arrows as follows;
take the disjoint union of the blocks and identify each outlet
$v \in V$ with the outlet $\theta(v)$ to obtain $Q$.
The permutation $f$ on the set of arrows of $Q$ is induced by the permutations
on each of the blocks.

\begin{definition}
A pair $(Q,f)$ consisting of a quiver $Q$ and a permutation $f$ on its set of
arrows is \emph{block-decomposable} if it can be obtained by the above
procedure.
\end{definition}

\begin{proposition}
A block-decomposable pair $(Q,f)$ is a triangulation quiver.
Conversely, any triangulation quiver is block-decomposable.
\end{proposition}

\begin{example}
Since each of the blocks of types A and B has only one outlet, there is
only one way to match a pair consisting of two such blocks. In contrast,
there are two different ways to completely match two blocks of type C,
yielding the triangulation quivers $3b$ and $3''$ of Table~\ref{tab:quivers}.
The block decompositions of the triangulation quivers with at most three
vertices are given in Table~\ref{tab:block}.

\begin{table}
\begin{center}
\begin{tabular}{cc}
\textbf{Quiver} & \textbf{Block decomposition} \\
\hline
$1$ & A, A \\
$2$ & B, A \\
$3a$ & B, B \\
$3b$ & C, C \\
$3'$ & C, A, A, A \\
$3''$ & C, C
\end{tabular}
\end{center}
\caption{Block decompositions of the triangulation quivers with at most three
vertices. The numbers of the quivers refer to Table~\ref{tab:quivers}.}
\label{tab:block}
\end{table}
\end{example}

\begin{remark}
In a block decomposition of a triangulation quiver $(Q,f)$, the blocks of type
A are in bijection with the elements of $Q_1^f$, whereas those of type B are
in bijection with the elements of $Q_1^g$.
\end{remark}

In the theory of cluster algebras, quivers without loops (i.e.\ cycles of
length~$1$) and $2$-cycles (cycles of length~$2$) play an important role. The
block decomposition allows to quickly characterize those triangulation
quivers without loops and $2$-cycles. Indeed, a loop can only arise from a
block of types A or B, whereas a $2$-cycle arises either from a block
of type B or from gluing two blocks of type C, identifying two pairs of
vertices at opposing directions of the arrows. This can be rephrased as
follows.

\begin{proposition} \label{p:tri2cyc}
Let $(Q,f)$ be a triangulation quiver. Then the length of any non-trivial
cycle in $Q$ is at least 3 if and only if the following conditions hold:
\begin{enumerate}
\renewcommand{\theenumi}{\roman{enumi}}
\item
There are no arrows fixed by the permutation $f$; and

\item
the length of any cycle of the permutation $g$ is at least $3$.
\end{enumerate}
\end{proposition}

The block decomposition is also useful in proving the next statement.

\begin{proposition} \label{p:gcycles}
Let $(Q,f)$ be a triangulation quiver. Then the number of cycles of the
permutation $g$ does not exceed the number of vertices of $Q$, and equality
holds if and only if $(Q,f)$ is a disjoint union of any of the triangulation
quivers $1$, $2$, $3a$ or $3b$ of Table~\ref{tab:quivers}.
\end{proposition}

\section{Triangulations of marked surfaces and their quivers}
\label{sec:surface}

In this section we explain how (ideal) triangulations of marked
surfaces give rise to triangulation quivers. Marked surfaces
were considered by Fomin, Shapiro and Thurston~\cite{FST08} in their
work on cluster algebras from surfaces. Let us recall the setup and
definitions.

A \emph{marked surface} is a pair $(S,M)$ consisting of a compact,
connected, oriented, Riemann surface $S$ (possibly with boundary
$\partial S$) and a finite non-empty set $M$ of points in $S$,
called \emph{marked points}, such that each connected component of
$\partial S$ contains at least one point from $M$. The points in $M$
which are not on $\partial S$ are called \emph{punctures}.
We exclude the following surfaces:
\begin{itemize}
\item
a sphere with one or two punctures;

\item
an unpunctured digon;
\end{itemize}
(a \emph{sphere} is a surface of genus $0$ with empty boundary, a \emph{disc}
is a surface of genus $0$ with one boundary component, an \emph{$m$-gon} is a
disc with $m$ marked points on its boundary, and for $m=1,2,3$ an $m$-gon is
called \emph{monogon}, \emph{digon} and \emph{triangle}, respectively).

Up to homeomorphism, $(S,M)$ is determined by the following discrete data:
\begin{itemize}
\item
the genus $g$ of $S$;

\item
the number $b \geq 0$ of boundary components;

\item
the sequence $(n_1,n_2,\dots,n_b)$ where
$n_i \geq 1$ is the number of marked points on the $i$-th boundary component,
considered as a multiset;

\item
the number $p$ of punctures.
\end{itemize}

\subsection{Triangulation quivers from triangulations}

\begin{figure}
\begin{align*}
\begin{array}{c}
\xymatrix@=1pc{
& & & {\circ} \ar@{-}@/^/[dd]^k \\
{\circ} \ar@{-}@/_/[drrr]_i \ar@{-}@/^/[urrr]^j \\
& & & {\circ}
}
\end{array}
& &
\begin{array}{c}
\xymatrix@=0.5pc{
{\bullet_j} \ar[drr]^{f(\alpha)} \\
&& {\bullet_k} \ar[dll] \\
{\bullet_i} \ar[uu]^{\alpha}
}
\end{array}
\end{align*}
\caption{A triangle (left) and the corresponding block in the triangulation
quiver (right).}
\label{fig:triquiver}
\end{figure}
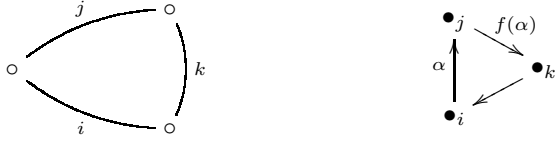

Let $(S,M)$ be a marked surface. An \emph{arc} $\gamma$ in $(S,M)$
is a curve in $S$ satisfying the following:
\begin{itemize}
\item
the endpoints of $\gamma$ are in $M$;
\item
$\gamma$ does not intersect itself, except that its endpoints may
coincide;
\item
the relative interior of $\gamma$ is disjoint from $M \cup \partial S$;
\item
$\gamma$ does not cut out an unpunctured monogon or an unpunctured digon.
\end{itemize}
Arcs are considered up to isotopy.
Two arcs are \emph{compatible} if there are curves in their respective
isotopy classes whose relative interiors do not intersect.
A \emph{triangulation} of $(S,M)$ is a maximal collection of pairwise
compatible arcs. The arcs of a triangulation cut the surface $S$ into
ideal triangles. The three sides of an ideal triangle need not be
distinct. Sides on the boundary of $S$ are called \emph{boundary
segments}.

\begin{definition}
Let $\tau$ be a triangulation of a marked surface $(S,M)$
which is not an unpunctured monogon. Construct a quiver
$Q_\tau$ and $f_\tau \colon (Q_\tau)_1 \to (Q_\tau)_1$ as follows:
\begin{itemize}
\item
The vertices of $Q_\tau$ are the arcs of $\tau$ together with
the boundary segments.
\item
At each vertex corresponding to a boundary segment add a loop
$\delta$ and set $f(\delta)=\delta$.

\item
For each ideal triangle in $\tau$ with sides $i, j, k$
(which may be arcs or boundary segments)
arranged in a clockwise order induced by the orientation of $S$,
add three arrows $i \xrightarrow{\alpha} j$, $j \xrightarrow{\beta} k$,
$k \xrightarrow{\gamma} i$ and set
$f(\alpha)=\beta$, $f(\beta)=\gamma$, $f(\gamma)=\alpha$ as in
Figure~\ref{fig:triquiver}.
\end{itemize}
\end{definition}

The next statement is immediate from the definitions, observing that in any
triangulation $\tau$, an arc $\gamma$ of $\tau$ is either the side of two
distinct triangles or there exists a triangle $\Delta$ such that two of its
sides are $\gamma$. In the latter case we say that the triangle $\Delta$ is
\emph{self-folded} and $\gamma$ is its \emph{inner side}.

\begin{lemma}
$(Q_\tau,f_\tau)$ is a triangulation quiver.
\end{lemma}

\begin{remark}
When $(S,M)$ is an unpunctured monogon, a triangulation is empty,
there is one boundary segment, and we agree that the associated
triangulation quiver is the one with one vertex shown in the top
row of Table~\ref{tab:quivers}.
\end{remark}

\begin{example}
Figure~\ref{fig:square} shows a triangulation of the square and the
corresponding triangulation quiver. There are four boundary segments and for
each $1 \leq i \leq 4$ the loop  $\delta_i$ corresponds to the boundary segment 
labeled $i$. The permutation $f$ on the arrows is given in cycle form by
$(\alpha_1 \, \alpha_2 \, \alpha_3)(\beta_1 \, \beta_2 \, \beta_3)
(\delta_1)(\delta_2)(\delta_3)(\delta_4)$.
\end{example}

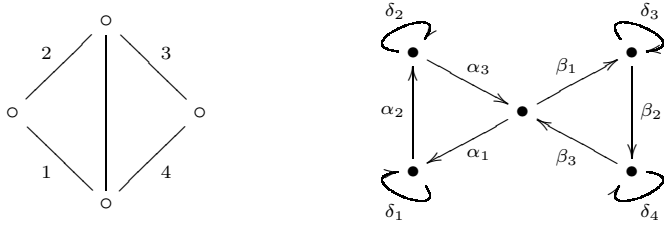
\begin{figure}
\begin{align*}
\begin{array}{c}
\xymatrix{
& {\circ} \ar@{-}[dr]^3 \ar@{-}[dd] \\
{\circ} \ar@{-}[ur]^2 && {\circ} \ar@{-}[dl]^4 \\
& {\circ} \ar@{-}[ul]^1
}
\end{array}
& &
\begin{array}{c}
\xymatrix@=1pc{
{\bullet} \ar@(l,ur)[]^{\delta_2} \ar[drr]^{\alpha_3}
&& && {\bullet} \ar@(ul,r)[]^{\delta_3} \ar[dd]^{\beta_2} \\
&& {\bullet} \ar[dll]^{\alpha_1} \ar[urr]^{\beta_1} \\
{\bullet} \ar@(dr,l)[]^{\delta_1} \ar[uu]^{\alpha_2}
&& && {\bullet} \ar@(r,dl)[]^{\delta_4} \ar[ull]^{\beta_3}
}
\end{array}
\end{align*}
\caption{A triangulation of a square (left) and the corresponding triangulation
quiver (right).}
\label{fig:square}
\end{figure}

\begin{remark} \label{rem:nvertex}
By using Euler characteristic considerations one sees that
if $(S,M)$ is not an unpunctured monogon, then
the number of vertices of the triangulation quiver associated to
any of its triangulations is
\begin{equation} \label{e:nvertex}
6(g-1) + 3(p+b) + 2(n_1 + n_2 + \dots + n_b) ,
\end{equation}
compare~\cite[Proposition~2.10]{FST08}.
\end{remark}

\begin{remark} \label{rem:triblocks}
In terms of the block decomposition of triangulation quivers described
in Section~\ref{ssec:blocks}, 
there is a natural block decomposition of $(Q_\tau, f_\tau)$ induced
by the triangulation $\tau$ with bijections
\begin{align*}
&\text{blocks of type A}
\longleftrightarrow \text{boundary segments}, \\
&\text{blocks of type B}
\longleftrightarrow \text{self-folded triangles in $\tau$}, \\
&\text{blocks of type C}
\longleftrightarrow \text{the other triangles in $\tau$}.
\end{align*}
In addition, there are also bijections
\begin{align*}
\text{cycles of $f_\tau$ of length $1$}
&\longleftrightarrow \text{boundary segments}, \\
\text{cycles of $f_\tau$ of length $3$}
&\longleftrightarrow \text{triangles in $\tau$}, \\
\text{cycles of $g_\tau$ of length $1$}
&\longleftrightarrow \text{self-folded triangles in $\tau$}, \\
\text{cycles of $g_\tau$}
&\longleftrightarrow \text{punctures and boundary components}.
\end{align*}
\end{remark}

We can also obtain the triangulation quiver via a ribbon graph naturally
associated to the triangulation. Informally speaking, one thinks of the
triangulation as the graph, but some modifications are needed at the
boundary components, as in the next definition.

\begin{definition} \label{def:trigraph}
Let $\tau$ be a triangulation of a marked surface $(S,M)$.
Associate to $\tau$ a ribbon graph defined as a graph $(V,E_\tau)$ with cyclic
ordering of the edges around each node as follows:
\begin{itemize}
\item
the set $V$ of nodes consists of the punctures in $M$ and the connected
components of $\partial S$,
\item
the set $E_\tau$ of edges consists of the arcs of $\tau$ and the boundary
segments.
\end{itemize}

Denote by $\pi \colon M \twoheadrightarrow V$ the map taking each puncture to
itself and each marked point on $\partial S$ to the boundary component it
belongs to. In the graph $(V,E_\tau)$, each edge is incident to the nodes
which are the images under $\pi$ of its endpoints.

The cyclic ordering is determined as follows.
If $v \in V$ is a puncture, then the edges incident to $v$ are arcs of $\tau$
and the cyclic ordering of them is the counterclockwise ordering induced by
the orientation of $S$.

If $v \in V$ is a boundary component, we arrange the set $\pi^{-1}(v)$ of
marked points on $v$ in a counterclockwise order $\{q_0,q_1,\dots,q_{n-1}\}$
such that for each $0 \leq i < n$ there is a boundary segment $\eps_i$ whose
endpoints are $q_i,q_{i+1}$ (where indices are taken modulo $n$).
The set of edges incident to $v$ thus consists of the boundary segments
$\eps_i$, which become loops in the graph (see Figure~\ref{fig:node}), and
the arcs incident to any of the marked points $q_i$. Their cyclic ordering is
obtained by taking the arcs incident to $q_0$ in the counterclockwise order
induced by the orientation of $S$, then $\eps_0$, then the arcs incident to
$q_1$ in a counterclockwise order, then $\eps_1$, etc.

\begin{figure}
\[
\xymatrix{
{\circ} \ar@{-}@(ur,ul)[] \ar@{-}@(ul,dl)[]
\ar@{-}@(dl,dr)[] \ar@{-}@(dr,ur)[]
}
\]
\caption{A boundary component with $4$ marked points becomes a node with $4$
loops. The arcs incident to the marked points are to be placed between these
loops.}
\label{fig:node}
\end{figure}
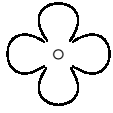
\end{definition}

The next statement is a consequence of the definitions.

\begin{proposition}
For any triangulation $\tau$ of a marked surface $(S,M)$,
the ribbon quiver corresponding under the bijection of
Proposition~\ref{p:ribbon} to the ribbon graph
constructed in Definition~\ref{def:trigraph} 
is the triangulation quiver $(Q_\tau, f_\tau)$.
\end{proposition}

\begin{example} \label{ex:surface3}
Table~\ref{tab:surface} lists the marked surfaces whose
triangulation quivers have at most three vertices. For each surface,
we list the corresponding triangulation quivers (and ribbon graphs)
appearing in Table~\ref{tab:quivers}.

Note that the unpunctured monogon and unpunctured triangle have only
empty triangulations, so for each of these surfaces there is only one quiver.
Similarly, a punctured monogon has only one triangulation, consisting of one
arc. A sphere with three punctures has two topologically inequivalent
triangulations and hence two triangulation quivers.
\end{example}

\subsection{Triangulation vs.\ adjacency quivers}
\label{ssec:triangadj}

The construction of the triangulation quiver of an ideal triangulation
resembles that of the adjacency quiver defined in~\cite[Definition~4.1]{FST08},
however there are several differences:
\begin{enumerate}
\renewcommand{\labelenumi}{\theenumi.}
\item
In the triangulation quiver there are vertices corresponding to the
boundary segments and not only to the arcs, as in the adjacency
quiver.

\item
Our treatment of self-folded triangles is different; in the triangulation
quiver there is a loop at each vertex corresponding to the inner side
of a self-folded triangle.

\item
We do not delete 2-cycles that arise in the quiver (e.g.\ when there
are precisely two arcs incident to a puncture).
\end{enumerate}

\begin{example}
Consider the triangulation of the square shown in Figure~\ref{fig:square}.
Its triangulation quiver consists of $5$ vertices whereas its adjacency
quiver is the Dynkin quiver $A_1$ (one vertex, no arrows).
\end{example}

\begin{table}
\begin{center}
\begin{tabular}[c]{cl}
\textbf{Quiver} & \textbf{Marked surface} \\
\hline
$1$ & monogon, unpunctured \\
$2$ & monogon, one puncture \\
$3a$, $3b$ & sphere, three punctures \\
$3'$ & triangle, unpunctured \\
$3''$ & torus, one puncture
\end{tabular}
\end{center}
\caption{The marked surfaces whose triangulation quivers have at most three
vertices. The numbers of the quivers refer to Table~\ref{tab:quivers}.}
\label{tab:surface}
\end{table}

As Example~\ref{ex:surface3} demonstrates,
these differences allow to attach triangulation quivers to 
marked surfaces that do not admit adjacency quivers, such as
a monogon, a triangle or a sphere with three punctures.
On the other hand, there are situations where
the triangulation quiver and the adjacency quiver of a triangulation
coincide. By abuse of notation, in the next statements by referring
to a triangulation quiver we actually mean the underlying quiver $Q$ of
the pair $(Q,f)$.
This is not ambiguous in view of Proposition~\ref{p:triauniq}

Recall that a surface $S$ is \emph{closed} if $\partial S$ is empty.
If $(S,M)$ is a marked surface and $S$ is closed, then all marked
points are punctures. The next statement is a reformulation of our result
in~\cite[\S2]{Ladkani12}.

\begin{lemma} \label{l:adjacency}
Let $(S,M)$ be a closed marked surface which is not a sphere with
less than four punctures.
Then for any triangulation $\tau$ of $(S,M)$ with at least three
arcs incident to each puncture, the triangulation quiver
and adjacency quiver associated to $\tau$ coincide.
\end{lemma}
The condition on $\tau$ in the lemma was called (T3) in~\cite{Ladkani12}.
In particular, we get the following corollary (cf.\
\cite[Lemma~5.3]{Ladkani12}).

\begin{corollary} \label{c:onep}
Let $(S,M)$ be a closed surface with exactly one puncture, i.e.\ $|M|=1$.
Then for any triangulation $\tau$ of $(S,M)$, the triangulation quiver
and the adjacency quiver associated to $\tau$ coincide.
\end{corollary}

\begin{example}
For a torus with empty boundary and one puncture, the adjacency quiver of
any triangulation is known as the Markov quiver and is given by the
quiver~$3''$ in the last row of Table~\ref{tab:quivers}.
\end{example}

Another difference between triangulation quivers and adjacency quivers
concerns the possibility to recover the topology of the underlying marked
surface. It is known~\cite[\S12]{FST08} that a quiver may arise 
as adjacency quiver of two triangulations of topologically inequivalent marked
surfaces. On the other hand, 
if $(Q_\tau, f_\tau)$ is the triangulation quiver
corresponding to a triangulation $\tau$ of a marked surface $(S,M)$,
then the topology of
$(S,M)$ can be completely recovered from $(Q_\tau, f_\tau)$.
Indeed, the cycles of the permutation $g$ on $(Q_\tau)_1$ are in bijection
with the punctures and boundary components of $(S,M)$.
For each such cycle $\omega$ set
$m_\omega = \left| \{ \alpha \in \omega : f(\alpha)=\alpha \}\right|$.
If $m_\omega = 0$, then $\omega$ corresponds to a puncture, otherwise
it corresponds to a boundary component with $m_\omega$ marked points on it.
In this way we recovered the parameters $p$, $b$ and the numbers
$n_1, \dots, n_b$. Once these are known, the genus of $S$ can be recovered
using Eq.~\eqref{e:nvertex}.

\subsection{A dimer model perspective}
\label{ssec:dimer}

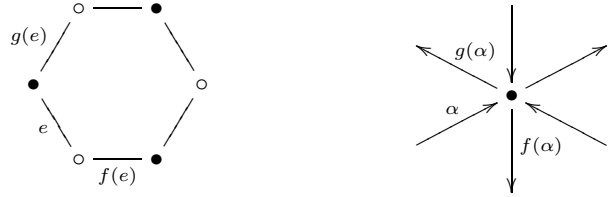
\begin{figure}
\begin{align*}
\begin{array}{c}
\xymatrix@=0.5pc{
& {\circ} \ar@{-}[rr] && {\bullet} \ar@{-}[ddr] \\ \\
{\bullet} \ar@{-}[uur]^{g(e)} &&&& {\circ} \ar@{-}[ddl] \\ \\
& {\circ} \ar@{-}[uul]^{e} && {\bullet} \ar@{-}[ll]^{f(e)}
}
\end{array}
&&
\begin{array}{c}
\xymatrix@=1pc{
&& \ar[dd] \\ 
&&&& {} \\
&& {\bullet} \ar[ull]_{g(\alpha)} \ar[urr] \ar[dd]^{f(\alpha)} \\
\ar[urr]^{\alpha} &&&& \ar[ull] \\
&& {} 
}
\end{array}
\end{align*}
\caption{A $2$-cell in a dimer model (left) and the corresponding vertex
with incident arrows (right).}
\label{fig:dimer}
\end{figure}

Dimer models on a torus have been used to construct non-commutative
crepant resolutions of toric Gorenstein
singularities~\cite{Bocklandt13,Broomhead12}.
Such resolution is a 3-Calabi-Yau algebra which is a (non-complete)
Jacobian algebra of a quiver with potential constructed from the dimer
model. In this section we explain how triangulations of closed surfaces
give rise to a very special kind of dimer models, yet the corresponding
(complete) Jacobian algebras (which are triangulation algebras to be
defined in Section~\ref{ssec:trialg}) have completely different properties,
as we shall see in Section~\ref{sec:triangquasi}.

A \emph{dimer model} on a closed, compact, connected, oriented surface $S$
is a bipartite graph
on $S$ whose complement is homeomorphic to a disjoint union of discs.
The set of nodes of this graph can thus be written as a disjoint union
$V^+ \cup V^-$. We call the elements of $V^+$ \emph{white nodes} and
those of $V^-$ \emph{black nodes}.
Denote by $E$ the set of edges. An edge $e \in E$ defines a pair
$(v^+_e,v^-_e) \in V^+ \times V^-$ consisting of the nodes incident to $e$.
Each connected component of the complement defines a $2$-cell, and an
edge is incident to exactly two $2$-cells. 

Define two permutations $f,g \colon E \to E$ on the set of edges as follows.
For an edge $e \in E$, let $f(e)$ be the edge following $e$ when going
clockwise around the node $v^+_e$ and let $g(e)$ be the edge following $e$
when going counterclockwise around the node $v^-_e$, see the left drawing
in Figure~\ref{fig:dimer}.

A dimer model gives rise to a quiver $Q$ by taking the graph dual to the
graph $(V^+ \cup V^-, E)$. The vertices of $Q$ are thus the $2$-cells,
and the arrows are in bijection with the edges. 
Let $\alpha$ be an arrow corresponding to an edge $e \in E$.
The endpoints of $\alpha$ are the two $2$-cells that $e$ is incident to,
and $\alpha$ is oriented in such a way that when going forward in the
direction of the arrow, the white node $v^+_e$ is seen to the right while
the black node $v^-_e$ is to the left, see Figure~\ref{fig:dimer}.

The permutations $f, g$ on $E$ induce permutations (denoted by the same
letters) on the set of arrows $Q_1$. For any vertex $i \in Q_0$, each of the
permutations $f$ and $g$ induces a bijection between the sets of
arrows starting at $i$ and those ending at $i$.

Now we restrict attention to dimer models whose $2$-cells are \emph{quadrilaterals}, i.e.\ consist of exactly four edges.
In this case, for each vertex $i$ of the quiver $Q$ there are exactly two 
arrows starting at $i$ and two arrows ending at $i$, and $(Q,f)$ thus becomes
a ribbon quiver.
Let us construct the corresponding ribbon graph $(V',E')$ explicitly in terms
of the dimer model. We have $V' = V^-$, that is, the nodes of the ribbon graph
are the black nodes of the dimer model, and the edges $E'$ of the ribbon graph
are in bijection with the quadrilaterals.
There are exactly two black nodes incident to each quadrilateral
and the corresponding edge $e'$ in the ribbon graph connects these nodes,
see Figure~\ref{fig:dimeribbon}.
The cyclic ordering around each node of $V'$ is induced by the embedding into
the oriented surface $S$.

\begin{figure}
\begin{align*}
\begin{array}{c}
\xymatrix@=0.5pc{
& {\circ} \ar@{-}[ddr] \\ \\
{\bullet} \ar@{-}[uur] && {\bullet} \ar@{-}[ddl] \\ \\
& {\circ} \ar@{-}[uul]
}
\end{array}
&&
\begin{array}{c}
\xymatrix@=0.5pc{
{\bullet} \ar@{-}[rr] && {\bullet}
}
\end{array}
\end{align*}
\caption{A quadrilateral in a dimer model (left) corresponds to an edge
in the ribbon graph (right).}
\label{fig:dimeribbon}
\end{figure}
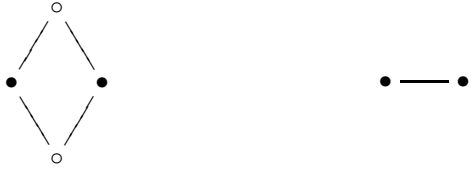

Finally we further restrict to dimer models whose $2$-cells are quadrilaterals
and moreover their white nodes are \emph{trivalent}, i.e.\ each $v^+ \in V^+$
is incident to exactly three edges. In this case the associated ribbon graph
$(V',E')$ is a triangulation of the marked surface $(S,V')$ and the
ribbon quiver $(Q,f)$ is a triangulation quiver.

Conversely, given a set $M$ of punctures, any triangulation $\tau$ of
$(S,M)$ without self-folded triangles gives rise to a dimer model
whose white nodes $V^+_\tau$ are the triangles of $\tau$, its black nodes
$V^-_\tau$ are the punctures $M$, and there is an edge connecting
$\Delta \in V^+_\tau$ with $v \in V^-_\tau$
if and only if $v$ is incident to $\Delta$ in $\tau$. The $2$-cells of
this dimer model are quadrilaterals (corresponding bijectively to the
arcs of $\tau$) and any $\Delta \in V^+_\tau$ is trivalent.
For example, Figure~\ref{fig:dimertetra} shows two dimer models; one for
the triangulation of a sphere with three punctures corresponding to the
triangulation quiver $3b$ of Table~\ref{tab:quivers}; and the other for
the tetrahedron of Figure~\ref{fig:selfdual},
which is a triangulation of a sphere with four punctures.

\begin{figure}
\begin{align*}
\begin{array}{c}
\xymatrix{
& {\circ} \ar@{-}[dl] \ar@{-}[d] \ar@{-}[dr] \\
{\bullet} & {\bullet} & {\bullet} \\
& {\circ} \ar@{-}[ul] \ar@{-}[u] \ar@{-}[ur]
}
\end{array}
&&
\begin{array}{c}
\xymatrix@=1pc{
&& {\bullet} \ar@{-}[d] \ar@{-}[ddrr] \ar@{-}[ddll] \\
&& {\circ} \ar@{-}[dr] \ar@{-}[dl] \\
{\circ} \ar@{-}[r] & {\bullet} && {\bullet} \ar@{-}[r] & {\circ} \\
&& {\circ} \ar@{-}[ur] \ar@{-}[ul] \\
&& {\bullet} \ar@{-}[u] \ar@{-}[uurr] \ar@{-}[uull]
}
\end{array}
\end{align*}
\caption{Dimer models on a sphere corresponding to the triangulation quiver
$3b$ of Table~\ref{tab:quivers} (left) and that of the tetrahedron (right).
Since they arise from triangulations,
all $2$-cells are quadrilaterals and all white nodes are trivalent.}
\label{fig:dimertetra}
\end{figure}
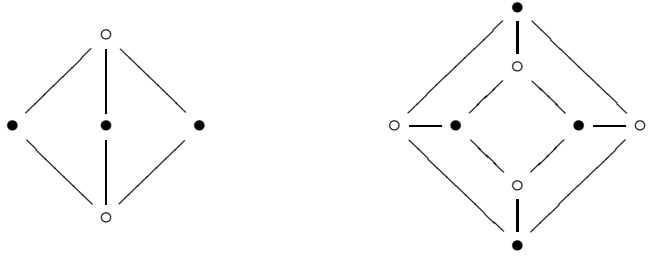

The various notions concerning dimer models, ribbon graphs, ribbon quivers and
triangulations are related as in the dictionary given in Table~\ref{tab:dimer}.

\begin{table}
\begin{center}
\begin{tabular}{cccc}
Dimer model & Ribbon graph & Ribbon quiver & Triangulation \\ \hline
$V^+$ &      & cycles of $f$ & triangles \\
$V^-$ & $V'$ & cycles of $g$ & punctures \\
$2$-cells & $E'$ & $Q_0$ & arcs \\
$E$   &      & $Q_1$
\end{tabular}
\end{center}
\caption{Dictionary between dimer models, ribbon graphs, ribbon quivers
and triangulations of a closed surface.
A ribbon graph/quiver arises when all the $2$-cells in the dimer model are
quadrilaterals, and a triangulation arises when, in addition, all the white
nodes $V^+$ are trivalent.}
\label{tab:dimer}
\end{table}

\section{Brauer graph algebras and triangulation algebras}
\label{sec:algebras}

In this section we introduce two classes of algebras which turn out to be
important for our study, one consists of the well known Brauer graph algebras
and the other consists of the newly defined triangulation algebras.
Roughly speaking, a Brauer graph algebra arises from any ribbon quiver
and auxiliary data given in the form of scalars and positive integer
multiplicities, whereas a triangulation algebra arises from any triangulation
quiver with similar auxiliary data.

\begin{definition}
Let $(Q,f)$ be a ribbon quiver. 
Recall from Section~\ref{sec:quivers} the permutation $g \colon Q_1 \to Q_1$
defined by $g(\alpha) = \overline{f(\alpha)}$ for any $\alpha \in Q_1$.
Given a function $\nu$ on the set $Q_1$ of arrows, we write
$\nu_\alpha$ instead of $\nu(\alpha)$.
We say that $\nu$ is \emph{$g$-invariant} if $\nu_{g(\alpha)} = \nu_\alpha$
for any $\alpha \in Q_1$. Similarly, we say that $\nu$ is
\emph{$f$-invariant} if $\nu_{f(\alpha)} = \nu_\alpha$ for any
$\alpha \in Q_1$.

Since the $g$-cycles are in bijection with the nodes of the
corresponding ribbon graph, 
a $g$-invariant function can thus be regarded as a function on the
nodes of that ribbon graph.
\end{definition}

Let $(Q,f)$ be a ribbon quiver. For an arrow $\alpha \in Q_1$, set
\begin{align*}
n_\alpha &= \min \{n>0 \,:\, g^n(\alpha)=\alpha\} \\
\oa &= \alpha \cdot g(\alpha) \cdot \ldots \cdot g^{n_\alpha-1}(\alpha) \\
\oa' &= \alpha \cdot g(\alpha) \cdot \ldots \cdot g^{n_\alpha-2}(\alpha)
\end{align*}

The function $\alpha \mapsto n_\alpha$ is obviously $g$-invariant, telling the
length of the $g$-cycle $\oa$ starting at $\alpha$. The path $\oa'$ is
``almost'' a cycle; when $n_\alpha=1$ the arrow $\alpha$ is a loop at some vertex
$i$ and $\oa'$ is understood to be the path of length zero starting at $i$.
Similarly, for an arrow $\alpha \in Q_1$, set
\begin{align*}
k_\alpha &= \min \{k>0 \,:\, f^k(\alpha)=\alpha\} \\
\xi_\alpha &=
\alpha \cdot f(\alpha) \cdot \ldots \cdot f^{k_\alpha-1}(\alpha) \\
\xi'_\alpha &=
\alpha \cdot f(\alpha) \cdot \ldots \cdot f^{k_\alpha-2}(\alpha)
\end{align*}

The function $\alpha \mapsto k_\alpha$ is obviously $f$-invariant, telling the
length of the $f$-cycle $\xi_\alpha$ starting at $\alpha$. The path
$\xi'_\alpha$ is ``almost'' a cycle; when $k_\alpha=1$ the arrow $\alpha$ is a
loop at some vertex $i$ and $\xi'_\alpha$ is understood to be the path of
length zero starting at $i$.

\begin{lemma}
For any $\alpha \in Q_1$, the paths $\oa'$ and $\xi'_{\balpha}$ are parallel,
i.e.\ they both start at the same vertex and end at the same vertex.
\end{lemma}

\subsection{Brauer graph algebras}

In this section we fix a field $K$.
Brauer graph algebras form a generalization of Brauer tree algebras.
They are algebras defined from combinatorial data consisting of
a ribbon graph together with multiplicities and scalars associated to its
nodes, see~\cite{Alperin86,Benson98,Kauer98}. Many authors start
with the ribbon graph and construct the quiver with relations of the
corresponding Brauer graph algebra, see for
example~\cite{GSS14,MarshSchroll14}.
We prefer to give the definition directly in terms of the associated ribbon
quiver.

\begin{definition} \label{def:BGA}
Let $(Q,f)$ be a ribbon quiver, and let $m \colon Q_1 \to \bZ_{>0}$
and $c \colon Q_1 \to K^{\times}$ be $g$-invariant functions of
\emph{multiplicities} and \emph{scalars}, respectively.
The \emph{Brauer graph algebra} 
$\Gamma(Q,f,m,c)$
associated to these data is the
quotient of the path algebra $KQ$ by the ideal generated by two
types of elements; the elements of the first type are the paths
$\alpha \cdot f(\alpha)$ for each $\alpha \in Q_1$ (``zero-relations'')
and the elements of the second type are the differences
$c_\alpha \oa^{m_\alpha} - c_{\balpha} \oba^{m_{\balpha}}$
(``commutativity-relations'').
In other words, 
\[
\Gamma(Q,f,m,c) = KQ /
( \alpha \cdot f(\alpha) \,,\,
c_\alpha \oa^{m_\alpha} - c_{\balpha} \oba^{m_{\balpha}}
)_{\alpha \in Q_1}.
\]
(It is clearly enough to take one commutativity-relation for each
pair of arrows $\alpha$ and $\balpha$, so these relations can be
seen as indexed by the vertices of $Q$).
\end{definition}

The next proposition is well known. Special biserial algebras have been defined
in~\cite[\S1]{SW83} and a classification of the indecomposable modules over
these algebras, implying that they are of tame representation type, is given
in~\cite[\S2]{WW85}.

\begin{proposition} \label{p:BGA}
A Brauer graph algebra is finite-dimensional, symmetric, special biserial and
hence of tame representation type.
\end{proposition}

Moreover, it has been recently shown that over an algebraically closed field
the classes of symmetric special biserial algebras and that of Brauer graph
algebras coincide~\cite[Theorem~1.1]{Schroll15}.

We present a few examples of Brauer graph algebras related to group algebras.

\begin{example} \label{ex:BGAf1}
Consider the ribbon quiver $(Q,f_1)$ of Example~\ref{ex:ribbonG1}.
In this case there is only one $g$-cycle and hence the auxiliary data
consists of one multiplicity $m \geq 1$ and one scalar $c \in K^{\times}$.
The path algebra $KQ$ is the free algebra $K\langle \alpha, \beta \rangle$
on the generators $\alpha$ and $\beta$, and
the Brauer graph algebra is
\[
K\langle \alpha, \beta \rangle /
( \alpha^2 , \beta^2 , c(\alpha \beta)^m - c(\beta \alpha)^m ) ,
\]
hence, up to isomorphism, we may set $c=1$. 
When $m=1$ and $\ch K = 2$, this algebra is isomorphic the group algebra of
Klein's four-group.
\end{example}

\begin{example} \label{ex:BGAf2}
Consider now the ribbon quiver $(Q,f_2)$ of Example~\ref{ex:ribbonG1}.
In this case there are two $g$-cycles and hence the auxiliary data
consists of two multiplicities $m, m' \geq 1$ and two scalars
$c, c' \in K^{\times}$.
The Brauer graph algebra is given by
\[
K\langle \alpha, \beta \rangle /
(\alpha \beta, \beta \alpha, c \alpha^m - c' \beta^{m'}) .
\]
When $m' = 1$, the arrow $\beta$ can be eliminated so the relations in the
above presentations are no longer minimal and the algebra becomes isomorphic
to the algebra $K[\alpha]/(\alpha^{m+1})$ considered in Example~\ref{ex:Kxn}.
\end{example}

\begin{example}
Let's describe as group algebras some Brauer graph algebras for a few
triangulation quivers appearing in Table~\ref{tab:quivers},
under the assumption that $\ch K = 2$.

The Brauer graph algebra of the triangulation quiver number $1$ with
multiplicity~$1$ was discussed in Example~\ref{ex:BGAf1}; it is the group
algebra of Klein's four group.

Assume now that $K$ contains a primitive third root of unity.
Then the Brauer graph algebra of the
triangulation quiver number $2$ with multiplicities $m_\alpha=2$ and
$m_\beta=m_\gamma=m_\eta=1$ is Morita equivalent to the group algebra
of the symmetric group $S_4$ \cite[V.2.5.1]{Erdmann90},
whereas that of the triangulation quiver number $3b$ with all multiplicities
set to $1$ is isomorphic to the group algebra of the alternating group
$A_4$ \cite[V.2.4.1]{Erdmann90}
(in both cases the scalars take the constant value~$1$).
\end{example}

We list a few remarks concerning Brauer graph algebras.

\begin{remark} \label{rem:BGAelim}
As Example~\ref{ex:BGAf2} shows, if $\alpha$ is an arrow such that $n_\alpha=1$
and $m_\alpha=1$, the corresponding commutativity-relation becomes
$c_\alpha \alpha - c_{\balpha} \oba^{m_{\balpha}}$ and in the
presentation of the Brauer graph algebra as quiver with relations the
arrow $\alpha$ can be eliminated at the expense of adding zero-relations
of a third kind, namely $\omega_\beta^{m_\beta} \beta$
for $\beta \in \{\balpha, g^{-1}(\balpha)\}$.
However, in order to keep the presentation unified, we will not eliminate
arrows and add the corresponding new relations.
\end{remark}

\begin{remark}
If $K$ is algebraically closed, or more generally, if $K$ contains an
$m_\alpha$-th root of $c_\alpha$ for each $\alpha \in Q_1$, then by
considering the automorphism of $KQ$ defined by choosing from each $g$-cycle
one arrow $\alpha$, sending it to $c_\alpha^{1/m_\alpha} \alpha$ and keeping
all other arrows intact, we see that
$\Gamma(Q,f,m,c) \simeq \Gamma(Q,f,m,\mathbf{1})$ where $\mathbf{1}$
is the constant function $\mathbf{1}_\alpha = 1$.
\end{remark}

In the next statements we explicitly compute the Cartan matrix of a Brauer
graph algebra and show that it depends only on the multiplicities and
the underlying graph of the ribbon graph corresponding to its defining ribbon
quiver.

Throughout, we fix a ribbon quiver $(Q,f)$ together with $g$-invariant
functions $m \colon Q_1 \to \bZ_{>0}$ and $c \colon Q_1 \to K^{\times}$
of multiplicities and scalars, and consider the Brauer graph algebra
$\Gamma=\Gamma(Q,f,m,c)$.
For any $i \in Q_0$, let $\alpha, \balpha$ be the two arrows starting at $i$.
By definition, the images of the paths $c_\alpha \oa^{m_\alpha}$ and
$c_{\balpha} \omega_{\balpha}^{m_{\balpha}}$ in $\Gamma$ are equal, and we
denote their common value by $z_i \in \Gamma$.
The next statement is a consequence of the definition.
\begin{lemma} \label{l:basisBGA}
A basis of $\Gamma(Q,f,m,c)$ is given by the images of the paths
\[
\{e_i\}_{i \in Q_0} \cup
\left\{ \alpha \cdot g(\alpha) \cdot \ldots \cdot g^r(\alpha)
\right\}_{\alpha \in Q_1, 0 \leq r < m_\alpha n_\alpha -1} \cup
\{z_i\}_{i \in Q_0} .
\]
\end{lemma}

Given the basis of Lemma~\ref{l:basisBGA}, an argument as
in~\cite[\S4.4]{Ladkani12} allows to compute the Cartan matrix and to draw some
conclusions.
For a $g$-cycle $\omega$ in $Q_1$, define a row vector
$\chi_\omega \in \bZ^{Q_0}$ by
$\chi_\omega(i) = |\{\alpha \in \omega : s(\alpha)=i\}|$ for $i \in Q_0$.
Denote by $\Omega_g$ the set of $g$-cycles in $Q_1$.
Recall from Section~\ref{ssec:ribbon} that the matrix 
$(\chi_\omega(i))_{\omega \in \Omega_g, i \in Q_0}$ encodes the underlying
graph of the ribbon graph corresponding to $(Q,f)$ and hence depends only on
that graph.
For any $\omega \in \Omega_g$, the square matrix $\chi_\omega^T \chi_\omega$ is
symmetric of rank $1$ whose $(i,j)$-entry is $\chi_\omega(i) \chi_\omega(j)$
for $i,j \in Q_0$. The $g$-invariant function $m$ on $Q_1$ induces a
function on $\Omega_g$ which will be denoted by the same letter.

\begin{proposition} \label{p:CartanBGA}
Let $C_\Gamma$ be the Cartan matrix of $\Gamma(Q,f,m,c)$. Then:
\begin{enumerate}
\renewcommand{\theenumi}{\alph{enumi}}
\item
$C_\Gamma = \sum_{\omega \in \Omega_g} m_\omega \chi_\omega^T \chi_\omega$.

\item
$\dim_K \Gamma = \sum_{\omega \in \Omega_g} m_\omega |\omega|^2$.

\item
The quadratic form $q_{C_\Gamma} \colon \bZ^{Q_0} \to \bZ$ defined by
$q_{C_\Gamma}(x) = x C_{\Gamma} x^T$ takes non-negative even values;
in particular it is non-negative definite.

\item
$\rank C_\Gamma \leq \min(|Q_0|, |\Omega_g|)$.
\end{enumerate}
\end{proposition}

The next remark will not be used in the sequel. Nevertheless, we list it here
for completeness.

\begin{definition}
Given a non-zero power series
$p(x) = \sum_{i=0}^{\infty} a_i x^i \in K[[x]] \setminus \{0\}$,
let $m = \min \{i \geq 0 : a_i \neq 0\}$. The \emph{least order term} of $p(x)$
is $c_m x^m$.
\end{definition}

For any non-trivial cycle $\omega$ in $Q$, the evaluation map
$\ev_\omega \colon K[[x]] \to \wh{KQ}$ sending $p \in K[[x]]$ to $p(\omega)$
is a continuous ring homomorphism.

\begin{remark} \label{rem:genBGA}
In analogy with the triangulation algebras to be defined in the next section
as quotients of complete path algebras by closed ideals,
one could consider an apparently more general, continuous, version of
a Brauer graph algebra defined by taking a $g$-invariant function
$p \colon Q_1 \to xK[[x]] \setminus \{0\}$ of non-zero power series without
constant term and forming the quotient of $\wh{KQ}$ by the closure of the
ideal generated by zero-relations and commutativity-relations
\[
\wh{KQ} /
\overline{(\alpha \cdot f(\alpha) \,,\,
p_\alpha(\oa) - p_{\balpha}(\oba))}_{\alpha \in Q_1}.
\]

However, it turns out that this algebra is isomorphic to the Brauer graph
algebra $\Gamma(Q,f,m,c)$ where the multiplicities $m_\alpha$ and scalars
$c_\alpha$ are such that each $c_\alpha x^{m_\alpha}$ is the least order
term of $p_\alpha(x)$.
\end{remark}

\subsection{Triangulation algebras}
\label{ssec:trialg}

In this section we define, for any triangulation quiver together with some
auxiliary data, a new algebra called triangulation algebra.
Throughout, we fix a field $K$.
We start with a construction of hyperpotentials on ribbon quivers described
in the following somewhat technical statement whose proof is omitted.

\begin{proposition} \label{p:hyperib}
Let $(Q,f)$ be a ribbon quiver.
\begin{enumerate}
\renewcommand{\theenumi}{\alph{enumi}}
\item \label{it:hyperib}
Let $p \colon Q_1 \to K[[x]]$ be $f$-invariant and
let $q \colon Q_1 \to K[[x]]$ be $g$-invariant.
Then the collection $(\rho_\alpha)_{\alpha \in Q_1}$ given by
\[
\rho_\alpha = 
p_\alpha(\xi_{f(\alpha)}) \cdot \xi'_{f(\alpha)} -
q_{\alpha}(\omega_{g(\alpha)}) \cdot \omega'_{g(\alpha)}
\]
is a hyperpotential on $Q$.

\item \label{it:potrib}
Let $P \colon Q_1 \to K[[x]]$ be $f$-invariant and
let $R \colon Q_1 \to K[[x]]$ be $g$-invariant.
Consider
\[
W = \sum_{\alpha} P_\alpha(\xi_\alpha) - \sum_{\beta} R_\beta(\omega_\beta)
\]
where the left sum runs over representatives $\alpha$ of the $f$-cycles
in $Q$ and the right sum runs over representatives $\beta$ of the
$g$-cycles. Then $W$ is a potential on $Q$ and
$\partial_\alpha W = \rho_\alpha$ for each $\alpha \in Q_1$, where
$\rho_\alpha$ are defined as in part~\eqref{it:hyperib} for the 
functions $p,q \colon Q_1 \to K[[x]]$ given by
$p_\alpha(x) = P'_\alpha(x)$ and $q_\alpha(x) = R'_\alpha(x)$.
\end{enumerate}
\end{proposition}

\begin{remark}
If $p_\alpha(x)$ and $q_\alpha(x)$ are monomials, that is,
there exist $f$-invariant function $\ell \colon Q_1 \to \bZ_{>0}$
and $g$-invariant function $m \colon Q_1 \to \bZ_{>0}$ such that
$p_\alpha(x)=x^{\ell_\alpha-1}$ and $q_\alpha(x) = x^{m_\alpha-1}$
for any $\alpha \in Q_1$, then the (non-complete) Jacobian algebra of the
hyperpotential in Proposition~\ref{p:hyperib}\eqref{it:hyperib} is the one
associated by Bocklandt to a weighted quiver polyherdon~\cite{Bocklandt13}.
In particular, if $P_\alpha(x)=Q_\alpha(x)=x$ for any $\alpha \in Q_1$,
then the potential in Proposition~\ref{p:hyperib}\eqref{it:potrib} is the
potential arising from the dimer model corresponding to $(Q,f)$, see
Section~\ref{ssec:dimer}.
\end{remark}

We are now ready to define what a triangulation algebra is.

\begin{definition} \label{def:triang}
Let $(Q,f)$ be a triangulation quiver. Let $m \colon Q_1 \to \bZ_{>0}$
and $c \colon Q_1 \to K^{\times}$ be $g$-invariant functions of
\emph{multiplicities} and \emph{scalars}, respectively, and assume
that $m_\alpha n_\alpha \geq 2$ for any $\alpha \in Q_1$.
Let $\lambda \colon Q_1^f \to K$, i.e.\ $\lambda$ is an assignment of a scalar
$\lambda_\alpha \in K$ for each $\alpha \in Q_1$ such that
$f(\alpha)=\alpha$.

The \emph{triangulation algebra} $\gL(Q,f,m,c,\lambda)$
associated to these data is the quotient
$\gL(Q,f,m,c,\lambda) = \wh{KQ}/\overline{\cJ}$
of the complete path algebra $\wh{KQ}$ by the closure of the ideal
$\cJ$ generated by the commutativity-relations
\begin{equation} \label{e:triangJ}
\begin{split}
\cJ = \Bigl(&
\left\{\balpha \cdot f(\balpha) - c_\alpha \oa^{m_\alpha-1} \cdot \oa'
\right\}_{\alpha \in Q_1 \,:\, f(\balpha) \neq \balpha},  \\
& \left\{\balpha^2 - \lambda_{\balpha} \balpha^3
- c_\alpha \oa^{m_\alpha-1} \cdot \oa'
\right\}_{\alpha \in Q_1 \,:\, f(\balpha) = \balpha}
\Bigr)
\end{split}
\end{equation}
(when the set $Q_1^f$ is empty then evidently $\lambda$ does not play any role
in the definition).
\end{definition}

The data defining a triangulation algebra can be used to define an
$f$-invariant function $p \colon Q_1 \to K[[x]]$ and a $g$-invariant function
$q \colon Q_1 \to K[[x]]$ as follows; set $p_\alpha(x)=x^2-\lambda_\alpha x^3$
if $\alpha \in Q_1^f$ and $p_\alpha(x)=1$ otherwise. Similarly, set
$q_\alpha(x) = c_\alpha x^{m_\alpha-1}$ for any $\alpha \in Q_1$.
We observe that the commutativity-relation in~\eqref{e:triangJ} corresponding
to an arrow $\alpha \in Q_1$ equals the element $\rho_{g^{-1}(\alpha)}$ of the
hyperpotential considered in Proposition~\ref{p:hyperib} arising from the
functions $p$ and $q$. This yields the following basic property of
triangulation algebras.

\begin{proposition} \label{p:potential}
Let $(Q,f)$ be a triangulation quiver. Let $m \colon Q_1 \to \bZ_{>0}$
and $c \colon Q_1 \to K^{\times}$ be $g$-invariant functions of
multiplicities and scalars, respectively and let $\lambda \colon Q_1^f \to K$.
\begin{enumerate}
\renewcommand{\theenumi}{\alph{enumi}}
\item
The triangulation algebra $\gL(Q,f,m,c,\lambda)$ is always a Jacobian algebra
of a hyperpotential on $Q$.

\item
Let $\mu_g = \lcm(\{m_\alpha\}_{\alpha \in Q_1})$ and let
\[
\mu_f = 
\begin{cases}
6 & \text{if $Q_1^f$ is non-empty and $\lambda_\alpha \neq 0$ for some
$\alpha \in Q_1^f$,} \\
3 & \text{if $Q_1^f$ is non-empty and $\lambda_\alpha = 0$ for any
$\alpha \in Q_1^f$,} \\
1 & \text{if $Q_1^f$ is empty.}
\end{cases}
\]
If $\ch K$ does not divide $\mu_f \mu_g$, then  $\gL(Q,f,m,c,\lambda)$
is a Jacobian algebra of a potential on $Q$.
\end{enumerate}
\end{proposition}

\sloppy 
Additional properties of triangulation algebras will be presented in
Section~\ref{sec:triangquasi}. Since these algebras are given as quotients by
closure of ideals generated by commutativity-relations, \emph{a-priori} it is
not even clear from the outset if they are finite-dimensional or not.
However, it turns out that under some mild conditions on the auxiliary data,
this is indeed the case as the closure $\overline{\cJ}$ contains sufficiently
many paths (that is, zero-relations), see Section~\ref{ssec:findim}.

\fussy 
It turns out that the concept of triangulation algebra is versatile enough to
capture two seemingly unrelated classes of algebras occurring in the literature.
Indeed,
\begin{itemize}
\item
many algebras of quaternion type are in fact triangulation algebras
(see Section~\ref{ssec:quat}); and

\item
for many triangulations of closed surfaces with punctures, the Jacobian
algebras of the quivers with potentials associated by
Labardini-Fragoso~\cite{Labardini09} are triangulation algebras (Section~\ref{ssec:Jaclosed}).
\end{itemize}

Let us quickly discuss the triangulation algebras on the triangulation quivers
with small number of vertices shown in Table~\ref{tab:quivers}
and refer to the relevant statements in the sequel. Under some mild
conditions on the auxiliary data, the triangulation algebras on the quivers
$1$, $2$, $3a$ and $3b$ are algebras of quaternion type
(Remark~\ref{rem:quat}); triangulation algebras on the quiver $1$ are
further discussed in Section~\ref{ssec:triang1}, whereas those on the quivers
$2$ and $3b$ are considered in Lemma~\ref{l:quatriang}.
A triangulation algebra on the quiver $3''$ with all multiplicities set to
$1$ coincides with the Jacobian algebra of the quiver with potential
associated with a triangulation of a torus with one puncture
(a special case of Proposition~\ref{p:QPtriang}); this algebra
has been considered in~\cite[Example~8.2]{Labardini09b}
and~\cite[Example~4.3]{Plamondon13}.

In order to complete the picture and also to provide some concrete examples,
the triangulation algebras on the quivers $3a$ and $3'$ are given in the next
two examples.

\begin{example} \label{ex:triang3a}
Let $(Q,f)$ be the triangulation quiver $3a$ of Table~\ref{tab:quivers} shown
in the picture below
\[
\xymatrix{
{\bullet_1} \ar@(ul,dl)[]_{\alpha} \ar@<-0.5ex>[r]_{\beta}
& {\bullet_2} \ar@<-0.5ex>[r]_{\delta} \ar@<-0.5ex>[l]_{\gamma}
& {\bullet_3} \ar@(dr,ur)[]_{\xi} \ar@<-0.5ex>[l]_{\eta}
}
\]
with $f=(\alpha \beta \gamma)(\xi \eta \delta)$.
Then $g=(\alpha)(\beta \delta \eta \gamma)(\xi)$ has three cycles and
any $g$-invariant function $\nu$ on $Q_1$ satisfies
$\nu_\beta = \nu_\gamma = \nu_\delta = \nu_\eta$, hence it depends on
three values which by abuse of notation will be denoted by $\nu_1, \nu_2, \nu_3$
where $\nu_1=\nu_\alpha$, $\nu_2 = \nu_\beta$ and $\nu_3=\nu_\xi$.

The auxiliary data needed to define a triangulation algebra on $(Q,f)$
thus consists of three positive integer multiplicities $m_1, m_2, m_3$
satisfying $m_1, m_3 \geq 2$ and three scalars $c_1, c_2, c_3 \in K^{\times}$
(the function $\lambda$ has empty domain and hence can be ignored).
The triangulation algebra  $\gL(Q,f,m,c,\lambda)$ is the quotient of the
complete path algebra of $Q$ by the closure of the ideal generated by the six
elements
\begin{align*}
&\beta \gamma - c_1 \alpha^{m_1-1}, &&
\alpha \beta - c_2 (\beta \delta \eta \gamma)^{m_2-1} \beta \delta \eta, &&
\delta \xi - c_2 (\gamma \beta \delta \eta)^{m_2-1} \gamma \beta \delta, \\
&\eta \delta - c_3 \xi^{m_3-1}, &&
\xi \eta - c_2 (\eta \gamma \beta \delta)^{m_2-1} \eta \gamma \beta, &&
\gamma \alpha - c_2 (\delta \eta \gamma \beta)^{m_2-1} \delta \eta \gamma.
\end{align*}
\end{example}

\begin{example}
Let $(Q,f)$ be the triangulation quiver $3'$ of Table~\ref{tab:quivers} shown
in the picture below
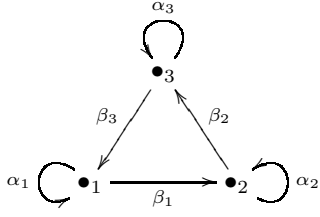
\[
\xymatrix@=1pc{
& {\bullet_3} \ar@(ur,ul)[]_{\alpha_3} \ar[ddl]_{\beta_3} \\ \\
{\bullet_1} \ar@(ul,dl)[]_{\alpha_1} \ar[rr]_{\beta_1}
&& {\bullet_2} \ar@(dr,ur)[]_{\alpha_2} \ar[uul]_{\beta_2}
}
\]
with $f=(\alpha_1)(\alpha_2)(\alpha_3)(\beta_1 \beta_2 \beta_3)$.
Then $g=(\alpha_1 \beta_1 \alpha_2 \beta_2 \alpha_3 \beta_3)$ has one cycle,
the set $Q_1^f$ consists of the arrows $\alpha_1, \alpha_2, \alpha_3$ and the
auxiliary data needed to define a triangulation algebra on $(Q,f)$ consists of
a multiplicity $m \geq 1$, one scalar $c \in K^{\times}$ and three scalars
$\lambda_1, \lambda_2, \lambda_3 \in K$.

The triangulation algebra  $\gL(Q,f,m,c,\lambda)$ is the quotient of the
complete path algebra of $Q$ by the closure of the ideal generated by the six
elements
\begin{align*}
&\beta_i \beta_{i+1} -
c (\alpha_i \beta_i \alpha_{i+1} \beta_{i+1} \alpha_{i-1} \beta_{i-1})^{m-1}
\alpha_i \beta_i \alpha_{i+1} \beta_{i+1} \alpha_{i-1} 
&& (1 \leq i \leq 3) \\
&\alpha_i^2 - \lambda_i \alpha_i^3 - 
c (\beta_i \alpha_{i+1} \beta_{i+1} \alpha_{i-1} \beta_{i-1} \alpha_i)^{m-1}
\beta_i \alpha_{i+1} \beta_{i+1} \alpha_{i-1} \beta_{i-1}
&& (1 \leq i \leq 3)
\end{align*}
where index arithmetic is taken modulo $3$ (i.e.\ $3+1=1$ and $1-1=3$).
\end{example}

\begin{definition}
We say that a $g$-invariant multiplicity function
$m \colon Q_1 \to \bZ_{>0}$ is \emph{admissible}
if $m_\alpha n_\alpha \geq 3$ for any arrow $\alpha \in Q_1$.
\end{definition}

One needs to check the condition in the definition only for the arrows $\alpha$
with $n_\alpha \leq 2$. In particular, these arrows occur as loops or as
part of $2$-cycles. The admissibility condition thus reads as follows:
if $n_\alpha=1$ then $m_\alpha \geq 3$, while
if $n_\alpha=2$ then $m_\alpha \geq 2$.
Note that when the pair $(n_\alpha,m_\alpha)$ equals $(1,2)$ or $(2,1)$ the
triangulation algebra is defined but the multiplicity is not admissible,
and when it equals $(1,1)$ the triangulation algebra is not even defined.

\begin{example}
For the triangulation quiver $3a$ of Table~\ref{tab:quivers} considered in
Example~\ref{ex:triang3a}, one has $n_\alpha = n_\xi = 1$ and $n_\beta = 4$.
Hence the multiplicity function $m$ is admissible if and only if
$m_\alpha \geq 3$ and $m_\xi \geq 3$.
\end{example}

We conclude this section by a series of remarks concerning the definition of
triangulation algebras and possible extensions thereof. The reader might skip
these remarks on first reading.

\begin{remark}
Since the path $\oa^{m_\alpha-1} \cdot \oa'$ is of length
$m_\alpha n_\alpha - 1$, the definition of a triangulation algebra makes
perfect sense when $m_\alpha n_\alpha=2$ for some arrow $\alpha$, but in this
case the right hand side of the corresponding commutativity-relation is just
$c_\alpha \alpha$, so the arrow $\alpha$ could be eliminated from $Q$
complicating somewhat the remaining relations.
The admissibility condition ensures that the generating relations lie
in the square of the ideal generated by all arrows of $Q$ so that no arrows
have to be deleted, compare with Remark~\ref{rem:BGAelim} for Brauer graph
algebras.
\end{remark}

\begin{remark}
When $\ch K \neq 2$, the scalars $\lambda_{\alpha}$ occurring in the
definition of a triangulation algebra do not play any role, 
i.e.\ $\gL(Q,f,m,c,\lambda) \simeq \gL(Q,f,m,c,\mathbf{0})$.
This can be shown by considering the automorphism of
$\wh{KQ}$ sending each arrow $\alpha$ with $f(\alpha)=\alpha$ 
to $\alpha - (\lambda_\alpha/2) \alpha^2$ and keeping the other arrows
unchanged. For the proof ones needs to know the additional zero relations
that hold in a triangulation algebra given in Proposition~\ref{p:qrel}.
\end{remark}

\begin{remark}
Even if $\lambda_\alpha=0$ for all the arrows $\alpha \in Q_1^f$,
there may be different $g$-invariant functions of scalars
$c,c' \colon Q_1 \to K^{\times}$ yielding isomorphic triangulation algebras,
that is,
\[
\gL(Q,f,m,c,\mathbf{0}) \simeq \gL(Q,f,m,c',\mathbf{0}),
\]
but in this survey we will not pursue a systematic study of this equivalence
relation on $g$-invariant functions of scalars.
\end{remark}

\begin{remark} \label{rem:trianga}
It is possible to slightly generalize Definition~\ref{def:triang} by
considering
also an $f$-invariant function $a \colon Q_1 \to K^{\times}$ of scalars
and setting
\begin{align*}
\cJ = \Bigl(&
\left\{a_{\balpha} \balpha \cdot f(\balpha) - c_\alpha \oa^{m_\alpha-1} \cdot \oa'
\right\}_{\alpha : f(\balpha) \neq \balpha},  \\
& \left\{a_{\balpha} \balpha^2 - \lambda_{\balpha} \balpha^3
- c_\alpha \oa^{m_\alpha-1} \cdot \oa'
\right\}_{\alpha : f(\balpha) = \balpha}
\Bigr) 
\end{align*}
(the current definition uses the constant function $a=\mathbf{1}$).

All the results of Section~\ref{sec:triangquasi} are valid also in this
more general setting, but for simplicity, we chose to present the material
without these extra scalars, since in many cases this apparent generalization
does not yield any new algebras.
Indeed, by using scalar transformation of the arrows
and replacing the scalar function $c$ by another $g$-invariant function
$c' \colon Q_1 \to K^{\times}$, we may always assume
that $a_{\alpha}=1$ for any arrow with $f(\alpha) \neq \alpha$, and if
$\alpha$ is an arrow such that $f(\alpha) = \alpha$ and the
ground field $K$ contains a third root of $a_\alpha$, we may assume
that $a_{\alpha}=1$ as well. This holds in particular when $K$ is
algebraically closed.
\end{remark}

\begin{remark} \label{rem:gentrian}
One could define an even more general version of a triangulation algebra by
utilizing the full power of Proposition~\ref{p:hyperib},
taking an $f$-invariant function $p \colon Q_1 \to K[[x]]^{\times}$ of
invertible power series, a $g$-invariant function $q \colon Q_1 \to
K[[x]] \setminus \{0\}$ such that the least order term
$c_\alpha x^{m_\alpha-1}$ of
each $q_\alpha(x)$ satisfies $m_\alpha n_\alpha \geq 2$,
and forming the quotient of the complete path algebra $\wh{KQ}$ by the
closure of the ideal $\cJ$ given by
\[
\cJ = \left(
\bigl\{
p_{\balpha}(\xi_{\balpha}) \cdot \balpha \cdot f(\balpha) -
q_{\alpha}(\oa) \cdot \oa' \bigr\}_{\alpha \in Q_1}
\right).
\]

However, it turns out by using techniques similar to that in the proof of
Theorem~\ref{t:quasi} that if the induced multiplicity function
$m \colon Q_1 \to \bZ_{>0}$ is admissible and $((Q,f),m)$ is not exceptional
(see Section~\ref{ssec:exceptional} below)
then the algebra $\wh{KQ}/\bar{\cJ}$ already occurs as an algebra of the form
discussed in Remark~\ref{rem:trianga} above,
compare with Remark~\ref{rem:genBGA} for Brauer graph algebras.

Nevertheless, for some triangulation quivers with non-admissible multiplicities
this generalized version does yield new algebras, see for example
Proposition~\ref{p:newquat} describing some new algebras of quaternion type not
appearing in the known lists.
\end{remark}

\subsection{Example -- triangulation algebras with one vertex}
\label{ssec:triang1}

In this section we work out in some detail the case of triangulation algebras
with one vertex. Already in this rather special case, one is able to
demonstrate many of the ideas and techniques that apply also in the general
case to be treated in Section~\ref{ssec:findim}.

Recall that the only triangulation quiver $(Q,f)$ with one
vertex has two loops $\alpha$ and $\beta$ with $f$ being the identity function
(Example~\ref{ex:triang1}).
Hence $\bar{\alpha} = \beta$, $\bar{\beta}=\alpha$
and $\oa = \alpha \beta$, $\oa' = \alpha$,
$\omega_\beta = \beta \alpha$, $\omega'_\beta = \beta$.
Since there is only one $g$-cycle, the multiplicities and scalars
are given by an integer $m \geq 1$ and some $c \in K^{\times}$.
In addition, there are parameters $\lambda_\alpha, \lambda_\beta \in K$
corresponding to the fixed points of $f$.

The triangulation algebra is the quotient
$\gL$ = $\wh{K \langle \alpha, \beta \rangle} / \bar{\cJ}$
(see the notation in Example~\ref{ex:BGAf1}),
where the generators of the ideal $\cJ$ are given by
\begin{align*}
\alpha^2 - \lambda_\alpha \alpha^3 - c (\beta \alpha)^{m-1} \beta &,&
\beta^2 - \lambda_\beta \beta^3 - c (\alpha \beta)^{m-1} \alpha .
\end{align*}

If $m=1$, the multiplicity is not admissible and
$\beta \in (\alpha^2, \cJ)$, $\alpha \in (\beta^2, \cJ)$, so by
induction we get $\alpha \in (\alpha^{4n}, \cJ)$ for any
$n \geq 1$, therefore
$\alpha \in \bar{\cJ}$ and similarly for $\beta$.
Hence the image of the arrows $\alpha$, $\beta$ in $\gL$ vanishes
and $\gL = K$.

If $m \geq 2$, the multiplicity is admissible. Define elements
$E_\alpha, E_\beta \in \wh{K \langle \alpha, \beta \rangle}$ by
\begin{align*}
E_\alpha &= (1 - \lambda_\alpha \alpha)^{-1} = 
1 + \lambda_\alpha \alpha + \lambda_\alpha^2 \alpha^2 + \dots \\
E_\beta &= (1 - \lambda_\beta \beta)^{-1} = 
1 + \lambda_\beta \beta + \lambda_\beta^2 \beta^2 + \dots
\end{align*}
Then $\alpha^2 = E_\alpha (\alpha^2 - \lambda_\alpha \alpha^3)$
and $\beta^2 = (\beta^2 - \lambda_\beta \beta^3) E_\beta$, hence
\begin{align*}
\alpha^2 \beta - c E_\alpha (\beta \alpha)^{m-2} \beta \alpha \beta^2
&= E_\alpha \left(
\alpha^2 - \lambda_\alpha \alpha^3
- c (\beta \alpha)^{m-1} \beta \right) \beta \in \cJ,
\\
\alpha \beta^2 - c \alpha^2 \beta (\alpha \beta)^{m-2} \alpha E_\beta
&= \alpha \left( (\beta^2 - \lambda_\beta \beta^3)
- c (\alpha \beta)^{m-1} \alpha \right) E_\beta \in \cJ,
\end{align*}
so
$\alpha^2 \beta - u \alpha \beta^2 \in \cJ$ and
$\alpha \beta^2 - \alpha^2 \beta v \in \cJ$
for some (infinite) linear combinations $u$ and $v$ of paths of positive
lengths. Therefore
\[
\alpha^2 \beta - u \alpha^2 \beta v 
= (\alpha^2 \beta - u \alpha \beta^2) + u(\alpha \beta^2 - \alpha^2 \beta v)
\in \cJ .
\]

It follows that $\alpha^2 \beta - u^n \alpha^2 \beta v^n
=\sum_{i=0}^{n-1} u^i ( \alpha^2 \beta - u \alpha^2 \beta v ) v^i \in \cJ$
for any $n \geq 1$, hence $\alpha^2 \beta \in \bar{\cJ}$. Similarly,
$\alpha \beta^2, \beta^2 \alpha, \beta \alpha^2 \in \bar{\cJ}$ so their
image in the quotient $\gL$ is zero.
Therefore, in $\gL$ we have
\[
\alpha^4 = E_\alpha (\alpha^2 - \lambda_\alpha \alpha^3) \alpha^2 =
E_\alpha c (\beta \alpha)^{m-1} \beta \alpha^2 = 0
\]
and similarly $\beta^4=0$. We deduce that
\begin{align*}
(\alpha \beta)^m &= (\alpha \beta)^{m-1} \alpha \beta = 
c^{-1}(\beta^2 - \lambda_\beta \beta^3) \beta = c^{-1} \beta^3 =
c^{-1} \beta (\beta^2 - \lambda_\beta \beta^3) \\
&= \beta (\alpha \beta)^{m-1} \alpha = (\beta \alpha)^m
= (\beta \alpha)^{m-1} \beta \alpha =
c^{-1}(\alpha^2 - \lambda_\alpha \alpha^3) \alpha = c^{-1} \alpha^3
\end{align*}
(compare with Remark~\ref{rem:dimercyc})
and a basis for $\gL$ is given by the $4m$ elements
\[
\{1\} \cup \left\{(\alpha \beta)^i, (\beta \alpha)^i \right\}_{0 < i < m}
\cup \left\{(\alpha \beta)^i \alpha, (\beta \alpha)^i \beta
\right\}_{0 \leq i < m} \cup \{(\alpha \beta)^m = (\beta \alpha)^m\}
\]
(compare with Proposition~\ref{p:basis}).

The discussion above shows parts~\eqref{it:p:rel1} and~\eqref{it:p:dim4m}
of the next statement.
For part~\eqref{it:p:rel2} one uses similar considerations whose details
will appear elsewhere.

\begin{proposition}
Let $\gL$ be a triangulation algebra with one vertex.
Then there exist parameters $m \geq 2$, $c \in K^{\times}$
and $\lambda_\alpha, \lambda_\beta \in K$ such that the following hold.
\begin{enumerate}
\renewcommand{\theenumi}{\alph{enumi}}
\item \label{it:p:rel1}
$\gL = KQ/I$ where $Q$ is the quiver
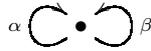
\[
\xymatrix@=1pc{
{\bullet} \ar@(dl,ul)[]^{\alpha} \ar@(dr,ur)[]_{\beta}
}
\]
and $I$ is the ideal of $KQ$ generated by the three elements
\begin{align}
\label{e:I1gen}
& \alpha^2 - c (\beta \alpha)^{m-1} \beta -
c \lambda_\alpha (\beta \alpha)^m, \, 
\beta^2 -  c (\alpha \beta)^{m-1} \alpha - c \lambda_\beta (\alpha \beta)^m \\
& \alpha^2 \beta .
\end{align}

\item \label{it:p:rel2}
The ideal $I$ is generated by the elements in~\eqref{e:I1gen} together with
the elements
\[
(\alpha \beta)^m - (\beta \alpha)^m ,\, 
(\alpha \beta)^m \alpha ,\, (\beta \alpha)^m \beta .
\]

\item \label{it:p:dim4m}
The algebra $\gL$ is finite-dimensional and $\dim_K \gL = 4m$.
\end{enumerate}
\end{proposition}

Comparing the description of $\gL$ with that of the local algebras of
quaternion type in~\cite[\S III]{Erdmann90}, we deduce:

\begin{corollary} \label{c:triang1}
An algebra of quaternion type with one simple module is a triangulation
algebra.
\end{corollary}

\subsection{Exceptional triangulation quivers with multiplicities}
\label{ssec:exceptional}

In proving the results on triangulation algebras one needs to distinguish
two exceptional cases where the triangulation quiver is self dual and
the admissible multiplicities are the minimal possible.

\begin{definition}
A pair $((Q,f),m)$ 
consisting of a connected triangulation quiver $(Q,f)$ and
a $g$-invariant multiplicity function $m \colon Q_1 \to \bZ_{>0}$
is \emph{exceptional}
if $(Q,f)$ is one of the two self-dual triangulation quivers shown
in Figure~\ref{fig:selfdual} and the function $m$ is the following:
\begin{itemize}
\item
$m_\alpha=3$ and $m_\beta = m_\gamma = m_\eta = 1$ in the punctured
monogon case (for the labeling of the arrows see row $2$ of
Table~\ref{tab:quivers});

\item
$m_\alpha=1$ for any arrow $\alpha$ in the tetrahedron case.
\end{itemize}
\end{definition}

The next statement characterizes the exceptional triangulation quivers with
multiplicities among all triangulation quivers with admissible multiplicities.

\begin{proposition} \label{p:except}
Let $(Q,f)$ be a connected triangulation quiver and
$m \colon Q_1 \to \bZ_{>0}$ an admissible $g$-invariant multiplicity function.
Then the following conditions are equivalent:
\begin{enumerate}
\renewcommand{\theenumi}{\alph{enumi}}
\item \label{it:p:except}
$((Q,f),m)$ is exceptional;

\item \label{it:p:mn3}
$m_\alpha n_\alpha = 3$ for all $\alpha \in Q_1$;

\item \label{it:p:mn1}
$(m_\alpha n_\alpha)^{-1} + (m_{f(\alpha)}n_{f(\alpha)})^{-1} + 
(m_{f^2(\alpha)} n_{f^2(\alpha)})^{-1} = 1$ for some $\alpha \in Q_1$.
\end{enumerate}
\end{proposition}

The implications \eqref{it:p:except}$\Rightarrow$\eqref{it:p:mn3}
and~\eqref{it:p:mn3}$\Rightarrow$\eqref{it:p:mn1} are trivial.
As with the other combinatorial statements of Section~\ref{sec:quivers},
the proof of the implication~\eqref{it:p:mn1}$\Rightarrow$\eqref{it:p:except}
will be given elsewhere.

\section{Representation-finite symmetric 2-CY-tilted algebras}
\label{sec:sym2CYfin}

In this section we classify all the symmetric 2-CY-tilted
algebras of finite representation type over an algebraically closed field.

Assume that $K$ is algebraically closed and consider first Brauer graph
algebras. By invoking the results of Erdmann and
Skowro\'{n}ski on the structure of the stable Auslander-Reiten
quiver of a self-injective special biserial
algebra~\cite[Theorems~2.1 and~2.2]{ES92}, recalling
that $\tau$-periodicity and $\Omega$-periodicity are equivalent for
symmetric algebras, we obtain:

\begin{proposition} \label{p:BGAperiodic}
The following conditions are equivalent for a Brauer graph algebra $\Gamma$.
\begin{enumerate}
\renewcommand{\theenumi}{\alph{enumi}}
\item
Any indecomposable non-projective $\Gamma$-module is $\Omega_{\Gamma}$-periodic;

\item
$\Gamma$ is of finite representation type.
\end{enumerate}
\end{proposition}

The class of Brauer graph algebras of finite representation type coincides
with that of the Brauer tree algebras.
A \emph{Brauer tree algebra} is a Brauer graph algebra whose underlying ribbon
graph is a tree and at most one node has multiplicity greater than 1
(this node is called \emph{exceptional}).

One of the first applications of Rickard's Morita theory for derived
categories was the derived equivalence classification of the Brauer
tree algebras~\cite[Theorem~4.2]{Rickard89}.
Rickard proved that any Brauer tree algebra is derived equivalent to a
Brauer tree algebra whose graph has a special shape, called a
\emph{Brauer star}, with the same number of edges and same
multiplicity of the exceptional node.

\begin{figure}
\begin{align*}
\xymatrix@=1pc{
& & {\circ} \\
{\circ} && && {\circ} \\
{\vdots} && {\circ_m} \ar@{-}[ull] \ar@{-}[uu] \ar@{-}[urr] && {\vdots}
}
& &
\xymatrix@=0.5pc{
&& {\bullet} \ar@(ur,ul)[]_{\beta} \ar[dll]_{\alpha} \\
{\bullet} \ar[dd]_{\alpha} \ar@(u,l)[]_{\beta} && &&
{\bullet} \ar[ull]_{\alpha} \ar@(r,u)[]_{\beta} \\ \\
\ar@{.}@/_1pc/[rrrr] && && \ar[uu]_{\alpha}
}
\end{align*}
\caption{A Brauer star and the corresponding ribbon quiver.}
\label{fig:star}
\end{figure}
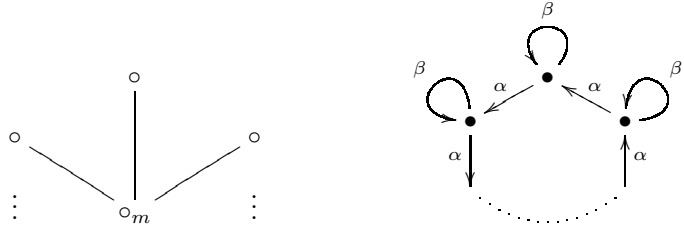

Figure~\ref{fig:star} shows a Brauer star with $n$ edges and multiplicity
$m$ of the exceptional node. The corresponding ribbon quiver is shown to
the right. The Brauer tree algebra has commutativity-relations of the form
$\beta - \alpha^{nm}$ and zero-relations
$\alpha \beta$ and $\beta \alpha$, hence the arrows $\beta$ can be
eliminated and one gets a symmetric Nakayama algebra with zero-relations
$\alpha^{nm+1}$. Some properties of this algebra are reviewed
in~\cite[\S4]{KessarLinckelmann10}. For the next statement, see 
paragraphs 4.11 and 4.12 there.

\begin{lemma} \label{l:Bstar}
Let $\Gamma$ be the Brauer star algebra with $n$ simple modules and
multiplicity $m$ of the exceptional node.
Let $S$ be a simple $\Gamma$-module.
If $m=n=1$ then $\Omega_\Gamma S \simeq S$, otherwise $\Omega^{2n}_\Gamma S
\simeq S$ but $\Omega^i_\Gamma S \not \simeq S$ for any $0 < i < 2n$.
\end{lemma}

We are now ready to state the classification of symmetric, 2-CY-tilted
algebras of finite representation type.

\begin{theorem} \label{t:sym2CYfin}
The following conditions are equivalent for an indecomposable, basic,
finite-dimensional algebra $\Gamma$ which is not simple.
\begin{enumerate}
\renewcommand{\theenumi}{\alph{enumi}}
\item \label{it:t:sym2CYfin}
$\Gamma$ is symmetric, 2-CY-tilted of finite representation type;

\item \label{it:t:sym4fin}
$\Gamma$ is symmetric of finite representation type and
$\Omega_\Gamma^4 M \simeq M$ for any $M \in \stmod \Gamma$;

\item \label{it:t:BGA2CY}
$\Gamma$ is a 2-CY-tilted Brauer graph algebra;

\item \label{it:t:BGA4}
$\Gamma$ is a Brauer graph algebra and $\Omega_\Gamma^4 M \simeq M$ for any
$M \in \stmod \Gamma$;

\item \label{it:t:BTA2}
$\Gamma$ is a Brauer tree algebra with at most two simple modules;

\item \label{it:t:BTAfig}
$\Gamma$ belongs to one of the three families of Brauer tree algebras shown in
Figure~\ref{fig:BGA2CY}.
\end{enumerate}
\end{theorem}

\begin{figure}
\[
\begin{array}{ccccc}
\xymatrix{
{\circ_m} \ar@{-}[r] & {\circ_1}
}
& \quad &
\xymatrix{
{\circ_1} \ar@{-}[r] & {\circ_m} \ar@{-}[r] & {\circ_1}
}
& \quad &
\xymatrix{
{\circ_m} \ar@{-}[r] & {\circ_1} \ar@{-}[r] & {\circ_1}
}
\\ \\
\xymatrix{
{\bullet} \ar@(ul,dl)[]_{\alpha}
}
& \quad &
\xymatrix{
{\bullet} \ar@<-0.5ex>[r]_{\beta} & {\bullet} \ar@<-0.5ex>[l]_{\gamma}
}
& \quad &
\xymatrix{
{\bullet} \ar@(ul,dl)[]_{\alpha} \ar@<-0.5ex>[r]_{\beta}
& {\bullet} \ar@<-0.5ex>[l]_{\gamma}
} \\
(\alpha^{m+1})
& \quad &
\left((\gamma \beta)^m \gamma, (\beta \gamma)^m \beta\right)
& \quad &
\left(\beta \gamma - \alpha^m, \gamma \alpha, \alpha \beta \right)
\\
_{m \geq 1}
& \quad &
_{m \geq 1}
& \quad &
_{m \geq 2}
\end{array}
\]
\caption{The indecomposable 2-CY-tilted symmetric algebras of
finite representation type that are not simple.
These algebras are Brauer tree algebras, and
for each family we show the ribbon graph with multiplicities (top);
the quiver, where we eliminated arrows (middle); and the corresponding
hyperpotential (bottom).
By allowing the value $m=0$ in the leftmost family one includes also the
simple, symmetric, 2-CY-tilted algebra $K$.}
\label{fig:BGA2CY}
\end{figure}
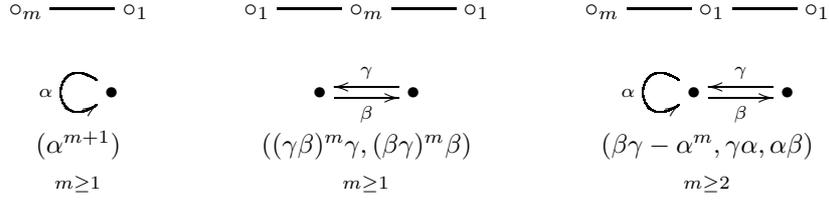

The implications \eqref{it:t:sym2CYfin}$\Rightarrow$\eqref{it:t:sym4fin}
and~\eqref{it:t:BGA2CY}$\Rightarrow$\eqref{it:t:BGA4} follow from
Proposition~\ref{p:period}.
The implication~\eqref{it:t:BGA4}$\Rightarrow$\eqref{it:t:BTA2} is
a consequence of Proposition~\ref{p:BGAperiodic} and Lemma~\ref{l:Bstar}.
The equivalence~\eqref{it:t:BTA2}$\Leftrightarrow$\eqref{it:t:BTAfig} is
clear. As each of the algebras shown in Figure~\ref{fig:BGA2CY} is
a Jacobian algebra of a hyperpotential and hence 2-CY-tilted by
Proposition~\ref{p:hyperCY}, this proves the implications
\eqref{it:t:BTAfig}$\Rightarrow$\eqref{it:t:sym2CYfin}
and~\eqref{it:t:BTAfig}$\Rightarrow$\eqref{it:t:BGA2CY}.

It remains to show that~\eqref{it:t:sym4fin}$\Rightarrow$\eqref{it:t:BTA2}.
By Riedtmann~\cite{Riedtmann80}, the stable Auslander-Reiten quiver of a
self-injective algebra of finite representation type has the form
$\bZ\Delta/\langle \phi \tau^{-r} \rangle$,
where $\Delta$ is a Dynkin graph $A_n$ ($n \geq 1$),
$D_n$ ($n \geq 4$) or $E_n$ ($n=6,7,8$), $\tau$ is the
translation of $\bZ\Delta$, $\phi$ is an automorphism of $\bZ\Delta$ with a
fixed vertex and $r \geq 1$.
Following Asashiba~\cite{Asashiba99}, these data are encoded in the
\emph{type} $(\Delta, r/(h_\Delta-1), t)$, where $h_\Delta$ is the Coxeter
number of $\Delta$ and $t$ is the order of $\phi$.

Asashiba~\cite{Asashiba99} classified the self-injective algebras up to
derived equivalence and described the possible types that
can occur. If $\Gamma$ is symmetric of finite representation type,
then our assumption in~\eqref{it:t:sym4fin} implies that $\tau_{\Gamma}^2$
acts as the identity on the vertices of the stable Auslander-Reiten quiver
of $\Gamma$ and hence either $r=2$ and $t=1$ or $r=1$ and $t \leq 2$.
Comparing this with the list of possible types in~\cite{Asashiba99},
one gets that the type of $\Gamma$ must be $(A_n, r/n, 1)$
for some $n \geq 1$ and $r \leq 2$ dividing $n$.
In particular, $\Gamma$ is derived equivalent to a symmetric Nakayama
algebra with $r \leq 2$ simple modules.
By~\cite{GabrielRiedtmann79,Rickard89}, $\Gamma$ is a Brauer tree algebra.

\begin{remark}
As a consequence of Theorem~\ref{t:sym2CYfin}, we see that the answer to
Question~\ref{q:sym2CYder} and Question~\ref{q:sym42CY}
is affirmative in the representation-finite case.
\end{remark}

\begin{remark}
The algebras listed in Figure~\ref{fig:BGA2CY} occur also as the endomorphism
algebras of cluster-tilting objects in the 2-Calabi-Yau
stable categories of maximal Cohen-Macaulay
modules over one dimensional simple hypersurface singularities
of types $A_{2m+1}$ and $D_{2m+2}$, see Proposition~2.4 and
Proposition~2.6 in~\cite{BIKR08}.
\end{remark}

\section{Triangulation algebras are of quasi-quaternion type}
\label{sec:triangquasi}

Let $K$ be a field.
Consider a connected triangulation quiver $(Q,f)$ together with
the following auxiliary data:
\begin{itemize}
\item
$g$-invariant function $m \colon Q_1 \to \bZ_{>0}$ of multiplicities;

\item
$g$-invariant function $c \colon Q_1 \to K^{\times}$ of scalars;

\item
a function $\lambda \colon Q_1^f \to K$, i.e.\ a scalar $\lambda_\alpha \in K$
for each arrow $\alpha$ with $f(\alpha)=\alpha$.
\end{itemize}
Assume that the following conditions hold:
\begin{itemize}
\item
$m$ is admissible, i.e.\ $m_\alpha n_\alpha \geq 3$ for each arrow $\alpha$;

\item
$((Q,f),m)$ is not exceptional; or 

\item
$((Q,f),m)$ is exceptional and the scalars $c \colon Q_1 \to K^{\times}$
satisfy
\begin{itemize}
\item
$\prod_{\alpha \in Q_1} c_\alpha \neq 1$ in the punctured monogon case; or
\item
$c_\alpha c_{\balpha} c_{f(\alpha)} c_{f(\balpha)} \neq 1$ for some $\alpha \in Q_1$
in the tetrahedron case.
\end{itemize}
\end{itemize}

Denote by $\gL = \gL(Q,f,m,c,\lambda)$ the triangulation algebra
(Definition~\ref{def:triang}) associated with the above data and by
$\Gamma = \Gamma(Q,f,m,c)$ the Brauer graph algebra
(Definition~\ref{def:BGA}).

\begin{theorem} \label{t:quasi}
Under the above conditions, we have:
\begin{enumerate}
\renewcommand{\theenumi}{\alph{enumi}}
\item \label{it:t:finite}
$\gL$ is finite-dimensional.

\item \label{it:t:symmetric}
$\gL$ is symmetric.

\item \label{it:t:tame}
$\gL$ is of tame representation type. Moreover, if $((Q,f),m)$ is not
exceptional then the Brauer graph algebra $\Gamma$ is a degeneration of $\gL$.

\item \label{it:t:potential}
$\gL$ is a Jacobian algebra of a hyperpotential
and therefore it is 2-CY-tilted, i.e.\ there is a 2-Calabi-Yau
triangulated category $\cC$ and a cluster-tilting object $T$ in $\cC$
such that $\gL \simeq \End_{\cC}(T)$.

\item \label{it:t:quasi}
$\gL$ is of quasi-quaternion type.

\item \label{it:t:quasimut}
More generally, for any cluster-tilting object $T'$ in $\cC$
which is reachable from $T$, the 2-CY-tilted algebra $\End_{\cC}(T')$
is derived equivalent to $\gL$ and of quasi-quaternion type.
\end{enumerate}
\end{theorem}

\begin{remark}
The theorem holds also when the multiplicities are not admissible
(but still $m_\alpha n_\alpha \geq 2$ for any arrow $\alpha$),
but then more exceptional cases are needed to be taken care of.
\end{remark}

\begin{remark}
Remark~\ref{rem:nvertex} implies that for any positive integer $n \geq 1$
there exists a triangulation quiver with $n$ vertices and hence
a triangulation algebra with $n$ simple modules.
The algebras of part~\eqref{it:t:quasimut} of the theorem thus provide
many instances of tame, symmetric, indecomposable algebras with periodic
modules which seem
to be missing from the classification announced
in~\cite[Theorem~6.2]{ES08} and~\cite[Theorem~8.7]{Skowronski06}. Moreover,
they provide counterexamples to~\cite[Corollary~8.8(3)]{Skowronski06} which
claims to bound the number of simple modules of such algebras 
of infinite representation type by $10$.
\end{remark}

Part~\eqref{it:t:potential} of the theorem follows from
part~\eqref{it:t:finite}, Proposition~\ref{p:potential}
and Proposition~\ref{p:hyperCY}.
Part~\eqref{it:t:quasi} is a consequence of parts~\eqref{it:t:symmetric},
\eqref{it:t:tame}, \eqref{it:t:potential} and Proposition~\ref{p:period}.
Finally, part~\eqref{it:t:quasimut} is a consequence of
parts~\eqref{it:t:symmetric}, \eqref{it:t:potential},
Corollary~\ref{c:dereq}, part~\eqref{it:t:quasi}
and Proposition~\ref{p:derquasi}.
The ideas behind the proof of parts~\eqref{it:t:finite} and~\eqref{it:t:tame}
are explained in the next sections.

\subsection{Remarks on finite-dimensionality}
\label{ssec:findim}

We keep the notations as in the preceding section.

Motivated by the dimer model perspective of Section~\ref{ssec:dimer},
a path $\alpha \cdot f(\alpha) \cdot gf(\alpha)$ may be called 
\emph{zig-zag} path. A crucial point in proving part~\eqref{it:t:finite}
of Theorem~\ref{t:quasi} is the vanishing of the images of these zig-zag paths
in $\gL$, whose proof we sketch below.
Let $\cJ$ be the ideal defining the triangulation algebra, so that
$\gL = \wh{KQ} / \bar{\cJ}$, and let $\alpha \in Q_1$ be any arrow.
Lemma~\ref{l:gf2fg2} implies that $gf(\beta) = fg^{-2}(\beta)$ for any
arrow $\beta$, so we can repeatedly use the commutativity-relations
defining $\cJ$ to deduce that
\[
\alpha \cdot f(\alpha) \cdot gf(\alpha) -
u \cdot \alpha \cdot f(\alpha) \cdot gf(\alpha) \cdot v \in \cJ
\]
for some $u, v \in \wh{KQ}$.

If $((Q,f),m)$ is not exceptional, then one shows using
Proposition~\ref{p:except} that $u$ and $v$ are
linear combinations of paths of positive length, and therefore
deduces that $\alpha \cdot f(\alpha) \cdot gf(\alpha) \in \bar{\cJ}$
as done in the case discussed in Section~\ref{ssec:triang1}.

If $((Q,f),m)$ is exceptional, then one shows that $u = C_u + u'$ and
$v = C_v + v'$ where $u',v'$ are linear combinations of paths of positive
length (in the tetrahedron case even $u'=v'=0$) and
$C_u, C_v \in K^{\times}$ are scalars satisfying $C_u C_v \neq 1$ by our
additional assumption on the function $c \colon Q_1 \to K^{\times}$. Hence
$\alpha \cdot f(\alpha) \cdot gf(\alpha) \in \bar{\cJ}$ as well.

The next proposition provides a more refined version of part~\eqref{it:t:finite}
of Theorem~\ref{t:quasi} by giving a presentation of the triangulation algebra
as quiver with relations. 
\begin{proposition} \label{p:qrel}
Under the hypotheses of Theorem~\ref{t:quasi}, the triangulation algebra
$\gL(Q,f,m,c,\lambda)$ is the quotient of the path algebra $KQ$ by the
ideal generated by the elements
\begin{align}
\label{e:comm3}
&\balpha \cdot f(\balpha) - c_\alpha \oa^{m_\alpha-1} \cdot \oa'
&& \alpha \in Q_1 \text{ and } f(\balpha) \neq \balpha, \\
\label{e:comm1}
&\balpha^2 - c_\alpha \oa^{m_\alpha-1} \cdot \oa'
- c_\alpha \lambda_{\balpha} \oa^{m_\alpha}
&& \alpha \in Q_1 \text{ and } f(\balpha) = \balpha, \\
\label{e:zigzeroF}
&\alpha \cdot f(\alpha) \cdot gf(\alpha) && \alpha \in Q_1.
\end{align}
\end{proposition}

\begin{remark}
The ideal of relations in Proposition~\ref{p:qrel} is not changed if 
the relations of type~\eqref{e:zigzeroF} are replaced by
\begin{align}
\label{e:zigzeroG}
&\alpha \cdot g(\alpha) \cdot fg(\alpha) && \alpha \in Q_1.
\end{align}

Moreover, it turns out that it is enough to specify a zero-relation as
in~\eqref{e:zigzeroF} or~\eqref{e:zigzeroG} for just one arrow
$\alpha \in Q_1$, as the relations for the other arrows would then
follow from the commutativity-relations~\eqref{e:comm3} and~\eqref{e:comm1}.

In the cases where $((Q,f),m)$ is exceptional, 
our assumption on the scalars~$c$ implies that all the
zero-relations~\eqref{e:zigzeroF} and~\eqref{e:zigzeroG} already follow from
the relations~\eqref{e:comm3} and~\eqref{e:comm1},
provided that in the punctured monogon case the scalar $\lambda_\eta$
associated to the loop $\eta \in Q_1^f$ vanishes
(no additional assumption is needed in the tetrahedron case).
\end{remark}

\begin{remark} \label{rem:dimercyc}
Let $i \in Q_0$ and let $\alpha, \balpha$ be the two arrows starting at $i$.
Then the images in $\gL$ of the following cycles starting at $i$ are equal:
\[
\alpha \cdot f(\alpha) \cdot f^2(\alpha) =
c_{\balpha} \omega_{\balpha}^{m_{\balpha}} =
c_{\alpha} \oa^{m_\alpha} = \balpha \cdot f(\balpha) \cdot f^2(\balpha).
\]
Denote this common value by $z_i \in \gL$.
\end{remark}

The next statement is a generalization of our result
in~\cite[\S4.1]{Ladkani12} and its proof is similar.

\begin{proposition} \label{p:basis}
A basis of the triangulation algebra $\gL(Q,f,m,c,\lambda)$ is given
by the images of the paths
\[
\{e_i\}_{i \in Q_0} \cup
\left\{ \alpha \cdot g(\alpha) \cdot \ldots \cdot g^r(\alpha)
\right\}_{\alpha \in Q_1, 0 \leq r < m_\alpha n_\alpha -1} \cup
\{z_i\}_{i \in Q_0} .
\]
\end{proposition}

\begin{example}
In the case of the triangulation quiver with one vertex, the same basis for the
triangulation algebra has been constructed in Section~\ref{ssec:triang1}.
\end{example}

Using the basis of Proposition~\ref{p:basis}, one can explicitly compute the
Cartan matrix of a triangulation algebra in terms of its defining combinatorial
data. As the details of the computation are similar to those given
in~\cite[\S4.4]{Ladkani12}, we will state the result without proof.

For a triangulation quiver $(Q,f)$, recall that $\Omega_g$ denotes the set
of $g$-cycles in $Q_1$. The vectors $\chi_\omega \in \bZ^{Q_0}$ for
$\omega \in \Omega_g$ have been defined at the end of Section~\ref{ssec:ribbon},
see also the paragraph preceding Proposition~\ref{p:CartanBGA}.

\begin{proposition} \label{p:Cartan}
Consider $\gL=\gL(Q,f,m,c,\lambda)$ and $\Gamma=\Gamma(Q,f,m,c)$.
Under the hypotheses of Theorem~\ref{t:quasi}, we have:
\begin{enumerate}
\renewcommand{\theenumi}{\alph{enumi}}
\item \label{it:p:Cartan}
$C_\gL = C_{\Gamma} =
\sum_{\omega \in \Omega_g} m_\omega \chi_\omega^T \chi_\omega$.

\item \label{it:p:dim}
$\dim_K \gL = \dim_K \Gamma = 
\sum_{\omega \in \Omega_g} m_\omega |\omega|^2$.

\item \label{it:p:nonneg}
The quadratic form $q_{C_\gL} \colon \bZ^{Q_0} \to \bZ$ defined by
$q_{C_\gL}(x) = x C_{\gL} x^T$ takes non-negative even values;
in particular it is non-negative definite.

\item \label{it:p:rank}
$\rank C_\gL \leq |\Omega_g| \leq |Q_0|$.

\item \label{it:p:det0}
$\det C_\gL \neq 0$ if and only if $(Q,f)$ is one of the
triangulation quivers $1$, $2$, $3a$ or $3b$ of Table~\ref{tab:quivers}.
In this case, $\det C_{\gL} = 4 \cdot \prod_{\omega \in \Omega_g} m_\omega$.
\end{enumerate}
\end{proposition}

\begin{remark} 
From part~\eqref{it:p:det0} we see that any triangulation algebra with
$n>10$ simple modules provides a counterexample
to~\cite[Corollary~8.8(1)]{Skowronski06} which states that a tame, symmetric
algebra of infinite representation type with periodic modules and singular
Cartan matrix must be isomorphic to the trivial extension of a tubular algebra
(and hence has at most $10$ simple modules).
\end{remark}

\begin{remark} \label{rem:quat}
Combining part~\eqref{it:p:det0} of Proposition~\ref{p:Cartan} with 
part~\eqref{it:t:quasi} of Theorem~\ref{t:quasi} we deduce that the
triangulation algebra $\gL(Q,f,m,c,\lambda)$ is of quaternion type
(and not just of quasi-quaternion type)
if and only if $(Q,f)$ is any of the triangulation quivers $1$, $2$, $3a$ or
$3b$. Further details will be given in Section~\ref{ssec:quat}.
\end{remark}

Parts~\eqref{it:p:Cartan}, \eqref{it:p:dim}, \eqref{it:p:nonneg} follow from
the corresponding statements of Proposition~\ref{p:CartanBGA}, observing that
the basis constructed in Proposition~\ref{p:basis} for $\gL$ and
that constructed in Lemma~\ref{l:basisBGA} for $\Gamma$ consist of images (in
the respective algebras) of the same set of paths.
The inequality $|\Omega_g| \leq |Q_0|$ in part~\eqref{it:p:rank} is a
consequence of Proposition~\ref{p:gcycles}, which also implies the ``only if''
direction of part~\eqref{it:p:det0}. The ``if'' direction follows from explicit
calculations which are presented in the next example for the purpose of
illustration.

\begin{example}
For each of the triangulation quivers $1$, $2$, $3a$ and $3b$ in
Table~\ref{tab:quivers} we compute the Cartan matrix of a triangulation algebra
on that quiver. Recall that for a $g$-cycle $\omega$, the quantity $m_\omega$
used in Proposition~\ref{p:Cartan} equals any of the values $m_\alpha$ for
$\alpha \in \omega$.

\begin{enumerate}
\item
There is one $g$-cycle $(\alpha \beta)$ with $\chi_{(\alpha \beta)} = (2)$.
The Cartan matrix is $(4m_\alpha)$.

\item
There are two $g$-cycles $(\alpha)$ and $(\eta \gamma \beta)$ with
$\chi_{(\alpha)}=(1,0)$ and $\chi_{(\eta \gamma \beta)} = (1,2)$, the Cartan
matrix is
\[
m_\alpha \begin{pmatrix} 1 & 0 \\ 0 & 1 \end{pmatrix}
+ m_\eta \begin{pmatrix} 1 & 2 \\ 2 & 4 \end{pmatrix}
\]
and its determinant is $4 m_\alpha m_\eta$.

\item[($3a$)]
There are three $g$-cycles $(\alpha)$, $(\beta \delta \eta \gamma)$ and $(\xi)$
with
$\chi_{(\alpha)} = (1,0,0)$,
$\chi_{(\beta \delta \eta \gamma)} = (1,2,1)$ and
$\chi_{(\xi)} = (0,0,1)$,
the Cartan matrix is
\[
m_\alpha \begin{pmatrix} 1 & 0 & 0 \\ 0 & 0 & 0 \\ 0 & 0 & 0 \end{pmatrix} +
m_\beta \begin{pmatrix} 1 & 2 & 1 \\ 2 & 4 & 2 \\ 1 & 2 & 1 \end{pmatrix} +
m_\xi \begin{pmatrix} 0 & 0 & 0 \\ 0 & 0 & 0 \\ 0 & 0 & 1 \end{pmatrix}
\]
and its determinant is
$4 m_\alpha m_\beta m_\xi$.

\item[($3b$)]
There are three $g$-cycles $(\alpha_1 \beta_1)$, $(\alpha_2 \beta_2)$ and
$(\alpha_3 \beta_3)$ with
$\chi_{(\alpha_1 \beta_1)} = (1,1,0)$,
$\chi_{(\alpha_2 \beta_2)} = (0,1,1)$ and
$\chi_{(\alpha_3 \beta_3)} = (1,0,1)$,
the Cartan matrix is
\[
m_{\alpha_1}
\begin{pmatrix} 1 & 1 & 0 \\ 1 & 1 & 0 \\ 0 & 0 & 0 \end{pmatrix} +
m_{\alpha_2}
\begin{pmatrix} 0 & 0 & 0 \\ 0 & 1 & 1 \\ 0 & 1 & 1 \end{pmatrix} +
m_{\alpha_3}
\begin{pmatrix} 1 & 0 & 1 \\ 0 & 0 & 0 \\ 1 & 0 & 1 \end{pmatrix}
\]
and its determinant is
$4 m_{\alpha_1} m_{\alpha_2} m_{\alpha_3}$.
\end{enumerate}
\end{example}

\subsection{Remarks on tameness}

In this section we sketch the proof of part~\eqref{it:t:tame} of
Theorem~\ref{t:quasi}. We keep the notations as in the preceding sections.

First, we recall the notion of degeneration of algebras appearing in the
statement of part~\eqref{it:t:tame}.
For a positive integer $d$, denote by $\alg_d(K)$ the affine variety of
associative algebra structures with unit on the vector space $K^d$. The
group $\GL_d(K)$ acts on $\alg_d(K)$ by transport of structure and its
orbits correspond bijectively to the isomorphism classes of
$d$-dimensional $K$-algebras. Given two $d$-dimensional algebras
$\Gamma$ and $\gL$ viewed as points in $\alg_d(K)$, we say that $\Gamma$ is a
\emph{degeneration} of $\gL$ if $\Gamma$ lies in the closure of the 
$\GL_d(K)$-orbit of $\gL$ in the Zariski topology of $\alg_d(K)$.

Assume that $((Q,f),m)$ is not exceptional and
let $N=\lcm(m_\alpha n_\alpha)_{\alpha \in Q_1}$, so that
$N/(m_\alpha n_\alpha)$ is a positive integer for any $\alpha \in Q_1$.
For each $\alpha \in Q_1$,
set
\[
e_\alpha = 1 - \left(
(m_\alpha n_\alpha)^{-1} + (m_{f(\alpha)} n_{f(\alpha)})^{-1}
+ (m_{f^2(\alpha)} n_{f^2(\alpha)})^{-1} \right) .
\]
If $f(\alpha)=\alpha$, set also $e'_\alpha = 1 - 2 (m_\alpha n_\alpha)^{-1}$.
Note that $e_\alpha$ and $e'_\alpha$ are rational numbers and moreover
$N e_\alpha$, $N e'_\alpha$ are integers by construction.

Since $((Q,f),m)$ is not exceptional, Proposition~\ref{p:except}
implies that $e_\alpha > 0$ for any arrow $\alpha$ (hence also
$e'_\alpha > 0$ for any $\alpha \in Q_1^f$).
For $t \in K$, let $I_t$ be the ideal of the path algebra $KQ$ generated by
the elements
\begin{align*}
&\balpha \cdot f(\balpha)
- c_\alpha t^{N e_{\balpha}} \oa^{m_\alpha-1} \oa'
&& \alpha \in Q_1 \text{ and } f(\balpha) \neq \balpha, \\
&\balpha^2
- c_\alpha t^{N e_{\balpha}} \oa^{m_\alpha-1} \oa'
- c_\alpha \lambda_{\balpha} t^{N e'_{\balpha}} \oa^{m_\alpha}
&& \alpha \in Q_1 \text{ and } f(\balpha) = \balpha, \\
& \alpha \cdot f(\alpha) \cdot gf(\alpha) && \alpha \in Q_1, \\
& c_\alpha \oa^{m_\alpha} - c_{\balpha} \oba^{m_{\balpha}} && \alpha \in Q_1,
\end{align*}
and let $\gL_t = KQ/I_t$.

\begin{proposition} \label{p:degen}
Assume that $((Q,f),m)$ is not exceptional.

\noindent
Let $\Gamma=\Gamma(Q,f,m,c)$ be the Brauer graph algebra
and let $\gL=\gL(Q,f,m,c,\lambda)$ be the triangulation algebra
as in Theorem~\ref{t:quasi}. Then:
\begin{enumerate}
\renewcommand{\theenumi}{\alph{enumi}}
\item
$\gL_0 \simeq \Gamma$.

\item
$\gL_1 \simeq \gL$.

\item
For any $t \in K^{\times}$, the automorphism of $KQ$ defined by sending each
arrow $\alpha$ to $t^{N/(m_\alpha n_\alpha)} \alpha$ maps $I_1$ onto $I_t$,
hence $\gL_t \simeq \gL$.

\item
$\Gamma$ is a degeneration of $\gL$.
\end{enumerate}
\end{proposition}

We constructed a one-parameter family of algebras $\{\gL_t\}$ for $t \in K$
such that $\gL_t \simeq \gL$ for $t \neq 0$ and $\gL_0 \simeq \Gamma$.
Since $\Gamma$ is tame (Proposition~\ref{p:BGA}), a degeneration
theorem of Geiss~\cite{Geiss95} implies that $\gL$ is also tame.

\begin{remark}
In~\cite[\S6]{Holm99} Holm establishes the tameness of the algebras of
quaternion type with $2$ or $3$ simple modules
by showing that some of them degenerate to algebras of dihedral type
and then applying the result in~\cite{Geiss95}. Proposition~\ref{p:degen}
can be seen as a generalization of this statement to arbitrary triangulation
quivers.
\end{remark}

\begin{example}
The algebras of quaternion type with one simple module are precisely the
algebras listed as items (5) and (5') in the paper~\cite{Ringel74} by Ringel
dealing with the representation type of local algebras, but their representation
type was not determined in that paper.
Their tameness was later established by Erdmann~\cite[III.1.2]{Erdmann90}
as a consequence of the result of~\cite{BD75} mentioned in
Section~\ref{ssec:group}.

Since these algebras are triangulation algebras (see Corollary~\ref{c:triang1}),
Proposition~\ref{p:degen} thus yields an alternative proof of their tameness.
For the purpose of illustration, let us carry out the explicit calculations.

Recall from Section~\ref{ssec:triang1} that a triangulation quiver with one
vertex has two loops $\alpha$ and $\beta$ with the function $f$ being the
identity. Hence there is one $g$-cycle and the auxiliary algebraic data
is given by a positive integer multiplicity $m$, which is admissible if
$m \geq 2$, and scalars $c \in K^{\times}$ and $\lambda_\alpha, \lambda_\beta
\in K$.

Therefore $n_\alpha = n_\beta = 2$ and $m_\alpha = m_\beta = m$, hence
$N=2m$ and
\begin{align*}
e_\alpha = e_\beta = 1- \frac{3}{2m} &,&
N e_\alpha = N e_\beta = 2m-3, \\
e'_\alpha = e'_\beta = 1 - \frac{2}{2m} &,&
N e'_\alpha = N e'_\beta = 2m-2,
\end{align*}
so the defining relations of the algebra $\gL_t$ (for any $t \in K$)
are given by
\begin{align*}
&\alpha^2 - c t^{2m-3} (\beta \alpha)^{m-1} \beta
- c \lambda_\alpha t^{2m-2} (\beta \alpha)^m,\, \beta^2 \alpha,\\
&\beta^2 - c t^{2m-3} (\alpha \beta)^{m-1} \alpha
- c \lambda_\beta t^{2m-2} (\alpha \beta)^m,\, \alpha^2 \beta,\,
(\alpha \beta)^m - (\beta \alpha)^m .
\end{align*}

If $t \neq 0$, the linear map defined by sending $\alpha$ to $t \alpha$
and $\beta$ to $t \beta$ induces an isomorphism between the algebras $\gL_1$ 
and $\gL_t$, the former being equal to the triangulation algebra associated
with the auxiliary data as described in Section~\ref{ssec:triang1}.
Therefore the algebra $\gL_0$, which is precisely the Brauer graph algebra
associated with these data (see Example~\ref{ex:BGAf1}), is a degeneration of
the corresponding triangulation algebra. It follows that the latter algebra is
also tame.
\end{example}

\begin{remark}
When $((Q,f),m)$ is exceptional, Proposition~\ref{p:degen} does not apply but
the triangulation algebra is still tame since it is of tubular type~\cite{BS03},
see also Section~\ref{ssec:tubular} below.
\end{remark}

\section{Known families of algebras as triangulation algebras}
\label{sec:known}

\subsection{Algebras of quaternion type}
\label{ssec:quat}

In~\cite[pp.~303-306]{Erdmann90}, Erdmann gave a list of the possible quivers
with relations of the algebras of quaternion type, and asked whether any such
algebra is indeed of quaternion type~\cite[VII.9]{Erdmann90}.
Later, Holm~\cite[\S6]{Holm99}
proved that the algebras in this list are of tame representation type.
Erdmann and Skowro\'{n}ski~\cite[Theorem~5.9]{ES06} proved that these algebras
are periodic
of period dividing $4$ by constructing projective bimodule resolutions for them
and deduced that they are indeed of quaternion type.
In this section we give an alternative proof of the periodicity 
of modules for these algebras by showing that all the algebras in Erdmann's
list are 2-CY-tilted.

Assume that the ground field is algebraically closed.
We say that an algebra is of
\emph{possibly quaternion type} if it appears in Erdmann's list.
Consider two families of algebras in Erdmann's list whose quivers with
relations are shown in Figure~\ref{fig:quat}, where 
for the convenience of the reader we tried to keep the
notations as close as possible to the original ones.
The first family $\cQ(2\cB)_1^{k,s}(a,c)$ depends on integer parameters
$k \geq 1$, $s \geq 2$ such that $k+s \geq 4$
and scalars $a \in K^{\times}$ and $c \in K$.
If $(k,s)=(1,3)$, one should assume that $a \neq 1$, otherwise one could set
$a=1$.
The second family $\cQ(3\cK)^{a,b,c}$
depends on three integers $1 \leq a \leq b \leq c$ such that at most one
of them equals $1$. The scalar $d \in K^{\times}$ should be set to $1$,
unless $(a,b,c)=(1,2,2)$ and then $d \neq 1$.

\begin{figure}
\[
\begin{array}{lcl}
\begin{array}{l}
\cQ(2\cB)_1^{k,s}(a, c) \\
_{k \geq 1,\, s \geq 2,\, k+s \geq 4} \\
_{a \in K^{\times}\!, \, c \in K}
\end{array}
&
\begin{array}{c}
\xymatrix{
{\bullet_0} \ar@(ul,dl)[]_{\alpha} \ar@<-0.5ex>[r]_{\beta}
& {\bullet_1} \ar@(dr,ur)[]_{\eta} \ar@<-0.5ex>[l]_{\gamma}
}
\end{array}
&
\begin{array}{l}
\alpha^2 - a (\beta \gamma \alpha)^{k-1} \beta \gamma
- c (\beta \gamma \alpha)^k \\
\beta \eta - (\alpha \beta \gamma)^{k-1} \alpha \beta \\
\eta \gamma - (\gamma \alpha \beta)^{k-1} \gamma \alpha \\
\gamma \beta - \eta^{s-1} \\
\alpha^2 \beta \,,\, \gamma \alpha^2
\end{array}
\\ \\
\begin{array}{l}
\cQ(3\cK)^{a,b,c} \\
_{1 \leq a, \, \max(2,a) \leq b \leq c} \\
_{d \in K^{\times}}
\end{array}
&
\begin{array}{c}
\xymatrix@=1pc{
& {\bullet_2} \ar@<-0.5ex>[ddl]_{\lambda} \ar@<-0.5ex>[ddr]_{\eta} \\ \\
{\bullet_0} \ar@<-0.5ex>[rr]_{\beta} \ar@<-0.5ex>[uur]_{\kappa}
&& {\bullet_1} \ar@<-0.5ex>[ll]_{\gamma} \ar@<-0.5ex>[uul]_{\delta}
}
\end{array}
&
\begin{array}{l}
\beta \delta - (\kappa \lambda)^{a-1} \kappa \\
\eta \gamma - (\lambda \kappa)^{a-1} \lambda \\
\delta \lambda - (\gamma \beta)^{b-1} \gamma \\
\kappa \eta - (\beta \gamma)^{b-1} \beta \\
\lambda \beta - d (\eta \delta)^{c-1} \eta \\
\gamma \kappa - d (\delta \eta)^{c-1} \delta \\
\lambda \beta \gamma \,,\, \kappa \eta \delta
\end{array}
\end{array}
\]
\caption{Quivers with relations of some algebras of possibly quaternion type.}
\label{fig:quat}
\end{figure}
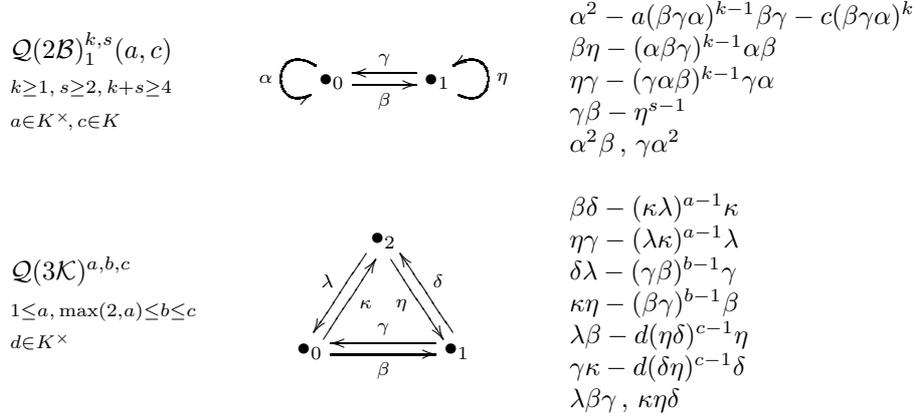

Our presentation slightly deviates from the lists in the existing literature.
The next two remarks explain the differences and the motivation behind them.

\begin{remark}
In the literature, the family $\cQ(2\cB)^{k,s}_1(a,c)$ is defined only for
$s \geq 3$ and $k \geq 1$. We extended the definition to include the case
where $s=2$ and $k \geq 2$. In this case the arrow $\eta$ can be
eliminated from the quiver and one actually gets the algebras in another
family $\cQ(2\cA)^k(c)$ for $k \geq 2$ and $c \in K$.
The reason for including these algebras is to have a complete list of the
derived equivalence classes of the algebras of (possibly) quaternion type,
needed in the proof of Theorem~\ref{t:quat2CY} below. 
Thus one has to modify the statements of~\cite[Proposition~5.8]{ES06}, \cite[Theorem~5.7]{ES08}, \cite[Theorem~5.1]{Holm99}
and~\cite[Theorem~8.6]{Skowronski06} accordingly,
otherwise the derived equivalence classes of the algebras $\cQ(2\cA)^2(c)$
would be missing.
\end{remark}

\begin{remark}
The parameter $d$ for the family $\cQ(3\cK)^{a,b,c}$ does not appear in the
literature. In fact, in the original tables of~\cite{Erdmann90}, the
parameters were assumed to satisfy $2 \leq a \leq b \leq c$.
Only in~\cite{Holm99} one value of $1$ was allowed. Note that if $a=1$
then the two arrows $\kappa$ and $\lambda$ can be eliminated from the
quiver and one actually gets the algebras in another family
$\cQ(3\cA)_1^{b,c}(d)$ for $b,c \geq 2$. 
For this family, if $(b,c)=(2,2)$, one should assume that $d \neq 1$,
otherwise one could set $d=1$.

We also slightly modified the presentation of the zero-relations.
If $a \geq 2$, then as noted in~\cite[Theorem~VII.8.8]{Erdmann90},
all the twelve zig-zag paths
\begin{align*}
\beta \gamma \kappa &,& \beta \delta \eta &,&
\gamma \beta \delta &,& \gamma \kappa \lambda &,&
\delta \eta \gamma &,& \delta \lambda \kappa &,&
\eta \delta \lambda &,& \eta \gamma \beta &,&
\lambda \kappa \eta &,& \lambda \beta \gamma &,&
\kappa \lambda \beta &,& \kappa \eta \delta
\end{align*}
vanish, and it suffices to specify one such zero-relation.
However, if $a=1$, there are zig-zag paths that do not vanish; only those
paths containing $\kappa$ or $\lambda$ do vanish, and it suffices to
specify two zero-relations as in Figure~\ref{fig:quat}. These relations
correspond to the two zero-relations occurring in the definition of the
family $\cQ(3\cA)_1^{b,c}(d)$.
\end{remark}

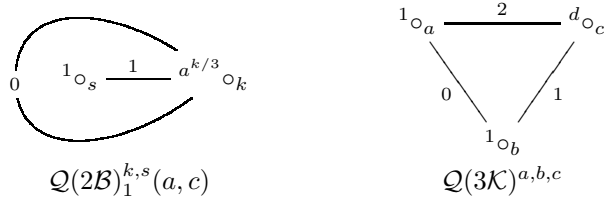
\begin{figure}
\[
\begin{array}{ccc}
\begin{array}{c}
\xymatrix@=0.75pc{
{_0} & {{^1}\circ_{s}} \ar@{-}[rr]^(0.4)1
&& {{^{a^{k/3}}}\circ_{k}} \ar@{-}@(ul,u)[lll]_{} \ar@{-}@(dl,d)[lll]_{}
}
\end{array}
& \qquad &
\begin{array}{c}
\xymatrix@=1pc{
{{^1}\circ_a} \ar@{-}[rr]^2 && {{^d}\circ_c} \ar@{-}[ddl]^1 \\ \\
& {{^1}\circ_b} \ar@{-}[uul]^0
}
\end{array}
\\
\cQ(2\cB)_1^{k,s}(a,c)
& \qquad &
\cQ(3\cK)^{a,b,c}
\end{array}
\]
\caption{Description of some families of algebras of (possibly) quaternion
type~\cite{Erdmann90} as triangulation algebras.
The labeling of the edges
corresponds to that of the vertices in Figure~\ref{fig:quat}.
The subscript at each node indicates the corresponding multiplicity whereas
the superscript indicates the scalar.}
\label{fig:quatriang}
\end{figure}

\begin{lemma} \label{l:quatriang}
An algebra in the family $\cQ(2\cB)_1^{k,s}(a,c)$ or
$\cQ(3\cK)^{a,b,c}$ is a triangulation algebra.
\end{lemma}
\begin{proof}[Proof (sketch)]
First, rescaling the arrow $\alpha$ by a factor of $a^{1/3}$
we slightly change the presentation of the algebra $\cQ(2\cB)_1^{k,s}(a,c)$
and get the relations
\begin{align*}
&\alpha^2 - a^{k/3}(\beta \gamma \alpha)^{k-1} \beta \gamma
- \lambda (\beta \gamma \alpha)^k
&& \gamma \beta - \eta^{s-1} \\
&\beta \eta - a^{k/3}(\alpha \beta \gamma)^{k-1} \alpha \beta
&& \alpha^2 \beta \\
&\eta \gamma - a^{k/3}(\gamma \alpha \beta)^{k-1} \gamma \alpha
&& \gamma \alpha^2
\end{align*}
for some scalar $\lambda \in K$.

This algebra is isomorphic to the triangulation algebra 
$\gL(Q,f,m,c,\lambda)$ with the following data:
\begin{itemize}
\item
The triangulation quiver $(Q,f)$ is isomorphic to quiver $2$ in 
Table~\ref{tab:quivers}, with the permutation $f$ written in cycle form as
$(\alpha)(\eta \, \gamma \, \beta)$, so that the permutation $g$ is
$(\alpha \, \beta \, \gamma)(\eta)$;

\item
In terms of the corresponding ribbon graph, the $g$-invariant functions $m$
and~$c$ are shown in Figure~\ref{fig:quatriang}.
The multiplicities are admissible when $s \geq 3$. The triangulation quiver
with multiplicities is exceptional precisely when $(k,s)=(1,3)$, and the
assumption that $a \neq 1$ in this case ensures that Theorem~\ref{t:quasi}
holds;

\item
The scalar $\lambda_\alpha$ for the loop $\alpha$ with $f(\alpha)=\alpha$ 
is $\lambda$.
\end{itemize}

Similarly, the algebra $\cQ(3\cK)^{a,b,c}$ is a triangulation algebra
for the following data:
\begin{itemize}
\item
the triangulation quiver is isomorphic to quiver $3b$ in
Table~\ref{tab:quivers}, with the permutation $f$ written in cycle form as
$(\beta \, \delta \, \lambda)(\kappa \, \eta \, \gamma)$,
so that the permutation $g$ is
$(\beta \, \gamma)(\delta \, \eta)(\kappa \, \lambda)$;

\item
In terms of the corresponding ribbon graph, the $g$-invariant multiplicity
and scalar functions are shown in Figure~\ref{fig:quatriang}.
The multiplicities are admissible when $a \geq 2$.
\end{itemize}

In both cases Theorem~\ref{t:quasi} holds (even when the multiplicities are
not admissible) and one deduces the extra zero-relations as in
Section~\ref{ssec:findim}.
\end{proof}

\begin{theorem} \label{t:quat2CY}
The following assertions are true:
\begin{enumerate}
\renewcommand{\theenumi}{\alph{enumi}}
\item
An algebra of possibly quaternion type is 2-CY-tilted.

\item \label{it:t:posquat}
An algebra of possibly quaternion type is actually of quaternion type.

\item \label{it:t:quat2CY}
An algebra of quaternion type is 2-CY-tilted.
\end{enumerate}
\end{theorem}
\begin{proof}[Proof (sketch)]
An algebra of possibly quaternion type has at most three simple modules.
The case of one simple module was considered in Corollary~\ref{c:triang1},
so let $\gL$ be such algebra with two or three simple modules.
A careful look at the derived equivalences constructed by Holm~\cite{Holm99}
for algebras of (possibly) quaternion type shows that there exist algebras
$\gL_0, \gL_1, \dots, \gL_n$ such that:
\begin{itemize}
\item
$\gL_0$ is one of the algebras of possibly quaternion type appearing in
Figure~\ref{fig:quat};

\item
For each $0 \leq i < n$, there exists an indecomposable projective
$\gL_i$-module $P_i$ such that $\gL_{i+1} \simeq \End U^+_{P_i}(\gL_i)$
or $\gL_{i+1} \simeq \End U^-_{P_i}(\gL_i)$ (cf.\ Section~\ref{ssec:dereq});

\item
$\gL_n \simeq \gL$.
\end{itemize}

The algebra $\gL_0$ is a triangulation algebra by Lemma~\ref{l:quatriang},
hence it is 2-CY-tilted by Theorem~\ref{t:quasi}\eqref{it:t:potential}.
By repeatedly applying Corollary~\ref{c:sym2CYmut} we see that
since $\gL_i$ is symmetric and 2-CY-tilted, so is $\gL_{i+1}$.
Therefore $\gL \simeq \gL_n$ is 2-CY-tilted.

Part~\eqref{it:t:posquat} now follows from Proposition~\ref{p:period} and
the tameness of the algebras of possibly quaternion type established by
Holm~\cite[\S6]{Holm99}. Part~\eqref{it:t:quat2CY} is a consequence of
Erdmann's classification, but see the caveat in Proposition~\ref{p:newquat}
below.
\end{proof}

The above proof also shows that all the algebras of quaternion type arise
as algebras of the form given in Theorem~\ref{t:quasi}\eqref{it:t:quasimut}.

\begin{corollary} \label{c:quat2CY}
Blocks of finite groups with generalized quaternion defect group are
2-CY-tilted.
\end{corollary}

\begin{remark}
It is actually possible to present all the algebras of (possibly) quaternion
type as Jacobian algebras of hyperpotentials and thus deduce an alternative,
direct proof of Theorem~\ref{t:quat2CY}.

It is also possible to present more families in the list of algebras of
quaternion type as triangulation algebras.
However, not all the algebras of quaternion type are triangulation algebras.
For example, the algebras in the family $\cQ(3\cC)^{k,s}$ have the quiver
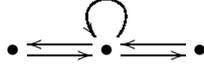
\[
\xymatrix{
{\bullet} \ar@<-0.5ex>[r] &
{\bullet} \ar@<-0.5ex>[r] \ar@<-0.5ex>[l] \ar@(ur,ul)[]^{} &
{\bullet} \ar@<-0.5ex>[l]
}
\]
which is not a triangulation quiver or obtained from one by
deleting arrows.
\end{remark}

\begin{remark}
When the ground field is of characteristic zero,
Burban, Iyama, Keller and Reiten have shown in~\cite[\S7]{BIKR08}
that certain algebras of quaternion
type occur as endomorphism algebras of cluster-tilting objects in
the 2-Calabi-Yau stable categories of maximal Cohen-Macaulay modules over
minimally elliptic curve singularities, and hence they are 2-CY-tilted.
Moreover, they described these algebras as quotients of the complete path
algebra by closed ideals. 

These algebras are organized in two families, denoted
$A_q(\lambda)$, where $q \geq 2$ and $\lambda \in K^{\times}$, 
and $B_{p,q}(\lambda)$, where $p,q \geq 1$ and $\lambda \in K^{\times}$.
The scalar $\lambda$ could be set to~$1$ except for the algebras
$A_2(\lambda)$ and $B_{1,1}(\lambda)$ corresponding to the simply elliptic
singularities, where one should assume $\lambda \neq 1$.

Comparing their definition in~\cite[\S7]{BIKR08} with
Definition~\ref{def:triang}, we see that
the algebras $A_q(\lambda)$ and $B_{p,q}(\lambda)$ are triangulation algebras
with the triangulation quivers numbered $2$ and $3a$ of
Table~\ref{tab:quivers}, respectively. The corresponding
multiplicities and scalars are shown in Figure~\ref{fig:elliptic}.
\end{remark}

\begin{figure}
\begin{align*}
\begin{array}{c}
\xymatrix@=0.75pc{
{} & {{^{\lambda^{-1}}}\circ_{q+1}} \ar@{-}[rr]^{}
& & {{^1}\circ_1} \ar@{-}@(ul,u)[lll]_{} \ar@{-}@(dl,d)[lll]_{}
}
\end{array}
& &
\begin{array}{c}
\xymatrix@=0.75pc{
{} & {{^1}\circ_{p+1}} \ar@{-}[rr]
& & {{^1}\circ_1} \ar@{-}@(ul,u)[lll]_{} \ar@{-}@(dl,d)[lll]_{}
\ar@{-}[rr]
& & {{^\lambda}\circ_{q+1}}
}
\end{array}
\\
A_q(\lambda) 
\text{ \scriptsize{($q \geq 2$, $\lambda \in K^{\times}$)}}
& &
B_{p,q}(\lambda)
\text{ \scriptsize{($p,q \geq 1$, $\lambda \in K^{\times}$)}}
\end{align*}
\caption{Description of 2-CY-tilted algebras arising from minimally elliptic
curve singularities~\cite{BIKR08} as triangulation algebras.
The subscript at each node indicates the corresponding multiplicity whereas
the superscript indicates the scalar.}
\label{fig:elliptic}
\end{figure}
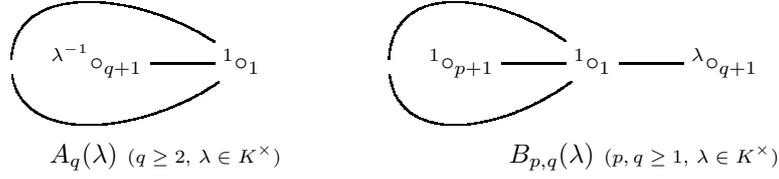

As the next proposition shows, by using triangulation quivers and power series
as in Remark~\ref{rem:gentrian}, we are able to find algebras of quaternion
type which seem not to appear in the known lists. Consider a new family of
algebras $\cQ(3\cA)_3^k$ defined for the integers $k>2$
by the quivers with relations given in Figure~\ref{fig:newquat}.
By computing their Cartan matrices, one verifies that these algebras do
not belong to any of the families in Erdmann's list.

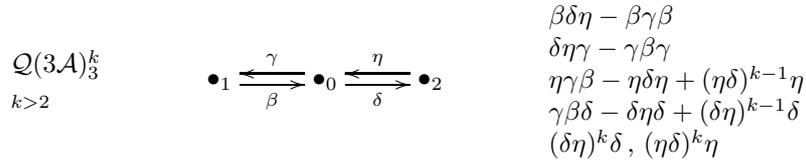
\begin{figure}
\begin{align*}
\begin{array}{l}
\cQ(3\cA)_3^k \\
_{k>2}
\end{array}
& &
\begin{array}{c}
\xymatrix{
{\bullet_1} \ar@<-0.5ex>[r]_{\beta} &
{\bullet_0} \ar@<-0.5ex>[r]_{\delta} \ar@<-0.5ex>[l]_{\gamma} &
{\bullet_2} \ar@<-0.5ex>[l]_{\eta}
}
\end{array}
& &
\begin{array}{l}
\beta \delta \eta - \beta \gamma \beta  \\
\delta \eta \gamma - \gamma \beta \gamma \\
\eta \gamma \beta - \eta \delta \eta + (\eta \delta)^{k-1} \eta \\
\gamma \beta \delta - \delta \eta \delta + (\delta \eta)^{k-1} \delta \\
(\delta \eta)^k \delta \,,\, (\eta \delta)^k \eta
\end{array}
\end{align*}
\caption{A new family of algebras of quaternion type.}
\label{fig:newquat}
\end{figure}

\begin{proposition} \label{p:newquat}
The algebras in the family $\cQ(3\cA)_3^k$ are 2-CY-tilted and of quaternion
type.
\end{proposition}
\begin{proof}
Let $\gL$ be the algebra $\cQ(3\cA)_3^k$ for some $k>2$.
By slightly modifying the proof of Lemma~5.12 in~\cite{Holm99}, one shows that
$U^-_{P_2}(\gL)$ is a tilting complex over $\gL$ whose endomorphism algebra
is isomorphic to an algebra of the form $\cQ(3\cA)_1^{k,2}$ in Erdmann's list.
The latter algebra has also the form $\cQ(3\cK)^{1,2,k}$.
Hence, by~\cite[Proposition~2.1]{Holm99} and Theorem~\ref{t:quat2CY},
the algebra $\gL$ is of quaternion type.

Corollary~\ref{c:sym2CYmut} would then imply that $\gL$ is 2-CY-tilted, but
let us give a direct proof of this fact. Indeed, the algebra $\gL$ is
a generalized version of a triangulation algebra, as considered in 
Remark~\ref{rem:gentrian}, for the triangulation quiver $3b$ 
of Table~\ref{tab:quivers} with the
$f$-invariant invertible power series $p_\alpha(x)$ all set to $1$
and the $g$-invariant power series given on the nodes of the corresponding
ribbon graph by
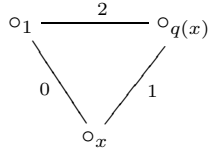
\[
\xymatrix@=1pc{
{\circ_{1}} \ar@{-}[rr]^2 && {\circ_{q(x)}} \ar@{-}[ddl]^1 \\ \\
& {\circ_{x}} \ar@{-}[uul]^0
}
\]
where $q(x)$ is any power series such that the least order term of $q(x)-x$
has degree $k-1$.
\end{proof}

All the algebras of quasi-quaternion type constructed so far are 2-CY-tilted.
In view of Theorem~\ref{t:sym2CYfin} and Theorem~\ref{t:quat2CY},
the following question, which is a reformulation of
Question~\ref{q:sym42CY} in the tame case, arises naturally.
\begin{question}
Let $\gL$ be an algebra of quasi-quaternion type. Is $\gL$ 2-CY-tilted?
\end{question}

\subsection{2-CY-tilted blocks}

By using results on the stable Auslander-Reiten quivers of
tame blocks~\cite{ES92} and wild blocks~\cite{Erdmann95},
Erdmann and Skowro\'{n}ski have characterized the blocks of group algebras
whose non-projective modules are periodic~\cite{ES15},
see also~\cite[Theorem~5.3]{ES08}.
As a consequence, by invoking Proposition~\ref{p:period},
Theorem~\ref{t:sym2CYfin} and Corollary~\ref{c:quat2CY}
we obtain the following characterization of 2-CY-tilted blocks.

\begin{proposition} \label{p:block2CY}
Let $B$ be a block of a group algebra over an algebraically
closed  field of characteristic~$p$ with defect group $D$.
Then $B$ is a 2-CY-tilted algebra if and only if either:
\begin{enumerate}
\renewcommand{\theenumi}{\alph{enumi}}
\item
$D$ is cyclic and $B$ has at most two simple modules; or

\item
$p=2$ and $D$ is a generalized quaternion group.
\end{enumerate}
\end{proposition}

\subsection{Symmetric algebras of tubular type $(2,2,2,2)$}
\label{ssec:tubular}

\begin{table}
\begin{center}
\begin{tabular}{clllc}
\textbf{No.\ of} & \textbf{Algebra} &
\textbf{Alternative} &
\textbf{Marked surface} & 
\textbf{Multiplicities} \\
\textbf{simples} & & \textbf{description} \\
\hline
\rule{0pt}{2.6ex}
2 & $A_2(\lambda)$ & $\cQ(2\cB)_1^{1,3}(\lambda,0)$ &
punctured monogon & $(1,3)$ \\
  & $\gL_3(\lambda)$ & $\cQ(2\cB)_1^{1,3}(\lambda,\lambda)$ \\
\hline 
\rule{0pt}{2.6ex}
3 & $A_1(\lambda)$ & $\cQ(3\cK)^{1,2,2}(\lambda)$ &
sphere, 3 punctures & $(1,2,2)$ \\
\hline
\rule{0pt}{2.6ex}
6 & $T(B_i(\lambda))$ & \cite[Fig.\ 1]{GKO13},
& sphere, 4 punctures & $(1,1,1,1)$ \\
& $^{1 \leq i \leq 4}$ & \cite[Fig.\ 1.6]{Jasso15} \\
\hline
\end{tabular}
\end{center}
\caption{The symmetric algebras of tubular type $(2,2,2,2)$ and their socle
deformations. 
Each family depends on a parameter $\lambda \in K \setminus \{0,1\}$.}
\label{tab:tubular}
\end{table}

In this section we show that the class of algebras considered in
Theorem~\ref{t:quasi} contains all the symmetric algebras of tubular type
$(2,2,2,2)$ and their socle deformations.
As a consequence, Question~\ref{q:sym42CY} has a positive
answer for the tame symmetric algebras of polynomial growth. For the
definitions of the terms in the next proposition we refer the reader to
the classification of tame symmetric algebras of polynomial growth by
Skowro\'{n}ski~\cite{Skowronski89} and to the surveys~\cite{ES08,Skowronski06}.
Recall that two self-injective algebras $\gL$ and $\gL'$ are \emph{socle
equivalent} if the factor algebras $\gL/\soc \gL$ and $\gL'/\soc \gL'$
are isomorphic.

\begin{proposition} \label{p:sympoly}
Let $\gL$ be a basic, indecomposable, representation-infinite tame symmetric
algebra of polynomial growth.
Then the following conditions are equivalent:
\begin{enumerate}
\renewcommand{\theenumi}{\alph{enumi}}
\item \label{it:p:sympoly4}
$\Omega_{\gL}^4 M \simeq M$ for any $M \in \stmod \gL$;

\item \label{it:p:tubular2222}
$\gL$ is socle equivalent to a symmetric algebra of tubular type $(2,2,2,2)$;

\item \label{it:p:sympoly2CY}
$\gL$ is a 2-CY-tilted algebra.
\end{enumerate}
\end{proposition}

The implication \eqref{it:p:sympoly4}$\Rightarrow$\eqref{it:p:tubular2222}
follows from known results in the literature, we refer
to~\cite[Proposition~6.2]{BES15}, \cite{BS02} or~\cite[Theorem~6.1]{ES08}.
The implication \eqref{it:p:sympoly2CY}$\Rightarrow$\eqref{it:p:sympoly4} is
a consequence of Proposition~\ref{p:period}. We prove the implication
\eqref{it:p:tubular2222}$\Rightarrow$\eqref{it:p:sympoly2CY} by using the
classification of the tame symmetric algebras of tubular type and their
socle deformations in~\cite{BS03,BS04}, keeping the notation introduced in
these papers.

Let $\gL$ be socle equivalent to a symmetric algebra of tubular type
$(2,2,2,2)$. Then $\gL$ may have 2, 3 or 6 simple modules.

In the case of 2 simple modules, the algebra $\gL$ is either $A_2(\lambda)$
of~\cite{BS03} or the non-standard $\gL_3(\lambda)$ of~\cite{BS04}, where
$\lambda \in K \setminus \{0,1\}$. We observe that $A_2(\lambda)$ is
isomorphic to the algebra $\cQ(2\cB)_1^{1,3}(\lambda,0)$ whereas
$\gL_3(\lambda)$ is isomorphic to the algebra
$\cQ(2\cB)_1^{1,3}(\lambda,\lambda)$, hence both algebras are triangulation
algebras by Lemma~\ref{l:quatriang}.

In the case of 3 simple modules, the algebra $\gL$ is $A_1(\lambda)$
of~\cite{BS03},
which is isomorphic to $\cQ(3\cK)^{1,2,2}(\lambda)$, so again
Lemma~\ref{l:quatriang} gives that $\gL$ is a triangulation algebra.

In the case of 6 simple modules, by~\cite[Proposition~5.2]{BHS03}
the algebra $\gL$ is the trivial extension algebra of a tubular algebra of
type $(2,2,2,2)$, and there are exactly four such algebras, denoted
by $T(B_i(\lambda))$ for $1 \leq i \leq 4$,
see~\cite[\S3.3]{Skowronski89} or~\cite[\S4]{BS02}.
It is instructive to compare the description of these trivial extension
algebras as quivers with relations with the lists of quivers with potentials
given in~\cite[Figure~1]{GKO13} or in~\cite[Figure~1.6]{Jasso15} describing
the endomorphism algebras of the cluster-tilting objects within the cluster
category associated to a weighted projective line with weights $(2,2,2,2)$,
and to see that these are identical.

Moreover, we observe that $T(B_4(\lambda))$ is a triangulation algebra for the
triangulation quiver whose ribbon graph is the tetrahedron with all
multiplicities set to~$1$. The marked surfaces realizing the
symmetric algebras of tubular type $(2,2,2,2)$ and their socle deformations
are summarized in Table~\ref{tab:tubular}.

\subsection{Jacobian algebras from closed surfaces}
\label{ssec:Jaclosed}

In this section we explain how Theorem~\ref{t:quasi}
implies that the Jacobian algebras of the quivers with potentials associated by
Labardini-Fragoso to triangulations of closed surfaces with punctures are of
quasi-quaternion type.

In~\cite{Labardini09}, Labardini-Fragoso constructed potentials on the
adjacency quivers of triangulations of marked surfaces and proved that flips
of triangulations result in mutations of their associated quivers with
potentials.
Denote by $Q'_\tau$ the adjacency quiver of a triangulation $\tau$ of
a marked surface $(S,M)$ as defined by Fomin, Shapiro and
Thurston~\cite[Definition~4.1]{FST08}
(we use the notation $Q'_\tau$ to distinguish it
from the underlying quiver $Q_\tau$ of the triangulation quiver associated
to $\tau$, see Section~\ref{ssec:triangadj}) and let $W_\tau$ be the
associated potential on $Q'_\tau$.
The notion of flip occurring in the next proposition is explained later in 
Section~\ref{ssec:flip}.

\begin{proposition}[\protect{\cite[Theorem~30]{Labardini09}}] \label{p:QPflip}
If a triangulation $\tau'$ of $(S,M)$ is obtained from $\tau$ by flipping an
arc $\gamma$, then the quiver with potential $(Q'_{\tau'},W_{\tau'})$ is
right equivalent to the mutation
of $(Q'_\tau,W_\tau)$ at the vertex of $Q'_\tau$ corresponding to $\gamma$.
\end{proposition}

We now assume that the surface $S$ is closed. In this case the potentials
depend on scalars attached to the punctures of $S$.
For ``nice'' triangulations of $(S,M)$, an equivalent
description of the quivers with potentials was given in~\cite{Ladkani12},
where we also showed that their Jacobian algebras are finite-dimensional and
symmetric.
In particular, the scalars can be encoded as a $g$-invariant function
$c \colon Q_1 \to K^{\times}$ and the Jacobian
algebra of the associated potential is a triangulation algebra, where all the
multiplicities are set to $1$.

\begin{proposition}[\protect{\cite[\S2]{Ladkani12}}] \label{p:QPtriang}
Let $\tau$ be a triangulation of a closed surface which is not a sphere with
less than four punctures, and assume that
at each puncture there are at least three incident arcs. Then 
$Q'_\tau=Q_\tau$, the constant multiplicity function $\mathbf{1}$ is
admissible, and the Jacobian
algebra $\cP(Q_\tau, W_\tau)$ is isomorphic to the triangulation algebra
$\gL(Q_\tau, f_\tau, \mathbf{1}, c)$.
\end{proposition}

Let $(S,M)$ be a closed surface which is not a sphere with less than four
punctures. In~\cite[\S5]{Ladkani12} we proved the existence of a triangulation
$\tau$ of $(S,M)$ satisfying the condition in Proposition~\ref{p:QPtriang}.
Therefore Theorem~\ref{t:quasi} applies for the triangulation algebra
$\cP(Q_\tau, W_\tau)$.
We note that in the case of a sphere with exactly four
punctures the ribbon graph of $\tau$ is a tetrahedron and
the corresponding assumption on the scalars attached to the
punctures has to be made.

Let $\cC$ be the triangulated 2-Calabi-Yau category of 
Theorem~\ref{t:quasi}\eqref{it:t:potential} such that
$\End_{\cC}(T) \simeq \cP(Q_\tau,W_\tau)$ for some cluster-tilting object
$T$ of $\cC$.
It is well known that any other triangulation $\tau'$ of $(S,M)$ can be
obtained from $\tau$ by a sequence of flips. Let $T'$ be the
cluster-tilting object of $\cC$ obtained from $T$ by the corresponding
sequence of Iyama-Yoshino mutations. Repeated application of
Proposition~\ref{p:QPmutIY} and Proposition~\ref{p:QPflip} shows that
$\End_{\cC}(T') \simeq \cP(Q'_{\tau'}, W_{\tau'})$,
hence part~\eqref{it:t:quasimut} of Theorem~\ref{t:quasi} applies and we get
the following result.

\begin{corollary} \label{c:Jaclosed}
Let $(S,M)$ be a closed surface which is not a sphere with less than
four punctures. Then the Jacobian algebras of the quivers with potentials
associated to the ideal triangulations of $(S,M)$ are finite-dimensional of
quasi-quaternion type and they are all derived equivalent to each other.
Moreover, each of these algebras arises as an algebra in
part~\eqref{it:t:quasimut} of Theorem~\ref{t:quasi} for a suitable
triangulation quiver.
\end{corollary}

\begin{remark}
The tameness of the algebras $\cP(Q'_\tau, W_\tau)$ has also been proved 
in~\cite{GLS16} using a different degeneration argument.
\end{remark}

\begin{remark}
Labardini-Fragoso showed also that the potentials $W_\tau$ are
non-degenerate~\cite{Labardini16}, but this fact is not needed in order to
establish Corollary~\ref{c:Jaclosed}.
\end{remark}

\section{Mutations}
\label{sec:mut}

Many of the algebras occurring in part~\eqref{it:t:quasimut} of
Theorem~\ref{t:quasi} are themselves triangulation algebras. In this
section we introduce a notion of mutation for triangulation quivers
and study its relations to other notions of mutation in the literature
including flips of triangulations, Kauer's elementary moves for Brauer graph
algebras~\cite{Kauer98}, mutations of quivers with potentials~\cite{DWZ08}
and Iyama-Yoshino mutations~\cite{IyamaYoshino08} within the triangulated
2-Calabi-Yau categories appearing in Theorem~\ref{t:quasi}.

\subsection{Mutation of triangulation quivers}
\label{ssec:mut}

A mutation of a triangulation quiver at some vertex is a new triangulation
quiver. We first give the definition in the case the vertex we mutate at
has no loops.

\begin{definition} \label{def:mut}
Let $(Q,f)$ be a triangulation quiver and let $k$ be a vertex of $Q$
without loops. Denote by $\alpha$, $\balpha$ the two arrows that start at $k$
and observe that our assumption on $k$ implies that there are six distinct arrows
\begin{align*}
\alpha_1 = \alpha &,&
\beta_1 = f(\alpha) &,&
\gamma_1 = f^2(\alpha) &,&
\alpha_2 = \balpha &,&
\beta_2 = f(\balpha) &,&
\gamma_2 = f^2(\balpha)
\end{align*}
which form two cycles of the permutation $f$.

The \emph{mutation} of $(Q,f)$ at $k$, denoted $\mu_k(Q,f)$,
is the triangulation quiver $(Q',f')$
obtained from $(Q,f)$ by performing the following steps:
\begin{enumerate}
\item
Remove the two arrows $\beta_1$ and $\beta_2$;

\item
Replace the four arrows $\alpha_1$, $\alpha_2$, $\gamma_1$ and $\gamma_2$ 
with arrows in the opposite direction $\alpha^*_1$, $\alpha^*_2$,
$\gamma^*_1$ and $\gamma^*_2$;

\item
Add new arrows $\delta_{12}$ and $\delta_{21}$ with
\begin{align*}
s(\delta_{12}) = s(\gamma_1) &,&
t(\delta_{12}) = t(\alpha_2) &,&
s(\delta_{21}) = s(\gamma_2) &,&
t(\delta_{21}) = t(\alpha_1) ,
\end{align*}
see Figure~\ref{fig:mut}(a).

\item
Define the permutation $f'$ on the new set of arrows $Q'_1$ 
by $f'(\eps)=f(\eps)$ if $\eps$ is an arrow of $Q$ which has not been
changed, and by
\begin{align*}
f'(\alpha^*_1) = \gamma^*_2 &,&
f'(\gamma^*_2) = \delta_{21} &,&
f'(\delta_{21}) = \alpha^*_1 \\
f'(\alpha^*_2) = \gamma^*_1 &,&
f'(\gamma^*_1) = \delta_{12} &,&
f'(\delta_{12}) = \alpha^*_2
\end{align*}
for the other arrows.
\end{enumerate}
\end{definition}
At the level of the underlying quivers, this is similar to Fomin-Zelevinsky
mutation~\cite{FZ02}.
Note, however, that the quivers $Q$ and $Q'$ may have $2$-cycles.
 
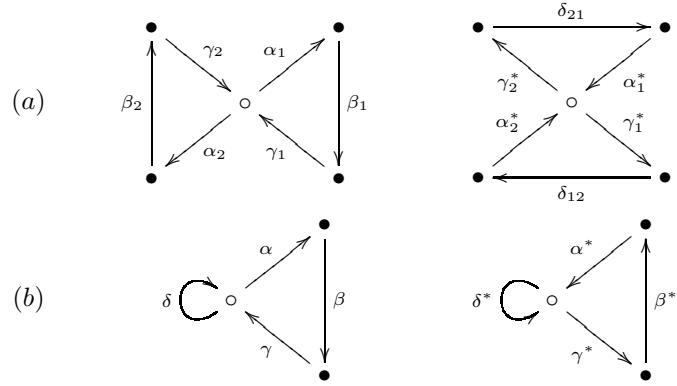
\begin{figure}
\[
\begin{array}{ccccc}
(a) & &
\begin{array}{c}
\xymatrix@R=1.5pc{
{\bullet} \ar[dr]^{\gamma_2} & & {\bullet} \ar[dd]^{\beta_1}  \\
& {\circ} \ar[ur]^{\alpha_1} \ar[dl]^{\alpha_2} \\
{\bullet} \ar[uu]^{\beta_2}  & & {\bullet} \ar[ul]^{\gamma_1}
}
\end{array}
& &
\begin{array}{c}
\xymatrix@R=1.5pc{
{\bullet} \ar[rr]^{\delta_{21}} && {\bullet} \ar[dl]^{\alpha^*_1} \\
& {\circ} \ar[dr]^{\gamma^*_1} \ar[ul]^{\gamma^*_2} \\
{\bullet} \ar[ur]^{\alpha^*_2}  && {\bullet} \ar[ll]^{\delta_{12}}
}
\end{array}
\\
(b) & &
\begin{array}{c}
\xymatrix@R=1.5pc{
&& {\bullet} \ar[dd]^{\beta} \\
& {\circ} \ar[ur]^{\alpha} \ar@(dl,ul)[]^{\delta} \\
&& {\bullet} \ar[ul]^{\gamma}
}
\end{array}
& &
\begin{array}{c}
\xymatrix@R=1.5pc{
&& {\bullet} \ar[dl]_{\alpha^*} \\
& {\circ} \ar[dr]_{\gamma^*} \ar@(ul,dl)[]_{\delta^*} \\
&& {\bullet} \ar[uu]_{\beta^*}
}
\end{array}
\end{array}
\]
\caption{Mutation of triangulation quivers at the middle vertex $\circ$;
(a) without loops; (b) with a loop fixed by the permutation $f$.
Some of the other vertices may coincide, and only the arrows that change are
shown.}
\label{fig:mut}
\end{figure}

Next, we define mutation at a vertex with loop.
\begin{definition} \label{def:mutloop}
Let $(Q,f)$ be a triangulation quiver and let $k$ be a vertex of $Q$ with
a loop. Denote by $\alpha$, $\balpha$ the two arrows that start at $k$ and
assume that $\balpha$ is a loop.
The \emph{mutation} of $(Q,f)$ at $k$, denoted $\mu_k(Q,f)$,
is the triangulation quiver $(Q',f')$
obtained from $(Q,f)$ by performing the following steps:
\begin{enumerate}
\setcounter{enumi}{-1}
\item
If $g(\balpha)=\balpha$, or if $\alpha$ is also a loop, then set
$(Q',f')=(Q,f)$.

Otherwise, there are four distinct arrows
\begin{align*}
\alpha &,& \beta=f(\alpha) &,& \gamma=f^2(\alpha) &,& \delta=\balpha=f(\balpha)
\end{align*}
which form two cycles of the permutation $f$.

\item
Replace the four arrows $\alpha$, $\beta$, $\gamma$ and $\delta$
by arrows in the opposite direction $\alpha^*$, $\beta^*$, $\gamma^*$
and $\delta^*$, see Figure~\ref{fig:mut}(b);

\item
Define the permutation $f'$ on the new set of arrows $Q'_1$ 
by $f'(\eps)=f(\eps)$ if $\eps$ is an arrow of $Q$ which has not been
changed, and by
\begin{align*}
f'(\alpha^*) = \gamma^* &,&
f'(\beta^*) = \alpha^* &,&
f'(\gamma^*) = \beta^* &,&
f'(\delta^*) = \delta^*
\end{align*}
for the other arrows.
\end{enumerate}
Note that the arrow $\delta^*$ is also a loop at $k$ so we could have
avoided the reversal of $\delta$. This reversal is done in order to stress the
analogy to the general case of Definition~\ref{def:mut}.
\end{definition}

\begin{example}
We describe all the mutations of the triangulation quivers appearing in
Table~\ref{tab:quivers}.
For each of the triangulation quivers $1$, $2$, $3'$ and $3''$, a mutation at
any vertex gives an isomorphic triangulation quiver.
For the triangulation quiver $3b$, a mutation at any vertex is isomorphic to
the triangulation quiver $3a$.
For the triangulation quiver $3a$, a mutation at the vertex $2$ is
isomorphic to the triangulation quiver $3b$, whereas a mutation at any of the
other vertices gives the triangulation quiver $3a$.
\end{example}

\begin{remark}
As can be seen from Figure~\ref{fig:mut}, mutation is an involution.
In other words, if $(Q,f)$ is a triangulation quiver and $k$ is a vertex of
$Q$, then the triangulation quiver $\mu_k(\mu_k(Q,f))$ is isomorphic to
$(Q,f)$.
\end{remark}

The permutation $f'$ on $Q'_1$ defines the permutation $g'$ by
$g'(\alpha')=\overline{f'(\alpha')}$ for any $\alpha' \in Q'_1$.
The next statement is a consequence of the definitions.

\begin{lemma} \label{l:mutcycle}
Let $(Q',f')$ be a mutation of the triangulation quiver $(Q,f)$
at some vertex. Then:
\begin{enumerate}
\renewcommand{\theenumi}{\alph{enumi}}
\item \label{it:l:fcyc}
The permutations $f$ and $f'$ have the same cycle structure.

\item \label{it:l:gnum}
The permutations $g$ and $g'$ have the same number of cycles.
\end{enumerate}
\end{lemma}

\begin{example} \label{ex:mutg}
Although $g$ and $g'$ have the same number of cycles, the lengths of the
cycles may change. For example, for the triangulation quiver 3b of
Table~\ref{tab:quivers} the lengths are $2$, $2$, $2$ whereas for
the mutated triangulation quiver 3a they are $1$, $4$, $1$.
\end{example}

Lemma~\ref{l:mutcycle}\eqref{it:l:gnum} implies that
any $g$-invariant function $\nu$ gives rise to a $g'$-invariant function
$\nu'$ on $Q'_1$ with the same image. Explicitly, this is done by setting
$\nu'_\eps = \nu_\eps$ for the arrows in $Q'_1$ that are also in $Q_1$ and
\begin{align*}
\nu'_{\alpha^*_1} = \nu'_{\gamma^*_1} = \nu_{\beta_1} &,&
\nu'_{\alpha^*_2} = \nu'_{\gamma^*_2} = \nu_{\beta_2} &,&
\nu'_{\delta_{12}} = \nu_{\gamma_1} (= \nu_{\alpha_2}) &,&
\nu'_{\delta_{21}} = \nu_{\gamma_2} (= \nu_{\alpha_1})
\end{align*}
for the other arrows in the case of Definition~\ref{def:mut} and
\begin{align*}
\nu'_{\alpha^*} = \nu'_{\gamma^*} = \nu'_{\delta^*} = \nu_\beta &,&
\nu'_{\beta^*} = \nu_{\gamma} (=\nu_{\delta} = \nu_{\alpha})
\end{align*}
in the case of Definition~\ref{def:mutloop}.
In particular, any two $g$-invariant functions $m \colon Q_1 \to \bZ_{>0}$ and
$c \colon Q_1 \to K^{\times}$ of multiplicities and scalars on $(Q,f)$
give rise to $g'$-invariant functions of multiplicities
$m' \colon Q'_1 \to \bZ_{>0}$ and scalars $c' \colon Q'_1 \to K^{\times}$
on $(Q',f')$.

\begin{remark}
Since the lengths of the cycles of $g$ may change under mutation, even
if a multiplicity function $m \colon Q_1 \to \bZ_{>0}$ 
on $(Q,f)$ was admissible, the multiplicity function $m'$ on
$(Q',f')$ may not be admissible anymore.
\end{remark}

\begin{example}
Continuing Example~\ref{ex:mutg}, if $m$ is the multiplicity function for
the triangulation quiver $3b$ taking the constant value $2$,
then $m'$ takes the constant value $2$ on the arrows of the triangulation
quiver $3a$. Hence $m$ is admissible while $m'$ is not.
\end{example}

Similarly, Lemma~\ref{l:mutcycle}\eqref{it:l:fcyc} implies that any function
$\theta$ on the set $Q_1^f$ of fixed points of $f$ gives rise to a function
$\theta'$ on the set $(Q'_1)^{f'}$ of fixed points of $f'$.
Explicitly, in the case of
Definition~\ref{def:mut} we have $\theta'=\theta$, whereas in the case of
Definition~\ref{def:mutloop} we have
$\theta'_{\delta^*} = \theta_\delta$
and $\theta'_\eps=\theta_\eps$ for any
arrow $\eps \neq \delta$ with $f(\eps)=\eps$.

\subsection{Mutations and flips}
\label{ssec:flip}

Fomin, Shapiro and Thurston have shown in~\cite[Proposition~4.8]{FST08} that
if two triangulations
are related by flipping an arc, then their adjacency quivers are related by a
Fomin-Zelevinksy mutation at the vertex corresponding to that arc. In this
section we discuss an analogous statement for triangulation quivers.

Let $\tau$ be a triangulation of a marked surface $(S,M)$.
If $\gamma$ is an arc of $\tau$ which is not the inner side of a self-folded
triangle, then it is possible to replace $\gamma$ by another arc $\gamma'$
to obtain a triangulation $\tau' = \tau \setminus \{\gamma\} \cup \{\gamma'\}$
which is not topologically equivalent to $\tau$, see Figure~\ref{fig:flip}.
The triangulation $\tau'$ is called the \emph{flip} of $\tau$ at $\gamma$.

\begin{figure}
\begin{align*}
\xymatrix{
& {\circ} \ar@{-}[dr] \ar@{-}[dd]_{\gamma} \\
{\circ} \ar@{-}[ur] \ar@{-}[dr] && {\circ} \\
& {\circ} \ar@{-}[ul] \ar@{-}[ur]
}
&&
\xymatrix{
& {\circ} \ar@{-}[dr] \\
{\circ} \ar@{-}[ur] \ar@{-}[rr]^{\gamma'} \ar@{-}[dr] && {\circ} \\
& {\circ} \ar@{-}[ul] \ar@{-}[ur]
}
\end{align*}
\caption{Flip of a triangulation at the arc $\gamma$.
The sides of the quadrilateral may be arcs or boundary segments.}
\label{fig:flip}
\end{figure}
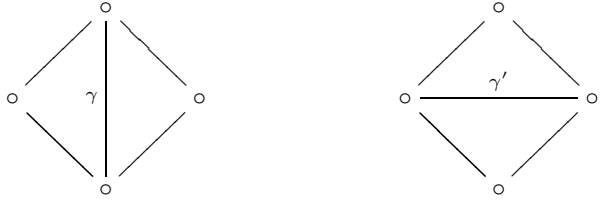

\begin{lemma}
The triangulation quivers of two triangulations related by a flip at some arc
are related by a mutation at the vertex corresponding to that arc.
\end{lemma}
\begin{proof}
First we verify that a vertex corresponding to a flippable arc cannot
have loops. Indeed, for a loop $\alpha$ at some vertex $k$ we have that either
$f(\alpha)=\alpha$ or $g(\alpha)=\alpha$. In the former case $k$ corresponds
to a boundary segment, whereas in the latter case it corresponds to an arc
which is the inner side of a self-folded triangle.

Now the claim follows by comparing
Figure~\ref{fig:flip} and Figure~\ref{fig:mut}(a)
using the construction of triangulation quiver visualized in 
Figure~\ref{fig:triquiver}.
\end{proof}

Consider now a mutation of a triangulation quiver $(Q,f)$
at a vertex with a loop fixed by the permutation $f$. The change of the
associated ribbon graphs is illustrated in Figure~\ref{fig:ribmut}.
In particular, if $\tau$ is a triangulation of a marked surface $(S,M)$
and $k$ is a vertex corresponding to a boundary segment of $(S,M)$, then
a mutation of $(Q_\tau, f_\tau)$ at $k$ is 
a triangulation quiver $(Q_{\tau'}, f_{\tau'})$
of a triangulation $\tau'$ of a new marked
surface $(S',M')$ which is obtained from $(S,M)$ 
as follows: remove the boundary segment corresponding to $k$
from the boundary component containing it represented by the left node of
the ribbon graph in Figure~\ref{fig:ribmut}, and
add it to the component (or puncture) represented by the right node.
The arcs of $\tau'$ are identical to those of $\tau$.

\begin{figure}
\begin{align*}
\xymatrix@=1pc{
{\circ} \ar@{-}@(dr,ur) \ar@{-}@/_2pc/[rrr] \ar@{-}@/^2pc/[rrr]
&&& {\circ}
}
&&
\xymatrix@=1pc{
{\circ} \ar@{-}@/_2pc/[rrr] \ar@{-}@/^2pc/[rrr]
&&& {\circ} \ar@{-}@(dl,ul)
}
\end{align*}
\caption{The mutation of Figure~\ref{fig:mut}(b) in terms of ribbon graphs.}
\label{fig:ribmut}
\end{figure}
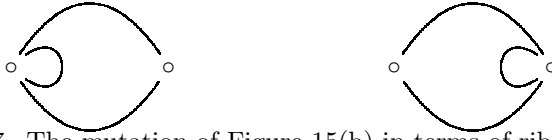

Adding or removing a boundary segment is equivalent to adding or removing
one marked point.
Here, it makes sense to consider punctures as boundary components with zero
marked points. So, when we remove a boundary segment from a component with just
one marked point we get a puncture, and conversely, when we add
a boundary segment to a puncture we get a boundary component with one marked
point.

This point of view can be made more systematic by using the notion of
orbifolds and their triangulations as introduced by Felikson, Shapiro and 
Tumarkin~\cite[\S4]{FST12}. The precise details are outside the scope of this
survey, but let us just mention that any marked surface $(S,M)$ gives rise to
a closed orbifold $\mathcal{O}$ by replacing each boundary component of
$(S,M)$ containing $n$ marked points by a puncture and $n$
orbifold points, each connected to that puncture by a so-called pending arc.
Any triangulation of $(S,M)$ yields a triangulation of the orbifold
$\mathcal{O}$.

The transitivity of flips on triangulations of orbifolds implies the next
proposition, which provides a partial converse to Lemma~\ref{l:mutcycle}.

\begin{proposition} \label{p:mutflip}
Let $\tau$ be a triangulation of a marked surface $(S,M)$ with
$p$ punctures and $b$ boundary components, and let $\tau'$ be a triangulation
of a marked surface $(S',M')$ with $p'$ punctures and $b'$ boundary
components.
Then the following conditions are equivalent:
\begin{enumerate}
\renewcommand{\theenumi}{\alph{enumi}}
\item
The triangulation quiver $(Q_{\tau'}, f_{\tau'})$ can be obtained from
$(Q_\tau, f_\tau)$ by a finite sequence of mutations;

\item
The topological parameters of the marked surfaces $(S,M)$ and $(S',M')$
satisfy
\begin{align} \label{e:topomut}
\genus(S)=\genus(S') &,& p+b=p'+b' &,&
|M|-p=|M'|-p';
\end{align}

\item
The permutations $f_\tau$ and $f_{\tau'}$ have the same cycle structure
and the permutations $g_\tau$ and $g_{\tau'}$ have the same number of cycles.
\end{enumerate}
\end{proposition}

\begin{remark}
Two closed surfaces $(S,M)$ and $(S',M')$ satisfy \eqref{e:topomut}
if and only if they are homeomorphic (i.e.\ they have the same genus and
the same number of punctures).
\end{remark}

\subsection{Mutations and Kauer moves}
Rickard~\cite[Theorem~4.2]{Rickard89} proved that a Brauer tree algebra is
derived
equivalent to a Brauer star algebra by constructing a tilting complex over the
former whose endomorphism algebra is isomorphic to the latter.
Later, K\"{o}nig and Zimmermann~\cite{KZ97} have shown that a Brauer tree can be
transformed to a Brauer star by applying a sequence of small changes, replacing
one edge at a time. In each such replacement, the Brauer tree algebras of the
two trees are related by a tilting complex of length 2 which is of the form
given in~\eqref{e:silt}, so in particular they are derived equivalent.

In~\cite{Kauer98}, Kauer considered more generally Brauer graph algebras
and defined similar moves, which he called \emph{elementary moves}.
For each edge $e$ of a Brauer graph he defined a new graph 
obtained by replacing $e$ (i.e.\ taking it out and putting it back in a 
different place) such that
if $\Gamma$ is the Brauer graph algebra corresponding to the original graph
and $P$ is the indecomposable projective $\Gamma$-module corresponding to
the edge $e$, then $\End_{\cD^b(\Gamma)} U^+_P(\Gamma)$ is the Brauer graph
algebra corresponding to the new graph.

There are three kinds of elementary moves; the first involves edges that are
leaves in the graph (i.e.\ they are incident to nodes without any additional
incident edges); the second involves edges that are loops whose two half-edges
are successive in the cyclic ordering around their common node; and the third
involves the other edges. In terms of the ribbon quiver, the first case
corresponds to vertices with a loop $\alpha$ such that $g(\alpha)=\alpha$; the
second to vertices with a loop $\alpha$ such that $f(\alpha)=\alpha$; and
the third to vertices without loop.

\begin{proposition} \label{p:mutKauer}
Let $(Q,f)$ be a triangulation quiver, let $k$ be a vertex of $Q$ and
let $(Q',f')$ be the mutation of $(Q,f)$ at $k$. Then:
\begin{enumerate}
\renewcommand{\theenumi}{\alph{enumi}}
\item
The ribbon graphs of $(Q,f)$ and $(Q',f')$ are related by an elementary move
at the edge corresponding to the vertex $k$.

\item
Let $m \colon Q_1 \to \bZ_{>0}$ and $c \colon Q_1 \to K^{\times}$ be
$g$-invariant functions of multiplicities and scalars, respectively, and let 
$m' \colon Q'_1 \to \bZ_{>0}$ and $c' \colon Q'_1 \to K^{\times}$ be the
$g'$-invariant functions induced from $m$ and $c$.
Then the Brauer graph algebras
$\Gamma = \Gamma(Q,f,m,c)$ and $\Gamma'= \Gamma(Q',f',m',c')$
satisfy
\begin{align} \label{e:BGAmut}
\End_{\cD^b(\Gamma)} U^-_{P_k}(\Gamma) \simeq \Gamma' \simeq
\End_{\cD^b(\Gamma)} U^+_{P_k}(\Gamma)
\end{align}
and in particular they are derived equivalent.
\end{enumerate}
\end{proposition}

If $(Q_\tau, f_\tau)$ is a triangulation quiver arising from a triangulation
$\tau$ of a marked surface $(S,M)$, then by Remark~\ref{rem:triblocks} we can
think
of the multiplicities and scalars as quantities attached to each puncture
and boundary component of $(S,M)$. By combining Proposition~\ref{p:mutflip}
and Proposition~\ref{p:mutKauer} we deduce the next corollary which implies
in particular that the derived equivalence class of a Brauer graph algebra
from a triangulation quiver may depend only on the surface and not on the
particular triangulation.

\begin{corollary} \label{c:BGAsurface}
Let $(S,M)$ and $(S',M')$ be two marked surfaces whose topological parameters
satisfy Eq.~\eqref{e:topomut}.
Let $\tau$ be any triangulation of $(S,M)$ and let $\tau'$ be
any triangulation of $(S',M')$. Then:
\begin{enumerate}
\renewcommand{\theenumi}{\alph{enumi}}
\item
The triangulation quiver $(Q_{\tau'}, f_{\tau'})$ can be
obtained from $(Q_\tau, f_\tau)$ by a sequence of mutations, hence
any $g$-invariant function $\nu$ on $(Q_\tau)_1$ yields a
$g$-invariant function $\nu'$ on $(Q_{\tau'})_1$.

\item
The Brauer graph algebras 
$\Gamma(Q_\tau, f_\tau, m, c)$ and $\Gamma(Q_{\tau'}, f_{\tau'}, m', c')$
are derived equivalent for any $g$-invariant function of multiplicities
$m \colon (Q_\tau)_1 \to \bZ_{>0}$ and scalars
$c \colon (Q_\tau)_1 \to K^{\times}$.
\end{enumerate}
\end{corollary}

\begin{remark}
It has also been observed by Marsh and Schroll~\cite{MarshSchroll14} that
by viewing triangulations of marked surfaces as ribbon graphs, flips of
triangulations become elementary moves of Brauer graphs and hence a marked
surface gives rise to a collection of derived equivalent Brauer graph algebras.
Note that in the
case of surfaces with non-empty boundary, the Brauer graph algebras they
consider are somewhat different than the algebras considered here.
\end{remark}

\begin{remark}
Recently, a description of Kauer's elementary moves in terms of the ribbon
quivers has been given in~\cite{Aihara15}.
\end{remark}

\subsection{Mutations and quivers with potentials}
Let $(Q,f)$ be a triangulation quiver and let $(Q',f')$ be a
mutation of $(Q,f)$ at a fixed vertex $k$ of $Q$.

Let $R \colon Q_1 \to K[[x]]$ be a $g$-invariant function and let
$P \colon Q_1^f \to K[[x]]$ be a function whose values are power series
(i.e.\ $P_\alpha(x)$ is a power series for each $\alpha \in Q_1$ such that 
$f(\alpha)=\alpha$).
Consider the potential on $Q$ defined by
\begin{equation} \label{e:potR}
W = \sum_{\alpha \,:\, f(\alpha) = \alpha} P_\alpha(\alpha) +
\sum_{\alpha \,:\, f(\alpha) \neq \alpha}
\alpha \cdot f(\alpha) \cdot f^2(\alpha) -
\sum_\beta R_{\beta}(\omega_\beta) ,
\end{equation}
where the first sum runs over the fixed points of $f$, the second
runs over representatives $\alpha$ of the $f$-cycles of length $3$
the third runs over representatives $\beta$ of the $g$-cycles in $Q_1$.
This is a special case of a potential considered in
Proposition~\ref{p:hyperib}\eqref{it:potrib}, as $P$ can be extended to
an $f$-invariant function on all the arrows by setting $P_\alpha(x)=x$ for
any arrow $\alpha$ with $f(\alpha) \neq \alpha$.

By the discussion in Section~\ref{ssec:mut}, the function $R$ gives rise to a
$g'$-invariant function $R'$ and the function $P$ gives rise to a function $P'$
on the set $(Q'_1)^{f'}$ of fixed points of $f'$, hence to the potential on
$Q'$ given by
\begin{equation} \label{e:potR2}
W' = \sum_{\alpha' \,:\, f'(\alpha')=\alpha'} P'_{\alpha'}(\alpha') + 
\sum_{\alpha' \,:\, f'(\alpha') \neq \alpha'}
\alpha' \cdot f'(\alpha') \cdot f'^2(\alpha') -
\sum_{\beta'} R'_{\beta'}(\omega_{\beta'}) ,
\end{equation}
where the sums run over  fixed points $\alpha'$ of $f'$, representatives
$\alpha'$ of the $f'$-cycles of length $3$ and representatives $\beta'$ of
the $g'$-cycles in $Q'_1$.

The next proposition compares $(Q',W')$ with the mutation of the quiver
with potential $(Q,W)$ at the vertex $k$ as defined in~\cite[\S5]{DWZ08}.

\begin{proposition} \label{p:QPmut}
Assume that there are no $2$-cycles in $Q$ passing through the vertex~$k$.
Then $(Q',W')$ is right equivalent to the mutation of $(Q,W)$ at $k$.
\end{proposition}

The assumption in the proposition implies that $k$ has no loops and
therefore the mutation is governed by Definition~\ref{def:mut}.
In the notations of that definition, the condition in the proposition
is equivalent to the conditions that
$n_{\alpha_1} > 2$, $n_{\gamma_1} > 2$, $n_{\beta_1} > 1$ and
$n_{\beta_2} > 1$.

\subsection{Mutations and triangulation algebras}
\sloppy 
Let $(Q,f)$ be a triangulation quiver and let $k$ be a vertex of $Q$.
Let $m \colon Q_1 \to \bZ_{>0}$
and $c \colon Q_1 \to K^{\times}$ be $g$-invariant functions of multiplicities
and scalars, respectively and let $\lambda \colon Q_1^f \to K$. Assume that:
\begin{itemize}
\item
$m$ is admissible;

\item
if $((Q,f),m)$ is exceptional, the scalars $c \colon Q_1 \to K^{\times}$
satisfy the conditions stated before Theorem~\ref{t:quasi};

\item
$\ch K$ does not divide $\mu_f \mu_g$ (see Proposition~\ref{p:potential} for
the definition);

\item
there are no $2$-cycles in $Q$ passing through the vertex $k$.

\end{itemize}
\fussy 

Consider the triangulation algebra $\gL=\gL(Q,f,m,c,\lambda)$. Our first
two assumptions imply that Theorem~\ref{t:quasi} holds for $\gL$ and that in
particular, $\gL$ is symmetric and there is a triangulated 2-Calabi-Yau category
$\cC$ with a cluster-tilting object $T$ such that $\gL \simeq \End_{\cC}(T)$.

Let $(Q',f')$ be the mutation of $(Q,f)$ at $k$, let
$m' \colon Q'_1 \to \bZ_{>0}$ and  $c' \colon Q'_1 \to K^{\times}$ be the
$g'$-invariant functions induced from $m$ and $c$, and
let $\lambda' \colon (Q'_1)^{f'} \to K$ be the function on the arrows fixed by
$f'$ induced from $\lambda$. Our last assumption implies that the triangulation
algebra $\gL'=\gL(Q',f',m',c',\lambda')$ is well defined
(i.e.\ $m'_{\alpha'} n_{\alpha'} \geq 2$ for any $\alpha' \in Q'_1$),
but $m'$ is not necessarily admissible.

\begin{proposition} \label{p:mutriang}
Under the above assumptions, the following assertions hold true:
\begin{enumerate}
\renewcommand{\theenumi}{\alph{enumi}}
\item \label{it:p:potR}
$\gL \simeq \cP(Q,W)$, where the potential $W$ takes the form
in~\eqref{e:potR} for suitable $g$-invariant function
$R \colon Q_1 \to K[[x]]$ and function $P \colon Q_1^f \to K[[x]]$.


\item \label{it:p:potR2}
$\gL' \simeq \cP(Q',W')$ for the potential $W'$ given in~\eqref{e:potR2}
with the functions $R'$ and $P'$ corresponding to the functions $R$ and $P$
of part~\eqref{it:p:potR}.

\item \label{it:p:QP}
The quiver with potential $(Q',W')$ is right equivalent to the mutation
of the quiver with potential $(Q,W)$ at the vertex $k$.

\item \label{it:p:IY}
$\gL' \simeq \End_{\cC}(T')$, where $T'$ is the Iyama-Yoshino mutation of $T$
with respect to the indecomposable summand corresponding to the vertex $k$.

\item \label{it:p:mutriang}
$\gL'$ is derived equivalent to $\gL$ and is of quasi-quaternion type.
More precisely, we have isomorphisms
\begin{equation} \label{e:mutriang}
\End_{\dgL} U^-_{P_k}(\gL) \simeq \gL' \simeq
\End_{\dgL} U^+_{P_k}(\gL)
\end{equation}
where $P_k$ is the indecomposable projective $\gL$-module corresponding to
the vertex $k$.
\end{enumerate}
\end{proposition}

Claim~\eqref{it:p:potR} follows by our assumption on $\mu_f \mu_g$. Since
$\mu_{f'}=\mu_f$ and $\mu_{g'}=\mu_g$, claim~\eqref{it:p:potR2} follows in a
similar way. Claim~\eqref{it:p:QP} is a consequence of the previous claims
together with our assumption on the vertex $k$ and Proposition~\ref{p:QPmut}.
Claim~\eqref{it:p:IY} follows from~\eqref{it:p:QP}
and Proposition~\ref{p:QPmutIY}. Finally, claim~\eqref{it:p:mutriang} is
a consequence of~\eqref{it:p:IY} and Proposition~\ref{p:dereq}.

\begin{remark}
Our assumptions on the characteristic of $K$ and the vertex $k$ are needed
in order to use the theory of mutations of quivers with potentials.
It seems very likely that the statements in parts~\eqref{it:p:IY}
and~\eqref{it:p:mutriang}
of Proposition~\ref{p:mutriang} are still true even if we drop the assumption
on the characteristic of $K$ and weaken the assumption on the vertex $k$,
requiring only that the triangulation algebra $\gL'$ is defined.
%
%
\end{remark}

\subsection{Construction of infinitely many non-degenerate potentials}

We conclude by presenting an application of the preceding results to
the theory of quivers with potentials.

For a mutation $(Q',W')$ of a quiver with potential $(Q,W)$, the underlying
quiver $Q'$ may have $2$-cycles even if the quiver $Q$ did not have such.
Thus, $(Q',W')$ could not be further mutated at the vertices lying on these
2-cycles.

A quiver with potential $(Q,W)$ is \emph{non-degenerate} if, for any sequence
of mutations of quivers with potentials, the underlying quiver does not contain
any $2$-cycles~\cite[Definition~7.2]{DWZ08}.
The existence of non-degenerate potentials is crucial to
several approaches to solve various conjectures on cluster algebras,
either via the representations of Jacobian algebras and their mutations
as in~\cite{DWZ10}, or via the generalized cluster categories~\cite{Plamondon11}.

Derksen, Weyman and Zelevinsky proved~\cite[Corollary~7.4]{DWZ08} that if
the ground field is
uncountable, then over any quiver without loops and $2$-cycles there is at
least one non-degenerate potential. It is interesting to know when such
non-degenerate potential is unique (up to right equivalence). For instance,
on quivers without oriented cycles there is only one potential, namely the
zero potential. In~\cite[Theorem~1.4]{GLS16},
Geiss, Labardini-Fragoso and Schr\"{o}er
proved that apart from one exception, the adjacency quiver of
a triangulation of a marked surface with non-empty boundary has only one
non-degenerate potential.
More generally, we proved in~\cite[\S4]{Ladkani13} that a non-degenerate
potential
is unique on any quiver belonging to the class $\cP$ of Kontsevich and
Soibelman~\cite[\S8.4]{KS08}, and that this class of quivers actually contains
the previous two instances.

In this section we consider the other extremity, namely, we apply the previous
results to construct quivers which have infinitely many non-degenerate
potentials whose Jacobian algebras are pairwise non-isomorphic.

Consider a triangulation quiver $(Q,f)$ such that:
\begin{equation} \label{e:g1f3}
\tag{$\star$}
\text{The permutation $g$ has one cycle and
all the cycles of $f$ are of length $3$.}
\end{equation}
These assumptions imply that the quiver $Q$ does not have loops or 2-cycles
(see Proposition~\ref{p:tri2cyc}).
Moreover, a potential as in Eq.~\eqref{e:potR} is controlled by one power series
$R(x) \in K[[x]]$, and all the cycles $\oa$ (where $\alpha$ runs over the arrows
of~$Q$) are rotationally equivalent. Denote by $\omega$ one of these cycles.

If $(Q',f')$ is a mutation of $(Q,f)$, then by Lemma~\ref{l:mutcycle} it
also satisfies~\eqref{e:g1f3} and hence Proposition~\ref{p:QPmut} can be
applied indefinitely to yield the following.

\begin{proposition}
Let $(Q,f)$ be a triangulation quiver satisfying condition~\eqref{e:g1f3}.
Then for any power series $R(x) \in xK[[x]]$, the potential $W_R$ on $Q$ given
by
\begin{equation} \label{e:potRg1f3}
W_R = -R(\omega) + \sum_\alpha \alpha \cdot f(\alpha) \cdot f^2(\alpha)
\end{equation}
(where the sum runs over representatives $\alpha$ of the $f$-cycles)
is non-degenerate.
\end{proposition}

Consider now a triangulation $\tau$ of a closed surface with exactly one
puncture. Then its triangulation quiver $(Q_\tau, f_\tau)$ satisfies
condition~\eqref{e:g1f3} by Remark~\ref{rem:triblocks}.
Moreover, the adjacency quiver of $\tau$ is $Q_\tau$ by Corollary~\ref{c:onep}.

\begin{corollary} \label{c:infpot}
Let $Q$ be the adjacency quiver of a triangulation of a closed surface with
exactly one puncture, and view it as triangulation quiver $(Q,f)$. Then:
\begin{enumerate}
\renewcommand{\theenumi}{\alph{enumi}}
\item
For any power series $R(x) \in xK[[x]]$, the potential $W_R$ on $Q$ defined by
Eq.~\eqref{e:potRg1f3} is non-degenerate.

\item
Let $R_0(x)=0$ and $R_m(x)=x^m$ for $m \geq 1$. Then
\[ 
\{W_{R_0} \} \cup \{W_{R_m} : \text{$m \geq 1$ is not divisible by $\ch K$}\}
\]
is an infinite set of non-degenerate potentials on $Q$ whose Jacobian algebras
are pairwise non-isomorphic.
\end{enumerate}
\end{corollary}

\begin{remark}
For a quiver as in Corollary~\ref{c:infpot}, it was known that there are
at least two inequivalent non-degenerate potentials, denoted in our notation
by $W_0$ and $W_x$, see~\cite[\S4.3]{Ladkani11b}, \cite[\S3]{Ladkani13}
and~\cite[Proposition~9.13]{GLS16}. Note that
the Jacobian algebra of $W_0$ is infinite dimensional whereas that
of any $W_{x^m}$ with $m \geq 1$ not divisible by $\ch K$ is a triangulation
algebra and hence finite-dimensional of quasi-quaternion type.
The latter Jacobian algebras are pairwise non-isomorphic because their
dimensions are all different; indeed, if the surface has genus $g \geq 1$ then
$Q$ has $6g-3$ vertices (Remark~\ref{rem:nvertex}) and hence the Jacobian
algebra of $W_{x^m}$ has dimension $m(12g-6)^2=36m(2g-1)^2$
(Proposition~\ref{p:Cartan}).
\end{remark}

\bibliographystyle{amsplain}
\bibliography{grpcluster}

\end{document}